\documentclass[10pt,a4paper,reqno]{amsart}
\usepackage{amssymb}
\usepackage{amscd}
\usepackage[pdftex,pdfpagelabels]{hyperref}
\usepackage{enumerate}
\usepackage{comment}
\usepackage{placeins}
\usepackage{enumitem}
\usepackage{graphicx}
\usepackage{cleveref}
\usepackage{siunitx}
\usepackage{tikz-cd}
\usepackage{nicefrac}
\usepackage{stix}
\usepackage{bm}
\DeclareMathAlphabet\mathbfcal{LS2}{stixcal}{b}{n}
\numberwithin{equation}{section}

\usepackage{mathtools}
\usepackage[tableposition=top]{caption}
\usepackage{booktabs,dcolumn}

\theoremstyle{plain}

\newtheorem{theorem}{Theorem}[section]
\newtheorem{proposition}[theorem]{Proposition}
\newtheorem{lemma}[theorem]{Lemma}
\newtheorem{corollary}[theorem]{Corollary}

\theoremstyle{definition}

\newtheorem{remark}[theorem]{Remark}

\newtheorem{example}[theorem]{Example}

\newcommand\R{\mathbb{R}}

\newcommand\N{\mathbb{N}}

\newcommand\eps{\varepsilon}
\newcommand\tuple{{\mathcal B}}
\newcommand\excess{{\mathcal{E}}}

\renewcommand{\mod}{\bmod}

\setlist{itemsep=2pt}

\begin{document}

\title{Decomposing a factorial into large factors}

\author[Alexeev]{Boris Alexeev}
\address{Unaffiliated, Athens, GA 30605.}
\email{boris.alexeev@gmail.com}

\author[Conway]{Evan Conway}
\address{UVA Department of Mathematics, Charlottesville, VA 22903.}
\email{auj4kq@virginia.edu}

\author[Rosenfeld]{Matthieu Rosenfeld}
\address{LIRMM, Univ Montpellier, CNRS, Montpellier, France.}
\email{matthieu.rosenfeld@umontpellier.fr}

\author[Sutherland]{Andrew V. Sutherland}
\address{MIT Department of Mathematics, Cambridge, MA 02139.}
\email{drew@math.mit.edu}

\author[Tao]{\\Terence Tao}
\address{UCLA Department of Mathematics, Los Angeles, CA 90095-1555.}
\email{tao@math.ucla.edu}

\author[Uhr]{Markus Uhr}
\address{Unaffiliated, Zurich, Switzerland.}
\email{uhrmar@gmail.com}

\author[Ventullo]{Kevin Ventullo}
\address{Google, Mountain View, CA.}
\email{kevinventullo@google.com}


\begin{abstract}  Let $t(N)$ denote the largest number such that $N!$ can be expressed as the product of $N$ integers greater than or equal to $t(N)$.
The bound $t(N)/N = 1/e-o(1)$ was apparently established in unpublished work of Erd\H{o}s, Selfridge, and Straus; but the proof is lost.  Here we obtain the more precise asymptotic
$$ \frac{t(N)}{N} = \frac{1}{e} - \frac{c_0}{\log N} + O\left( \frac{1}{\log^{1+c} N} \right)$$
for an explicit constant $c_0 = 0.30441901\dots$ and some absolute constant $c>0$, answering a question of Erd\H{o}s and Graham.  For the upper bound, a further lower order term in the asymptotic expansion is also obtained.  With computer assistance, we obtain highly precise computations of $t(N)$ for wide ranges of $N$, establishing several explicit conjectures of Guy and Selfridge on this sequence.  For instance, we show that $t(N) \geq N/3$ for $N \geq 43632$, with the threshold shown to be best possible.
\end{abstract}

\maketitle


\section{Introduction}

Given a natural number $M$, define a \emph{factorization} of $M$ to be a finite multiset $\tuple$ of natural numbers with product
$$ \prod \tuple \coloneqq \prod_{a \in \tuple} a = M,$$
where each $a$ appears according to its multiplicity in $\tuple$. More generally, define a \emph{subfactorization} of $M$ to be a finite multiset $\tuple$ such that $\prod \tuple$ divides $M$.  Given a threshold $t$, we say that a multiset $\tuple$ is \emph{$t$-admissible} if $a \geq t$ for all $a \in \tuple$.  For a given natural number $N$, we then define $t(N)$ to be the largest $t$ for which there exists a $t$-admissible factorization $\tuple$ of $N!$ of cardinality $|\tuple|=N$.

\begin{example}\label{nine}  The multiset
  $$ \{ 3,3,3,3,4,4,5,7,8\}$$
  is a $3$-admissible factorization of
$$ \prod \{ 3,3,3,3,4,4,5,7,8\} = 3^4 \times 4^2 \times 5 \times 7 \times 8 = 9!$$
of cardinality
$$|\{ 3,3,3,3,4,4,5,7,8\}| = 9,$$
 hence $t(9) \geq 3$.  One can check that no $4$-admissible factorization of $9!$ of this cardinality exists, hence $t(9) = 3$.
\end{example}

It is easy to see that $t(N)$ is non-decreasing in $N$ (any cardinality $N$ factorization of $N!$ can be extended to a cardinality $N+1$ factorization of $(N+1)!$ by adding $N+1$ to the multiset).  The first few elements of the sequence $t(N)$ are
$$ 1,1,1,2,2,2,2,2,3,3,3,3,3,4, \dots$$
(\href{https://oeis.org/A034258}{OEIS A034258}). The values of $t(N)$ for $N \leq 79$ were computed in \cite{guy-selfridge}, and the values for $N \leq 200$ can be extracted from \href{https://oeis.org/A034259}{OEIS A034259}, which describes the inverse sequence to $t$.  As part of our work, we extend this sequence to $N \leq 10^4$; see \cite{github} and \Cref{fig-long}.

When the factorial $N!$ is replaced with an arbitrary number the problem of determining $t(N)$ is essentially the bin covering problem, which is known to be NP-hard; see e.g., \cite{bincover}.  However, as we shall see in this paper, the special structure of the factorial (and in particular, the profusion of factors at the ``tiny primes'' $2,3$) make it possible to estimate $t(N)$ with very high precision.  For instance, we are able to show (see \Cref{long-table}) that
$$ 0 \leq t(9 \times 10^8) - \num{316560601} \leq 113.$$

\begin{remark}\label{subfac}  One can equivalently define $t(N)$ as the greatest $t$ for which there exists a $t$-admissible \emph{subfactorization} of $N!$ of cardinality \emph{at least} $N$.  This is because every such subfactorization can be converted into a $t$-admissible factorization of cardinality exactly $N$ by first deleting elements from the subfactorization to make the cardinality $N$, and then multiplying one of the elements of the subfactorization by a natural number to upgrade the subfactorization to a factorization.  This ``relaxed'' formulation of the problem turns out to be more convenient both for theoretical analysis of $t(N)$ and for numerical computations.
\end{remark}

By combining the obvious lower bound
\begin{equation}\label{obvious}
 \prod \tuple \geq t^{|\tuple|}
\end{equation}
for any $t$-admissible multiset $\tuple$ with Stirling's formula \eqref{stirling}, we obtain the trivial upper bound
\begin{equation}\label{trivial} \frac{t(N)}{N} \leq \frac{(N!)^{1/N}}{N} = \frac{1}{e} + O\left( \frac{\log N}{N}\right);
\end{equation}
see \Cref{fig1}.  In \cite[p.75]{erdos-graham} it was reported that an unpublished work of Erd\H{o}s, Selfridge, and Straus established the asymptotic
\begin{equation}\label{t1}
  \frac{t(N)}{N} = \frac{1}{e} + o(1)
\end{equation}
(first conjectured in \cite{erdos-71}) and asked if one could show the bound
\begin{equation}\label{Tbound}
   \frac{t(N)}{N} \leq \frac{1}{e} - \frac{c}{\log N}
\end{equation}
for some constant $c>0$ and sufficiently large $N$ (see \cite[Section B22, p.\ 122--123]{guy} and problem {\#}391 in \url{https://www.erdosproblems.com}); it was also noted that similar results were obtained in \cite{algr77} if one restricted the $a_i$ to be prime powers.  However, as later reported in \cite{erdos-96}, Erd\H{o}s ``believed that Straus had written up our proof [of \eqref{t1}]. Unfortunately Straus suddenly died and no trace was ever found of his notes. Furthermore, we never could reconstruct our proof, so our assertion now can be called only a conjecture''.   In \cite{guy-selfridge} it was observed that the lower bound $\frac{t(N)}{N} \geq \frac{3}{16}-o(1)$ could be obtained by rearranging powers of $2$ in the standard factorization $N! = \prod \{1,\dots,N\}$, i.e., by removing some powers of $2$ from some of the terms and redistributing them to other terms.  It was also claimed that this bound could be improved to $\frac{t(N)}{N} \geq \frac{1}{4}$ for sufficiently large $N$ by rearranging powers of $2$ and $3$, however we have found surprisingly that this is not the case; see \Cref{main}(v) below.

The following conjectures in \cite{guy-selfridge} were also made:
\begin{enumerate}
\item One has $t(N) \leq N/e$ for $N \neq 1,2,4$.
\item One has $t(N) \geq \lfloor 2N/7 \rfloor$ for $N \neq 56$.
\item One has $t(N) \geq N/3$ for $N \geq 3 \times 10^5$.  (It was also asked if the threshold $3 \times 10^5$ could be lowered.)
\end{enumerate}
The second and third conjectures also appear in \cite{guy}.

In this paper we answer all of these questions.

\begin{theorem}[Main theorem]\label{main} Let $N$ be a natural number.
\begin{itemize}
\item[(i)] If $N \neq 1,2,4$, then $t(N) \leq N/e$.
\item[(ii)]  If $N \neq 56$, then $t(N) \geq \lfloor 2N/7 \rfloor$.  Moreover, this can be achieved by rearranging only the prime factors $2,3,5,7$.
\item[(iii)]  If $N \geq 43632$, then $t(N) \geq N/3$.  The threshold $43632$ is best possible.
\item[(iv)]  For large $N$, one has
  \begin{equation}\label{asym}
    \frac{t(N)}{N} = \frac{1}{e} - \frac{c_0}{\log N} + O\left( \frac{1}{\log^{1+c} N} \right)
  \end{equation}
for some constant $c>0$, where $c_0$ is the explicit constant
\begin{equation}\label{c0-def}
  \begin{split}
  c_0 &\coloneqq \frac{1}{e} \int_0^1 f_e(x)\ dx \\
  &= 0.30441901\dots
\end{split}
\end{equation}
and for any $\alpha>0$, $f_\alpha \colon (0,\infty) \to \R$ denotes the piecewise smooth function
\begin{equation}\label{falpha-def}
  f_\alpha(x) \coloneqq \left\lfloor \frac{1}{x} \right\rfloor \log \frac{\lceil \nicefrac{1}{\alpha x} \rceil}{\nicefrac{1}{\alpha x}}.
\end{equation}
In particular, \eqref{t1} and \eqref{Tbound} hold.
\item[(v)]  The largest $N$ for which one can demonstrate $t(N) \geq N/4$ purely by rearranging powers of $2$ and $3$ in the standard factorization $N! = \prod \{1,\dots,N\}$ is $26244$.
\end{itemize}
\end{theorem}

\begin{remark} In fact the upper bound \eqref{Tbound} can be sharpened to
\begin{equation}\label{tna}
  \frac{t(N)}{N} \leq \frac{1}{e} - \frac{c_0}{\log N} - \frac{c_1+o(1)}{\log^2 N}
\end{equation}
for an explicit constant $c_1=0.75554808\dots$; see \Cref{upper-bound}.
\end{remark}

\begin{figure}[hbp!]
  \centering
  \includegraphics[width=0.8\textwidth]{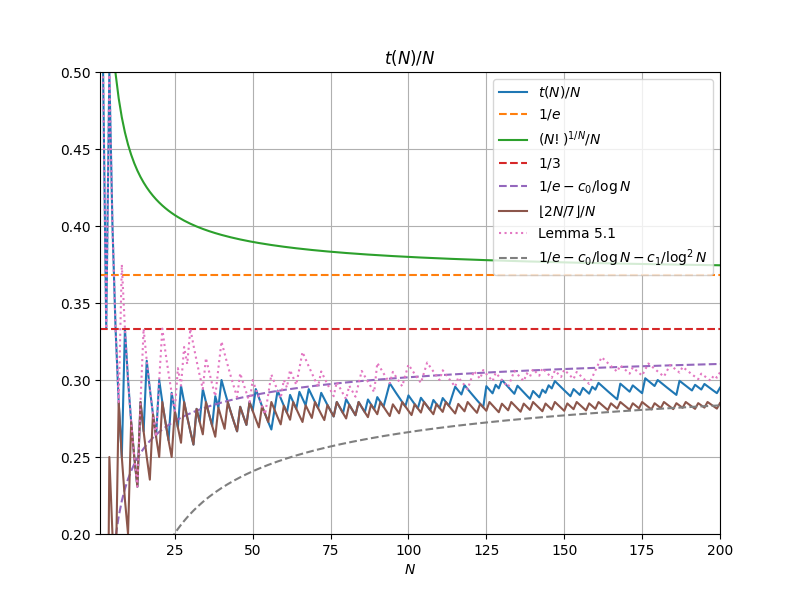}
  \vspace{-8pt}
  \caption{The function $t(N)/N$ (blue) for $N \leq 200$, using the data from \href{https://oeis.org/A034258}{OEIS A034258}, as well as the trivial upper bound $(N!)^{1/N}/N$ (green), the improved upper bound from \Cref{upper-crit} (pink), which is asymptotic to \eqref{asym} (purple), and the function $\lfloor 2N/7 \rfloor/N$ (brown), which we show to be a lower bound for $N \neq 56$.  \Cref{main}(iv) implies that $t(N)/N$ is asymptotic to \eqref{asym} (purple), which in turn converges to $1/e$ (orange), although we believe that \eqref{tna} (gray) asymptotically becomes a sharper approximation.  The threshold $1/3$ (red) is permanently crossed at $N=43632$.
  }\label{fig1}
  \end{figure}

  \begin{figure}
    \centering
    \includegraphics[width=0.8\textwidth]{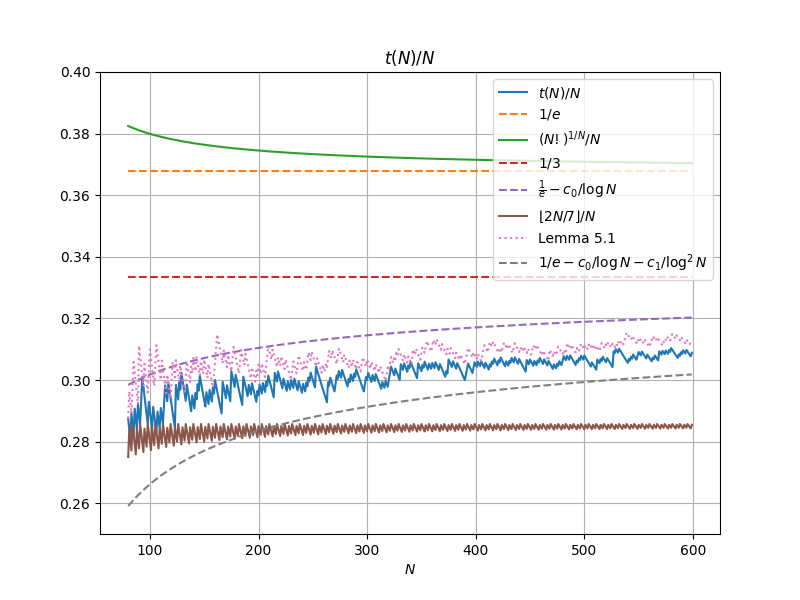}
    \vspace{-8pt}
    \caption{A continuation of \Cref{fig1} to the region $80 \leq N \leq 599$. }\label{fig1-alt}
  \end{figure}

\begin{figure}
  \centering
  \includegraphics[width=0.8\textwidth]{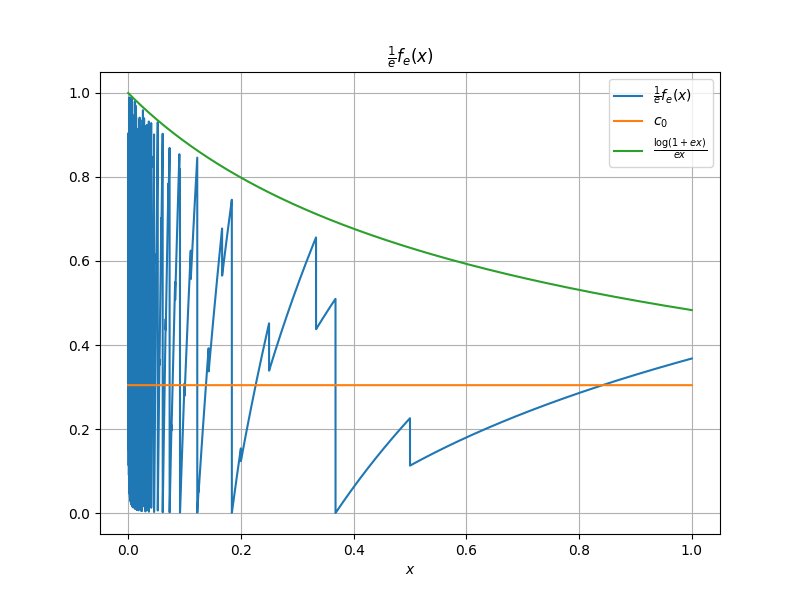}
  \caption{The piecewise continuous function $x\mapsto \frac{1}{e} f_e(x)$, together with its mean value $c_0 = 0.30441901\dots$ and the upper bound $\nicefrac{\log(1+ex)}{ex}$.  The function exhibits an oscillatory singularity at $x=0$ similar to $\sin \nicefrac{1}{x}$ (but it is always nonnegative and bounded). Informally, the function $f_e$ quantifies the difficulty that large primes in the factorization of $N!$ have in becoming only slightly larger than $N/e$ after multiplying by a natural number.}\label{fig-mean}
\end{figure}

For future reference, we observe the simple bounds
\begin{equation}\label{falpha-bound}
   0 \leq f_\alpha(x) \leq \frac{1}{x} \log \frac{\nicefrac{1}{\alpha x}+1}{\nicefrac{1}{\alpha x}}
= \frac{1}{x} \log\left( 1 + \alpha x \right) \leq \alpha
\end{equation}
for all $x>0$; in particular, $f_\alpha$ is a bounded function.  It however has an oscillating singularity at $x=0$; see \Cref{fig-mean}.

In \Cref{c0-app} we give some details on the numerical computation of the constant $c_0$.

\begin{remark}\label{old} In a previous version \cite{tao} of this manuscript, the weaker bounds
  $$ \frac{1}{e} - \frac{O(1)}{\log N} \leq \frac{t(N)}{N} \leq \frac{1}{e} - \frac{c_0+o(1)}{\log N}$$
were established, which were enough to recover \eqref{t1}, \eqref{Tbound}, and \Cref{main}(i). Numerically, the upper bound in \eqref{tna} appears to be a rather good approximation, and we conjecture that it is a lower bound as well.
\end{remark}

As one might expect, the proof of \Cref{main} proceeds by a combination of both theoretical analysis and numerical calculations.  Our main tools to obtain upper and lower bounds on $t(N)$ can be summarized as follows (and in \Cref{cases-table}):

\begin{itemize}
  \item In \Cref{greedy-sec}, we discuss \emph{greedy algorithms} to construct subfactorizations, that provide quickly computable, though suboptimal, lower bounds on $t(N)$ for small, medium, and moderately large values.
  \item In \Cref{linprog-sec}, we present a \emph{linear programming} (and \emph{integer programming}) method that provides quite accurate upper and lower bounds on $t(N)$ for small and medium values of $N$, and which we apply in \Cref{upper-sec} to establish a general upper bound (\Cref{upper-crit}) on $t(N)$ that can be used to obtain \Cref{main}(i).
  \item In \Cref{rearrange-sec}, we extend the \emph{rearrangement approach} from \cite{guy-selfridge} to give computer-assisted proofs of \Cref{main}(ii), \Cref{main}(iii) for sufficiently large $N$, and \Cref{main}(v).  We also give an analytic proof of \eqref{t1}.
  \item In \Cref{accounting-sec}, we introduce an \emph{accounting equation} linking the ``$t$-excess'' of a subfactorization with its ``$p$-surpluses'' at various primes, which provides an alternate proof of \Cref{upper-crit}, and also is the starting point for the modified factorization technique discussed below.
  \item In \Cref{approx-sec}, we give \emph{modified approximate factorization} strategy, which provides lower bounds on $t(N)$, that become asymptotically quite efficient.
\end{itemize}

\begin{table}[ht]
  \centering
  \begin{tabular}{|l|r@{}c@{}l|l|}
  \hline
\rule{0pt}{11pt}  Part & \multicolumn{3}{c|}{Range of $N$} & Method used \\
  \hline
\rule{0pt}{12pt}(i) & & $N$ & ${}\leq 10^4$ & Linear programming \\
 & $80<{}$ &$N$ & & \Cref{upper-crit} \\
\hline
\rule{0pt}{12pt}(ii) & & $N$ & ${}\leq 1.2\times 10^7$ & Integer programming \\
 & $8.2\times 10^6 \leq{}$ & $N$ & & Rearrangement \\
\hline
\rule{0pt}{12pt}(iii) & & $N$ & ${}\leq 8 \times 10^4$ & Integer programming \\
 & $\num{67425} \leq{}$ & $N$ & ${}\leq 10^{14}$ & Greedy  \\
 & $10^{11} \leq{}$ & $N$ & & Modified approximate factorization \\
 & \multicolumn{3}{c|}{$N$ sufficiently large} & Rearrangement \\
\hline
\rule{0pt}{12pt}(iv) (upper) & \multicolumn{3}{c|}{$N$ sufficiently large} & \Cref{upper-crit} \\
\phantom{(iv)} (lower) & \multicolumn{3}{c|}{$N$ sufficiently large} & Modified approximate factorization \\
\hline
\rule{0pt}{12pt}(v) & & $N$ & ${}\leq 5 \times 10^6$ & Linear (or dynamic) programming \\
 & $1.4\times 10^6 \leq{}$ & $N$ & & Rearrangement \\
  \hline
\end{tabular}
\bigskip

\caption{The techniques in this paper can establish the various components of \Cref{main}, for various overlapping  ranges of $N$.}\label{cases-table}
\end{table}

The final approach is significantly more complicated than the other four, but gives the most efficient lower bounds in the asymptotic limit $N \to \infty$.  The key idea is to start with an approximate factorization
\begin{equation}\label{approx_factor}
   N! \approx \left(\prod_{j \in I} j\right)^A
\end{equation}
for some relatively small natural number $A$ (e.g., $A = \lfloor \log^2 N \rfloor$) and a suitable set $I$ of natural numbers greater than or equal to $t$; there is some freedom to select parameters here, and we will take $I$ to be the natural numbers in $(t, t(1+\sigma)]$ that are $3$-rough (coprime to $6$), where $t$ is the target lower bound for $t(N)$ we wish to establish, and $\sigma \coloneqq \frac{3N}{tA}$ is chosen to bring the number of terms in the approximate factorization close to $N$.  With this choice of $I$, the product in \eqref{approx_factor} contains approximately the right number of copies of $p$ for medium-sized primes $p$; but it has the ``wrong'' number of copies of large primes, and is also constructed to avoid the ``tiny'' primes $p=2,3$.  One then performs a number of alterations to this approximate factorization to correct for the ``surpluses'' or ``deficits'' at various primes $p>3$, using the supply of available tiny primes $p=2,3$ as a sort of ``liquidity pool'' to efficiently reallocate primes in the factorization.  A key point will be that the incommensurability of $\log 2$ and $\log 3$ (i.e., the irrationality of $\log 3/\log 2$) means that the $3$-smooth numbers (numbers of the form $2^n 3^m$) are asymptotically dense (in logarithmic scale), allowing for other factors to be exchanged for $3$-smooth factors with little loss.  The weaker results mentioned in \Cref{old}  only used the prime $2$ as a supply of ``liquidity'', and thus encountered inefficiencies due to the inability to ``make change'' when approximating another factor by a power of two.

\subsection{Linear and integer programming solvers}

This paper uses linear programming as a mathematical tool, and we also use linear and integer programming solvers for computations.
The solvers used were Gurobi~\cite{gurobi} and \texttt{lp\_solve}~\cite{lpsolve}.
A comment on the blog of one of the authors brought to our attention a question on MathOverflow\footnote{\url{https://mathoverflow.net/questions/419722/reliability-of-ilp-approach-to-number-theoretic-optimization}} regarding the reliability of integer linear programming solvers.
There, Max Alekseyev considers the related problem of factoring $N!$ into \emph{two} integer factors and maximizing the smaller factor.
For a fixed small $N$, it is tempting to express this problem as an integer linear program in a relatively straightforward way, and then asking a integer linear programming solver for the solution.
Unfortunately, many solvers produce incorrect (and worse, inadmissible) solutions to these programs, even for $N\le 40$.

Despite the similarity of the statement of the problems, most of the linear programs we analyze are much better-behaved than the ones discussed in this question.
Specifically, the linear programs mentioned by Alekseyev have coefficients (of the form $\log p$) which are transcendental, and the primary cause of the numerical issues is the numerical difficulty in verifying inequalities.
Our coefficients are small integers (often of the form $\nu_p(j)$), and as a result, solvers should not encounter numerical issues in simply verifying that inequalities between integers hold, even when $N$ is many millions.
(Some of our linear programs have coefficients which are rational numbers with large denominators, such as in \Cref{onethird-large}, but all of these solutions are verified in exact arithmetic precisely because of these numerical concerns.)

Nonetheless, we sought to verify all uses of a linear program solver in our work.
One of the particularly pleasing properties of the use of linear program solvers is that they often produce output that can be verified for correctness, even if the solver has a bug or numerical issue.
Specifically, when an integer program solver produces a lower bound on $t(N)$ (or a related quantity like $t_{2,3,5,7}(N)$), this corresponds to an explicit factorization of $N!$, so we can verify the factorization directly.
Alternatively, when a linear program solver produces an upper bound on $t(N)$ (or a related quantity like $t_{2,3}(N)$), we can verify that the dual linear program solution output is admissible in exact arithmetic, possibly after rounding the output.
We have done this for all parts of our main \Cref{main}, so our results do not depend on the correctness of any linear program solver.
(Note that the solvers did not produce any incorrect solutions to our programs.)

We have also used integer program solvers to produce the data in \Cref{fig-t5,fig-t7,fig-long}.
The lower bounds from these figures, which correspond to explicit factorizations, have been verified directly.
The upper bounds have not been verified, but we do expect them all to be exactly correct.
The data from these Figures is not used in any of the results.
(The data in other tables and figures, such as \Cref{fig-t2,fig-t3}, has been verified.)

\subsection{Author contributions and data}

This project was initially conceived as a single-author manuscript by Terence Tao, but since the release of the initial preprint \cite{tao}, grew to become a collaborative project organized via the GitHub repository \cite{github}, which also contains the supporting code and data for the project.  The contributions of the individual authors, according to the CRediT categories\footnote{\url{https://credit.niso.org/}}, are as follows:

\begin{itemize}
\item Boris Alexeev: Formal Analysis, Investigation, Methodology, Software, Validation, Writing -- review \& editing.
\item Evan Conway: Formal Analysis, Investigation, Software.
\item Matthieu Rosenfeld: Software.
\item Andrew V. Sutherland: Formal Analysis, Investigation, Methodology, Software, Validation, Writing -- review \& editing.
\item Terence Tao: Conceptualization, Formal Analysis, Methodology, Project Administration, Visualization, Writing -- original draft, Writing -- review \& editing.
\item Markus Uhr: Formal Analysis, Software.
\item Kevin Ventullo: Software.
\end{itemize}

\subsection{Acknowledgments}

AVS is supported by Simons Foundation Grant 550033.  TT is supported by NSF grant DMS-2347850.  We thank Thomas Bloom for the web site \url{https://www.erdosproblems.com}, where TT learned of this problem, as well as Bryna Kra and Ivan Pan for  corrections. We thank the referees for valuable suggestions and corrections.

\section{Notation and basic estimates}

In this paper the natural numbers $\N=\{1,2,3,\dots\}$ will start at $1$.

We use the usual asymptotic notation $X = O(Y)$, $X \ll Y$, or $Y \gg X$ to denote an inequality of the form $|X| \leq CY$ for some absolute constant $C$; if we need this constant to depend on additional parameters, we will indicate this by subscripts, thus for instance $O_M(Y)$ denotes a quantity bounded in magnitude by $C_M Y$ for some $C_M$ depending on $M$.  We also write $X \asymp Y$ for $X \ll Y \ll X$. For effective estimates, we will use the more precise notation $O_{\leq}(Y)$ to denote any quantity whose magnitude is bounded by exactly at most $Y$. We also use $O_{\leq}(Y)^+$ to denote a quantity of size $O_{\leq}(Y)$ that is also non-negative, that is to say it lies in the interval $[0,Y]$.  We also use $o(X)$ to denote any quantity bounded in magnitude by $c(N) X$, for some $c(N)$ that goes to zero as $N \to \infty$. We also use $X = \Omega(Y)$ to denote an inequality of the form $|X| \geq CY$ for some absolute constant $C$.

If $S$ is a statement, we use $1_S$ to denote its indicator, thus $1_S=1$ when $S$ is true and $1_S=0$ when $S$ is false.  If $x$ is a real number, we use $\lfloor x \rfloor$ to denote the greatest integer less than or equal to $x$, and $\lceil x \rceil$ to be the least integer greater than or equal to $x$.

Throughout this paper, the symbol $p$ (or $p_0$, $p_1$, etc.) is always understood to be restricted to be prime.
We use $(a,b)$ to denote the greatest common divisor of $a$ and $b$, $a|b$ to denote the assertion that $a$ divides $b$, and $\pi(x) = \sum_{p \leq x} 1$ to denote the usual prime counting function.  For a natural number $n$, we use $P_+(n)$ and $P_-(n)$ to denote the largest and smallest prime factors of $n$, respectively, with the convention that $P_-(1) = P_+(1) = 1$.

We use $\nu_p(a/b) = \nu_p(a)-\nu_p(b)$ to denote the $p$-adic valuation of a positive rational number $a/b$, that is to say the number of times $p$ divides the numerator $a$, minus the number of times $p$ divides the denominator $b$.  For instance, $\nu_2(32/27)=5$ and $\nu_3(32/27)=-3$.
If one applies a logarithm to the fundamental theorem of arithmetic, one obtains the identity
\begin{equation}\label{ftoa}
  \sum_p \nu_p(r) \log p = \log r
\end{equation}
for any positive rational $r$.
For a natural number $n$, we can write
\begin{equation}\label{nup-form}
  \nu_p(n) = \sum_{j=1}^\infty 1_{p^j|n}.
\end{equation}
Upon taking partial sums, we recover Legendre's formula
\begin{equation}\label{legendre}
  \nu_p(N!) = \sum_{j=1}^\infty \left\lfloor \frac{N}{p^j} \right\rfloor = \frac{N - s_p(N)}{p-1}
\end{equation}
where $s_p(N)$ is the sum of the digits of $N$ in the base $p$ expansion.

Given a multiset of integers $\tuple$ that is a putative factorization of $N!$,
we refer to the quantity $\nu_p\left( \frac{N!}{\prod \tuple} \right)$ as the \emph{$p$-surplus} of $\tuple$ with respect to the target $N!$, and similarly refer to the negative $-\nu_p\left( \frac{N!}{\prod \tuple} \right) = \nu_p\left( \frac{\prod \tuple}{N!} \right)$ of this surplus as the \emph{$p$-deficit}, with the multiset being \emph{$p$-balanced} if the $p$-surplus (or $p$-deficit) is zero.  Thus, $\tuple$ is a (complete) factorization of $N!$ if it is balanced at every prime $p$, and it is a subfactorization if it is in balance or surplus at every prime $p$.

Let $M(N,t)$ denote the maximal cardinality of a $t$-admissible subfactorization of $N!$; thus, by \Cref{subfac}, $t(N) \geq t$ if and only if $M(N,t) \geq N$.

To bound the factorial, we have the explicit Stirling approximation \cite{robbins}
\begin{equation}\label{stirling}
\log N! = N \log N - N + \log \sqrt{2\pi N} + O_\leq^+\left(\frac{1}{12N}\right),
\end{equation}
valid for all natural numbers $N$.

\subsection{Approximation by \texorpdfstring{$3$}{3}-smooth numbers}

The primes $2,3$ will play a special role\footnote{One could also run analogous arguments with other sets of tiny primes; for instance, the initial version \cite{tao} of this paper only utilized the prime $2$ in this fashion.} in this paper and will be referred to as \emph{tiny primes}.
Call a natural number \emph{$3$-smooth} if it is the product of tiny primes, i.e., it is of the form $2^n 3^m$ for some natural numbers $n,m$, and \emph{$3$-rough} if it is not divisible by any tiny prime, that is to say it is coprime to $6$.  Given a positive real number $x$, we use $\lceil x \rceil^{\langle 2,3 \rangle}$ to denote the smallest $3$-smooth number greater than or equal to $x$.  For instance, $\lceil 5 \rceil^{\langle 2,3 \rangle} = 6$ and $\lceil 10 \rceil^{\langle 2,3 \rangle} = 12$.

It will be convenient to introduce a variant of this quantity that is close to a power of $12$. The significance of the base $12$ is that the $3$-smooth portion $2^{\nu_2(N!)} 3^{\nu_3(N!)}$ of $N!$, which serves as our ``liquidity pool'', is approximately $2^N 3^{N/2} = \sqrt{12}^{N}$; see \eqref{legendre} above.  This makes $\log \sqrt{12}$ a natural ``unit of currency'' in which to conduct various factor exchanges, with various integer linear combinations of $\log 2$ and $\log 3$ usable as ``small change'' to approximate quantities that are not integer multiples of $\log \sqrt{12} = \log 2 + \frac{1}{2} \log 3$.

If $1 \leq L \leq x$ is an additional real parameter, we define
\begin{equation}\label{fancy-kappa-def}
  \lceil x \rceil^{\langle 2,3\rangle}_L \coloneqq 12^a \lceil x/12^a \rceil^{\langle 2,3 \rangle}
\end{equation}
for any real $x \geq L \geq 1$, where $a \coloneqq \left\lfloor \frac{x/L}{\log 12} \right\rfloor$ is the largest integer such that $12^a \leq x/L$.

For any $L \geq 1$, let $\kappa_L$ be the least quantity such that
\begin{equation}\label{kappa-def}
  x \leq \lceil x \rceil^{\langle 2,3\rangle} \leq \exp(\kappa_L) x
\end{equation}
holds for all $x \geq L$; see \Cref{fig:nextsmooth}.  In \Cref{power-sec} we establish the following facts:

\begin{figure}[hbp!]
  \centering
  \includegraphics[width=0.8\textwidth]{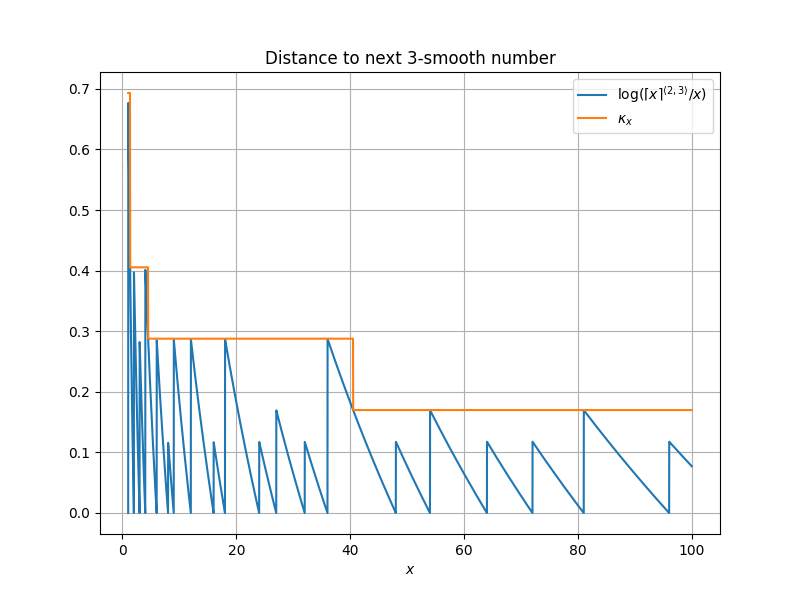}
  \vspace{-8pt}
  \caption{The function $\log \frac{\lceil x \rceil^{\langle 2,3 \rangle}}{x}$, compared against $\kappa_x$.
  }\label{fig:nextsmooth}
  \end{figure}

\begin{lemma}[Approximation by $3$-smooth numbers]\label{power-lemma}\
\begin{itemize}
\item[(i)]  We have $\kappa_{4.5} = \log\frac{4}{3} = 0.28768\dots$ and $\kappa_{40.5} = \log \frac{32}{27} = 0.16989\dots$.
\item[(ii)]  For large $L$, one has $\kappa_L \ll \log^{-c} L$ for some absolute constant $c>0$.
\item[(iii)]  If $1 \leq L \leq x$ are real numbers, then
\begin{equation}\label{mod-kappa}
  x \leq \lceil x \rceil^{\langle 2,3\rangle}_L \leq \exp(\kappa_L) x
\end{equation}
and for any $0 \leq \gamma < 1$ we have
\begin{equation}\label{12-2}
\frac{\nu_2(\lceil x \rceil^{\langle 2,3\rangle}_L) - 2 \gamma \nu_3(\lceil x \rceil^{\langle 2,3\rangle}_L)}{1-\gamma} \leq \frac{\log x  + \kappa^{(2)}_{L,\gamma} }{\log \sqrt{12}}
\end{equation}
and
\begin{equation}\label{12-3}
\frac{2\nu_3(\lceil x \rceil^{\langle 2,3\rangle}_L) - \gamma \nu_2(\lceil x \rceil^{\langle 2,3\rangle}_L)}{1-\gamma} \leq \frac{\log x + \kappa^{(3)}_{L,\gamma} }{\log \sqrt{12}}
\end{equation}
where
\begin{equation}\label{kappastar-2-def}
\kappa^{(2)}_{L,\gamma} \coloneqq \left(\frac{\log \sqrt{12}}{(1-\gamma)\log 2} - 1\right) \log(12L) + \frac{\kappa_L\log \sqrt{12}}{(1-\gamma) \log 2},
\end{equation}
\begin{equation}\label{kappastar-3-def}
  \kappa^{(3)}_{L,\gamma} \coloneqq \left(\frac{\log \sqrt{12}}{(1-\gamma)\log \sqrt{3}} - 1\right) \log(12L) + \frac{\kappa_L\log \sqrt{12}}{(1-\gamma)\log \sqrt{3}}.
\end{equation}
\end{itemize}
\end{lemma}

We remark that when $x$ is a power of $12$, the left-hand sides of \eqref{12-2}, \eqref{12-3} are both equal to $\frac{\log x}{\log \sqrt{12}}$; thus the estimates \eqref{12-2}, \eqref{12-3} are quite efficient asymptotically.

We use the notation $\sum^*$ to denote summation restricted to $3$-rough numbers, thus for instance $\sum_{a < k \leq b}^* 1$ denotes the number of $3$-rough numbers in $(a,b]$.  We have a simple estimate for such counts:

\begin{lemma}\label{lit}  For any interval $(a,b]$ with $0 \leq a \leq b$ one has $\sum_{a < k \leq b}^* 1 = \frac{b-a}{3} + O_{\leq}(4/3)$.
\end{lemma}

\begin{figure}
  \centering
  \includegraphics[width=0.8\textwidth]{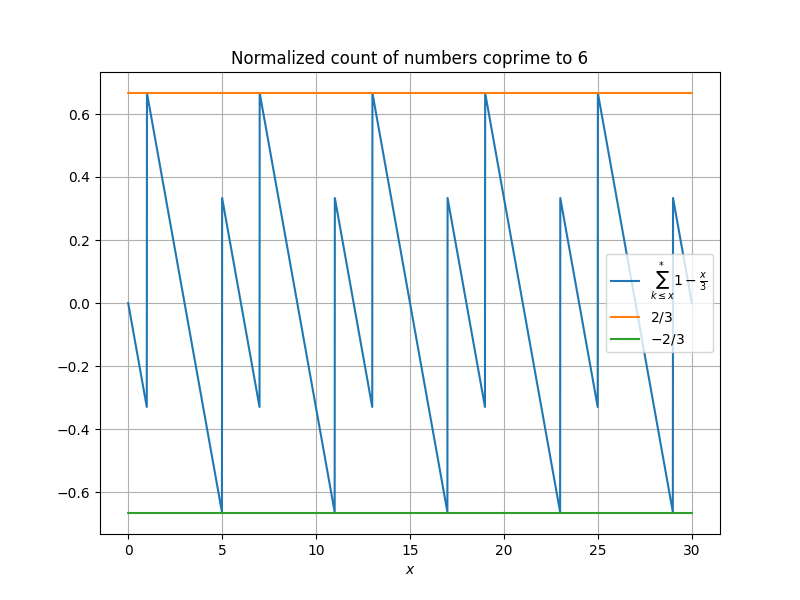}
  \vspace{-8pt}
  \caption{The function $\sum_{k \leq x}^* 1 - \frac{x}{3}$.}\label{fig-saw}
\end{figure}

\begin{proof}  By the triangle inequality, it suffices to show that $\sum_{0 < k \leq x}^* 1 - \frac{x}{3} = O_{\leq}(2/3)$ for all $x \geq 0$.  This is easily verified for $0 \leq x \leq 6$, and the left-hand side is $6$-periodic in $x$, giving the claim; see \Cref{fig-saw}.
\end{proof}

\subsection{Sums over primes}

We recall the effective prime number theorem from \cite[Corollary 5.2]{dusart}, which asserts that
\begin{equation}\label{pi-lower}
  \pi(x) \geq \frac{x}{\log x} + \frac{x}{\log^2 x}
\end{equation}
for $x \geq 599$ and
\begin{equation}\label{pi-upper}
  \pi(x) \leq \frac{x}{\log x} + \frac{1.2762 x}{\log^2 x}
\end{equation}
for $x >1$.

We will also need to control sums of somewhat oscillatory functions over primes, for which the bounds in \eqref{pi-lower}, \eqref{pi-upper} are of insufficient strength. Let $y<x$ be real numbers. Given a function $b \colon (y,x] \to \R$, its \emph{total variation}
$\|b\|_{\mathrm{TV}(y,x]}$ is defined as the supremum of the quantities $\sum_{j=0}^{J-1} |b(x_{j+1})-b(x_j)|$ for $y < x_0 \leq \dots \leq x_J \leq x$, and the \emph{augmented total variation} $\|b\|_{\mathrm{TV}^*(y,x]}$ is defined as
$$
\|b\|_{\mathrm{TV}^*(y,x]}
\coloneqq |b(y^+)| + |b(x)| + \|b\|_{\mathrm{TV}(y,x]},$$
where $b(y^+) \coloneqq \lim_{t \to y^+} b(t)$ denotes the right limit of $b$ at $y$ (which exists if $b$ is of finite total variation).  Equivalently, $\|b\|_{\mathrm{TV}^*(y,x]}$ is the total variation of $b$ if extended by zero outside of $(y,x]$. The indicator function $1_{(y,x]}$ clearly has an augmented total variation of $2$.

We will use this augmented total variation to control sums over primes.  More precisely, in \Cref{primes-sec} we will show

\begin{lemma}[Effective bounds for oscillatory sums over primes]\label{osc-lemma}  Let $1423 \leq y \leq x$, and let $b: (y,x] \to \R$ be of bounded total variation.
  \begin{itemize}
    \item [(i)] Wwe have the bound
\begin{equation}\label{bv-exact}
    \sum_{y < p \leq x} b(p) \log p = \int_y^x \left(1-\frac{2}{\sqrt{t}}\right) b(t)\ dt + O_{\leq}(\|b\|_{\mathrm{TV}^*(y,x]} E(x))
\end{equation}
where the error function $E(x)$ is defined as
\begin{equation}\label{tilde-e}
  E(x) \coloneqq 0.95 \sqrt{x} + 3.83 \times 10^{-9} x.
\end{equation}
  \item[(ii)] One has
\begin{equation}\label{pix}
  \pi(x) - \pi(y) = \int_y^x \left(1-\frac{2}{\sqrt{t}}\right)\frac{dt}{\log t} + O_\leq\left(2 \frac{E(x)}{\log y}\right),
\end{equation}
the upper bound
\begin{equation}\label{pixy-upper}
 \pi(x) - \pi(y) \leq \frac{x-y}{2\log y} + \frac{x-y}{2\log x} + 2 \frac{E(x)}{\log y}
\end{equation}
and the lower bound
\begin{equation}\label{pixy-lower}
  \pi(x) - \pi(y) \geq \left(1-\frac{2}{\sqrt{y}}\right) \frac{x-y}{\log \frac{x+y}{2}} - 2 \frac{E(x)}{\log y}.
\end{equation}
  \item[(iii)] If $b$ is non-negative, we have the upper bound
\begin{equation}\label{bv-upper}
   \sum_{y < p \leq x} b(p) \leq \frac{1}{\log y} \int_y^x b(t)\ dt + \|b\|_{\mathrm{TV}^*(y,x]} \frac{E(x)}{\log y}
\end{equation}
and the lower bound
\begin{equation}\label{bv-lower}
  \sum_{y < p \leq x} b(p) \leq \frac{1-\frac{2}{\sqrt{y}}}{\log x} \int_y^x b(t)\ dt - \|b\|_{\mathrm{TV}^*(y,x]} \frac{E(x)}{\log x}.
\end{equation}
\end{itemize}
One can replace all occurrences of $E(x)$ here by the classical error term $O(x \exp(-c \sqrt{\log x}))$ for some absolute constant $c>0$ (in which case the $\frac{2}{\sqrt{t}}$ type terms can be absorbed into the error term).
\end{lemma}

We remark that the accuracy in \eqref{bv-exact}, \eqref{pix} in particular is on par with what would be provided by the Riemann hypothesis, as long as $x$ is not too large (e.g., $x \leq 10^{18}$).  The other estimates in this lemma are not quite as precise, but are still adequate for our applications.  The error term $E(x)$ can be improved somewhat for large $x$ (see \eqref{etil-def}), but this simplified version will suffice for our analysis (in particular, the contribution of the second term in \eqref{tilde-e} will be negligible for our applications).  We make the easy remark that $E(x)$ is non-decreasing in $x$, while $E(x)/x$ is non-increasing.

\section{Greedy algorithms}\label{greedy-sec}

Recall that $t(N)$ can be interpreted as the largest $t$ for which one has $M(N,t) \geq N$, where $M(N,t)$ denotes the cardinality of the largest $t$-admissible subfactorization of $N$. Because of this, any algorithm that can produce lower bounds $M$ on $M(N,t)$ can also produce lower bounds on $t(N)$, as follows:

\begin{itemize}
  \item[Step 0:] Start with a heuristic lower bound $t$ for $t(N)$.
  \item[Step 1:] Use the provided algorithm to compute a lower bound $M(N,t)\ge M$.
  \item[Step 2:] If $M \ge N$, either \texttt{HALT} and report $t(N)\ge t$, or increase $t$ by some amount (possibly guided by the extent to which the lower bound $M$ exceeds $N$), and return to Step 1.
  \item[Step 3:] If $M < N$, decrease $t$ by some amount (possibly guided by the extent to which the lower bound $M$ falls short of $N$) and return to Step 1.
\end{itemize}

One can similarly use an algorithm that produces upper bounds for $M(N,t)$ to produce upper bounds for $t(N)$.

The algorithm described above is imprecisely specified, because it requires one to make some implementation decisions about how to select the parameter $t$ at various steps of the algorithm. In particular, having some accurate heuristics (or ``hints'') about what the correct value of $t(N)$ should be (possibly based on the outcomes of previous stages of the algorithm) can greatly accelerate its performance.  But regardless of this variability in speed, the $M(N,t)$ algorithm will in practice produce a certificate (e.g., an explicit subfactorization of $N!$) that can be quickly and independently verified by a separate computer program to confirm the lower bound on $t(N)$.  So the output of such imprecisely specified algorithms can at least be independently confirmed, if not reproduced exactly.  In particular, a lack of reproducibility does not prevent verification of a specific bound on $t(N)$, such as $t(N) \geq N/3$, so long as independently verifiable proof certificates (such as an $N/3$-admissible subfactorization of $N!$ of length at least $N$) are generated.

We therefore turn to the question of how to algorithmically obtain good upper and lower bounds on $M(N,t)$.  In this section we will discuss greedy methods to obtain lower bounds on this quantity; in the next section we will discuss how linear programming and integer programming methods can also be used to obtain both upper and lower bounds on $M(N,t)$.

The following \textit{greedy algorithm} produces reasonably good lower bounds on $M(N,t)$:

\begin{enumerate}
\item[Step 0:] Initialize $\tuple$ to be the empty multiset.
\item[Step 1:] If $\tuple$ is not a factorization of $N!$, determine the largest $p$ in surplus: $\nu_p\bigl(N!/\prod \tuple\bigr) > 0$.
\item[Step 2:] If $N! / \prod \tuple$ is divisible by a multiple of $p$ greater than or equal to $t$, determine the smallest such multiple, add it to $\tuple$, and return to Step 1.  Otherwise, \texttt{HALT}.
\end{enumerate}

This procedure clearly halts in finite time and produces a $t$-admissible subfactorization of $N!$.
The length of this subfactorization gives a lower bound on $M(N,t)$ that can be used to obtain lower bounds on $t(N)$ as discussed above.  For instance, applying this procedure with $N=9$, $t=3$ produces the $3$-admissible subfactorization
\[
\bigl\{7 \cdot 1,\ 5 \cdot 1,\ 3 \cdot 1,\ 3 \cdot 1,\ 3 \cdot 1,\ 3 \cdot 1,\ 2 \cdot 2,\ 2 \cdot 2,\ 2 \cdot 2 \bigr\},
\]
which recovers the bound $M(9,3) \geq 9$ (and hence $t(9) \geq 3$) from \Cref{nine}, albeit with a slightly different subfactorization, in which the $8$ is replaced by $4$.

The greedy approach works well for small $N$, producing the exact value of $t(N)$ for $N \leq 79$, but the quality of the bounds on $t(N)$ it produces declines as $N$ grows; see \Cref{fig-zoom}.  Its performance is also respectable (though not optimal) for medium $N$; for instance, when $N=3 \times 10^5$ and $t=N/3$, it establishes the lower bound $M(N,t) \geq N+372$, which is close to the exact value $M(N,t) = N+455$ we establish using the linear programming methods of the next section; see \eqref{mnt} and \eqref{mnt-lower}.

To handle the larger values of $N$ needed to establish Theorem~\ref{main}(iii) for $N\in [ 8\times 10^4, 10^{11}]$, and also for the broader range $N\in [67425,10^{14}]$ that comfortably overlaps regions in which we may apply other methods, we now consider how to efficiently implement Steps 1 and 2 of the greedy algorithm outlined above.  Let $p_1\ge p_2\ge \ldots \ge 2$ be the prime factors of $N>1$ listed with multiplicity in non-increasing order, and let $m_ip_i\ge t$ be the factor of $N!$ chosen by the greedy algorithm for the prime $p_i$ on input $t < N/2$.  Our implementation of the greedy algorithm is based on the following observations:

\begin{itemize}
\item[(a)] For large primes $p_i$ the greedy algorithm will always use $m_i=\lceil t/p_i\rceil$, with $m_j=m_i$ for all $p_j\in \left[\lceil t/m_i\rceil,\lceil t/(m_i-1)\rceil\right)$, including all $p_j\ge t$ for $m_i=1$.
\item[(b)] The sequence $m_i$ is nondecreasing ($m_{i+1}p_{i+1}\ge t$ implies $m_{i+1}p_i\ge t$ with $m_{i+1}$ dividing $\tuple$ after the $i$th step, so the greedy choice of $m_i$ satisfies $m_i\le m_{i+1}$).
\item[(c)] Each $m_i$ is $p_i$-smooth, and the ratio of the largest and smallest prime divisors of $m_i$ cannot exceed $t/m_i$ (if it did we could remove the smallest prime divisor from $m_i$).
\end{itemize}

Observation (a) allows us to efficiently handle large $p_i$, observation (b) enables an $O(N^{1+\epsilon})$ running time, and observation (c) allows us to more efficiently handle small $p_i$, which turns out to be the main bottleneck of the simple greedy algorithm sketched above.

In Section \ref{sec:fast-greedy} we describe a variant of the greedy algorithm that produces slightly weaker lower bounds on $M(N,t)$, but yields a power savings in the running time: we show that one can achieve an $O(N^{\nicefrac{2}{3}+\epsilon})$ running time asymptotically. The complexity of our implementation of this faster variant is actually $O(N^{\nicefrac{3}{4}+\epsilon})$, but it is faster than the asymptotically superior approach in the range of $N$ we are most concerned with.  While one might suppose that an algorithm whose output is a factorization of $N!$ into $N$ factors would require $\Omega(N)$ time (and $\Omega(N\log N)$ space for the output), we can compress this factorization using tuples of the form $(n,m,p_{\min},p_{\max})$ to represent $n$ occurrences of factors $mp$ for each prime $p\in [p_{\min},p_{\max}]$; for example, the tuple $(\pi(N)-\pi(t-1),1,t,N)$ encodes all factors $p_i\ge t$ in a single tuple.  This compression allows us to encode certificates of a subfactorization in $O(N^{\nicefrac{1}{2}+\epsilon})$ space that can be verified in $O(N^{\nicefrac{2}{3}+\epsilon})$ time, which is an important consideration for large $N$ (e.g. $N\approx 10^{14}$).

\subsection{Implementation of the greedy algorithm}
We assume $1< N/4< t < N/2$ throughout the rest of this section and fix $M \approx N^{\nicefrac{1}{2}}$.
Let $x_0=\lfloor N/M\rfloor \approx N^{\nicefrac{1}{2}}$, and partition the interval $(M,N]$ into subintervals $(x_k,x_{k+1})$ on which the step functions $f_t(x)\coloneqq\lceil t/x\rceil$ and $g_N(x)\coloneqq\lfloor N/x\rfloor$ are both constant, such that $x_r=N$, and $x_k$ is a point of discontinuity for at least one of $f_t(x)$ and $g_N(x)$ for $0<k<r$.  Under our assumption that the greedy algorithm uses optimal cofactors $m_i=\lceil t/p_i\rceil=f_t(p_i)$ for each prime $p_i>M$, we will have $g_N(x_k)$ factors $m_ip_i\ge t$ with $m_i=f_t(x_{k+1})$ for each prime $p_i\in (x_k,x_{k+1}]$, and we can compute the number of factors produced by the greedy algorithm that are divisible by a prime $p>M$ as
\begin{equation}\label{eq:picount}
\sum_{k=0}^{r-1} f_N(x_k)\left(\pi(x_{k+1})-\pi(x_k)\right).
\end{equation}
If $\tuple$ is the multiset of the factors $m_ip_i$ with $p_i\ge M$, for primes $p < M$ we can compute
\begin{equation}\label{eq:rval}
\nu_p\left(N!/\prod \tuple\right) = \nu_p(N!) - \sum_{k=0}^{r-1} f_N(x_k)\nu_p(f_t(x_{k+1}))
\end{equation}
using a precomputed table of factorizations of integers $m < N^{\nicefrac{1}{2}}$ and $\nu_p(N!)=\sum_{e=1}^{\lfloor \log_p(N)\rfloor} \Bigl\lfloor \frac{N}{p^e}\Bigr\rfloor$ in $O(N^{\nicefrac{1}{2}+\epsilon})$ time.  For $p < M$ we expect to have
\[
\sum_{k=0}^{r-1}f_N(x_k)\nu_p(f_t(x_{k+1})) \approx \frac{1}{p-1}\int_M^t \frac{N}{x\log x}dx \approx\frac{\log 2}{p-1} < \frac{1}{p-1} \approx \nu_p(N!),
\]
which motivates our observation (a) above.

\begin{remark} While our algorithm is based on the heuristic assumption that \eqref{eq:rval} is nonnegative for all $p<M$, it verifies this assumption at runtime. This verification did not fail in any of the computations used to prove that \Cref{main}(iii) holds for $67425\le N\le 10^{14}$, which is all that is needed for our results.  But if it were to fail, one could simply increase $M$ until it does not, and one can show that this will happen with $M=O(N^{\nicefrac{1}{2}+\epsilon})$, meaning that there is no impact on the asymptotic running time.  We have verified that \eqref{eq:rval} is nonnegative for all $10\le N\le 10^8$ and all primes $p \le M=\lceil\sqrt{N}\rceil$ with $t=\lceil N/3\rceil$.
\end{remark}

Computing the sum in \eqref{eq:picount} involves computing $\pi(x_k)$ for $O(N^{\nicefrac{1}{2}})$ values of $x_k\in (N^{1/2},N]$.
Up to a constant factor, this is the same as the cost of computing $\pi(N/x)$ for all positive integers $x \le N^{\nicefrac{1}{2}}$.
There are analytic methods to compute $\pi(x)$ in $O(x^{\nicefrac{1}{2}+\epsilon})$ time \cite{lagarias-odlyzko}, which implies that the time to compute \eqref{eq:picount} can be bounded by
\[
\sum_{x=1}^{N^{\nicefrac{1}{2}}} \left(\frac{N}{x}\right)^{\nicefrac{1}{2}+\epsilon} \ll \int_1^{N^{\nicefrac{1}{2}}} \left(\frac{N}{x}\right)^{\nicefrac{1}{2}+\epsilon}\!\! dx = O\left(N^{\nicefrac{3}{4}+\epsilon}\right).
\]
We can improve the running time by enumerating primes up to $N^{\nicefrac{2}{3}}$ using a sieve and computing $\pi(x)$ for $x\le N^{\nicefrac{2}{3}}$ as we go.  This yields an $O(N^{\nicefrac{2}{3}+\epsilon})$ bound on the time to compute \eqref{eq:picount}, and can be accomplished using $O(N^{\nicefrac{1}{3}+\epsilon})$ space.

In practice, it is more common to compute $\pi(x)$ using the $O(x^{\nicefrac{2}{3}+\epsilon})$ algorithm described in \cite{deleglise-rivat,lagarias-miller-odlyzko}, for which high-performance implementations are widely available.  In this case the optimal asymptotic approach is to sieve primes up to $N^{3/4}$, yielding an $O(N^{\nicefrac{3}{4}+\epsilon})$ algorithm. In our implementation we used a  slightly smaller sieving bound that more evenly balances the time spent sieving primes versus counting them in the range $N\le 10^{14}$ via the \texttt{primesieve}~\cite{walisch-primesieve} and \texttt{primecount} \cite{walisch-primecount} libraries used in our implementation.

If we precompute the prime factorizations of the positive integers $m\le t/M$, we can compute \eqref{eq:rval} for all primes $p \le M$ in $O(N^{\nicefrac{1}{2}+\epsilon})$ time (and verify our assumption that it is nonnegative).  This is all the information we will need in the next phase of the algorithm, which partitions the remaining $M$-smooth part of $N!$ into factors of size at least $t$.

We now recall observations (b) and (c) above, that the $m_i$ are nondecreasing and $p_i$-smooth.    This means that we can  precompute a table of $M$-smooth integers $m\le t$ and process them in increasing order as we consider decreasing primes $p_i$, thereby obtaining a quasilinear running time $O(N^{1+\epsilon})$. At each step the algorithm will determine the largest exponent $e$ such that $(m_i p_i)^e$ divides $N!/\tuple$, so we will have $m_i=m_{i+1}=\cdots = m_{i+n-1}$ and $p_i=p_{i+1}=\cdots = p_{i+n-1}$, with either $p_{i+n}<p_i$ or $m_{i+n} > m_i$ (possibly both).
The running time is then dominated by the time to precompute the prime factorizations of all the candidate $m_i$.  The additional constraint on the prime factors of the $m_i$ noted in (c) reduces the number of candidate cofactors $m$ we need to store in memory by a logarithmic factor.  This does not change the $O(N^{1+\epsilon})$ complexity bound, but it is significant in the practical range of interest, where it reduces the memory required by up to a factor of about 40 for $N\le 10^{11}$.  But we can handle much larger values of $N$ by modifying the algorithm as described in the next subsection.

\subsection{A fast variant of the greedy algorithm}\label{sec:fast-greedy}
We now give a variant of the greedy algorithm that produces slightly weaker bounds on $t(N)$ in general, but obtains an $O(N^{\nicefrac{3}{4}+\epsilon})$ running time using $O(N^{\nicefrac{2}{3}+\epsilon})$ space (and the space can be reduced to $O(N^{\nicefrac{1}{2}+\epsilon})$).  The algorithm fixes $M$ satisfying $M(M-1)\ge t$ and treats the primes $p_i\ge M$ exactly as in the first phase of the greedy algorithm described above, using $O(N^{\nicefrac{3}{4}+\epsilon})$ time and $O(N^{\nicefrac{2}{3}+\epsilon})$ space.

In the second phase, rather than precomputing a list of all candidate $m_i\le N$, the algorithm instead precomputes a list of $M$-smooth integers $m\le N^{\nicefrac{2}{3}}$.  As it considers the small primes $p_i < M$ in decreasing order, it will eventually reach a point where no precomputed value of~$m$ is a suitable cofactor for $p_i$ (this will certainly happen for $p_i < N^{1/3}$).  When this occurs it will instead look for a cofactor that is suitable for $p_i^2$, which will be smaller and easier to construct from the remaining part of the factorization of $N!$, allowing the algorithm to remove all but at most one factor of $p_i$.  It will continue in this fashion to consider suitable cofactors for both $p$ and $p^2$ until it eventually reaches a point where neither can be found, at which point all the remaining $p_i$ are either very small, $O(N^\epsilon)$, or occur with multiplicity 1.  In the final phase we simply construct factors of $N!$ larger than $t$ by combining available remaining primes in decreasing order.  This will occasionally result in factors that are substantially larger than the original greedy algorithm would use, but there are only a small number of these and the algorithm can construct them quickly using very little memory.  This allows it to handle large values of $N$ much more efficiently, as can be seen in in the timings in \Cref{greedy-table}.

For this fast variant of the greedy algorithm, in contrast to the original greedy algorithm, the computation is dominated by the first phase, which takes $O(N^{\nicefrac{3}{4}+\epsilon})$ time to handle the primes $p_i>M$; the rest of the algorithm takes only $O(N^{\nicefrac{2}{3}+\epsilon})$ time.

\subsection{Optimizing bounds produced by the greedy algorithm}
On inputs $N$ and~$t$ the greedy algorithm produces a lower bound on $M(N,t)$. If this lower bound is greater than or equal to $N$ we can deduce $t(N)\ge t$, but it may be possible to prove a better lower bound on $t(N)$ using a larger value of $t$, and this is desirable even in the context of proving \Cref{main}(iii) where it would suffice to use $t=\lceil N/3\rceil$.  The function $t(N)$ is nondecreasing, since adding $N+1$ to a $t$-admissible factorization of $N!$ yields a $t$-admissible factorization of $(N+1)!$. It follows that if $t(N) \ge (N+\delta)/3$ for some integer $\delta >0$ then $t(N')\ge N'/3$ for all $N'\in [N,N+\delta]$ (a range that may include $N'$ for which the greedy algorithm cannot directly prove $t(N')\ge N'/3$).

For large $N$ we expect to be able to choose $t$ so that the greedy algorithm (and its fast variant) can prove $M(N,t)\ge N$, and therefore $t(N)\ge t$, using $t=\lceil (N+\delta)/3\rceil$ with $\delta \ge cN$ for some $c>0$ that approaches $1/e-1/3$ as $N\to \infty$.  In the context of proving \Cref{main}, this allows us to establish $t(N)\ge N/3$ for all $N$ in any sufficiently large dyadic interval using just $O(1)$ calls to the greedy algorithm or its fast variant, provided that we are able to choose (or precompute) suitable values of $t$.

Let $t_0(N)$ denote the least $t$ for which the greedy algorithm proves $M(N,t)\ge N$ for all $t\le t_0(N)$.  Let $t_1(N)$ denote the largest $t$ for which the greedy algorithm proves $M(N,t_1(N))\ge N$ but cannot prove $M(N,t)\ge N$ for any $t>t_1(N)$.  There will typically be a substantial gap between $t_0(N)$ and $t_1(N)$, but for large $N$ we expect both to exceed $N/3$ by a constant factor.  In this section we consider two problems:

\begin{itemize}
\item[(a)] Given $N$, quickly produce some $t\in [t_0(N),t_1(N)]$.
\item[(b)] Compute the exact value of $t_1(N)$.
\end{itemize}

Our solution to (a) suffices to establish Theorem~\ref{main}(iii) for $N\in [10^6,10^{14}]$ using the fast variant of the greedy algorithm.
Our solution to (b) allows us to extend this range to $[67425,10^{14}]$.  There are smaller $N$ for which $t_1(N)\ge N/3$, the least of which is $N=44716$, but there is no way to extend the lower end of the range $[67425,10^{14}]$ using the greedy algorithm, as there is no choice of $N$ or $t$ that will allow the greedy algorithm to prove $t(67424) \ge 67424/3$ (even indirectly); here we need the linear programming methods described in the next section.

A simple solution to (a) uses a bisection search: pick initial values $t_{\rm low} = N/4$ and $t_{\rm high} > N/2$ for which we know $[t_{\rm low},t_{\rm high})$ contains $[t_0(N),t_1(N)]$ invoke the greedy algorithm (or its fast variant) with $t=\lceil(t_{\rm low}+t_{\rm high})/2\rceil$ and update $t_{\rm low}$ or $t_{\rm high}$ as appropriate, depending on whether the greedy algorithm proves $M(N,t)\ge N$ or not.  After $O(\log N)$ iterations we will have $t_{\rm low} = t_{\rm high}-1$, at which point we know $t_{\rm low}\in [t_0(N),t_1(N)]$.

We can do slightly better by using the value of the bound on $M(N,t)$ determined by the greedy algorithm in each iteration to guide the search, rather than simply checking whether it is above or below $N$.  Rather than choosing $t=\lceil(t_{\rm low}+t_{\rm high})/2\rceil$, we start with $t=\lceil N/3\rceil$ and in each iteration we replace the most recently tested value of $t$ with the nearest integer to
\[
\exp\left(\frac{B}{N}\log t\right),
\]
where $M(N,t)\ge B$ is the bound proved by the greedy algorithm on input $t$,  subject to the constraint that this value must lie in the interval $[t_{\rm low},t_{\rm high})$
(we use $(3t_{\rm low}+t_{\rm high})/4$ if it is below the interval and $(t_{\rm low}+3t_{\rm high})/4$ if it is above).  We then update $t_{\rm low}$ and $t_{\rm high}$ as above.
In practice this heuristic method converges about twice as fast as a standard bisection search.

We now consider the more challenging problem of computing $t_1(N)$.  It is not clear that this function can be computed in quasi-linear time.  In the worst case our approach potentially involves $O(N)$ calls to the greedy algorithm, whereas our solution to (a) uses only $O(\log N)$.  This limits the range of its applicability, but it is easy to parallelize the search, and this makes it feasible to compute $t_1(N)$ for $N$ as large as $10^9$ (but $N=10^{14}$ is surely out of reach).

To compute $t_1(N)$ we first use our solution to problem (a) to establish a lower bound on $t_1(N)\ge t_0(N)$.  To obtain an upper bound, we use the first phase of the greedy algorithm to compute the number $B_1$ of factors $m_ip_i$ divisible by primes $p_i\ge M\approx N^{1/2}$, along with the remaining factor $R=N!/\tuple$ expressed in terms of its valuation at primes $p < M$ via \eqref{eq:rval}.
We may then take $B_1 + \lfloor \log R / \log t\rfloor$ as an upper bound on the number of factors the greedy algorithm could produce in the best possible case. Note that increasing $t$ can only decrease this upper bound, so we can use a bisection search to find the least $t$ for which this upper bound is less than $N/3$, which is then a strict upper bound on $t_1(N)$.

Computing the lower and upper bounds on $t_1(N)$ involves only $O(\log N)$ calls to (the first phase of) the greedy algorithm and can be done quickly.  But the interval determined by these lower and upper bounds is typically large and appears to grow linearly with $N$.  We cannot apply a bisection search because there will typically be many $t$ in this interval for which the greedy algorithm produces more than $N$ factors that are interspersed with $t$ for which this is not the case.  Lacking a better alternative, we use an exhaustive search (which can easily be run in parallel on multiple cores) to find the largest such $t$ between our upper and lower bounds for which the greedy algorithm outputs at least $N/3$ factors, which gives us the value of $t_1(N)$.

\subsection{Proving \texorpdfstring{\Cref{main}}{Theorem 1.3}(iii) for \texorpdfstring{$\boldsymbol{67425 \le N \le 10^{14}}$}{67425 <= N <= 1e14}}
To establish \Cref{main}(iii) for $67425 \leq N \leq 10^{14}$ we may proceed as follows:
\begin{enumerate}
\item[Step 0:] Let $N=67425$.
\item[Step 1:] While $N\le 10^6$, compute $t_1(N)$ via problem (b) above using the standard greedy algorithm, verify that $t_1(N) > \lceil N/3\rceil$, and replace $N$ by $3t_1(N)$.
\item[Step 2:] While $N\le 10^{14}$, compute $t\in [t_0(N),t_1(N)]$ via (a) above using the fast variant of the greedy algorithm, verify that $t_1(N) > \lceil N/3\rceil$, and replace $N$ by $3t_1(N)$.
\end{enumerate}
The GitHub repository \cite{github} associated to this paper contains lists of the pairs $(N,t)$ that arise from the procedure above, 223 pairs for Step~1 and 336 pairs for Step~2, which can be used to quickly verify its success by invoking the greedy algorithm for each pair from Step~1, and the fast variant of the greedy algorithm for each pair from Step~2, and verifying in each case that a subfactorization of $N!$ with at least $N/3$ factors is produced.  The script \url{https://github.com/teorth/erdos-guy-selfridge/blob/main/src/fastegs/verifyhints.sh} performs this verification, which takes much less time (under a minute) than it does to run the procedure above.

It thus remains to establish \Cref{main}(iii) in the region $43632 \leq N < 67425$ and $N > 10^{14}$, and to show that $t(N)<N/3$ for $N=43631$.

\begin{table}[htp!]
  \centering
  \begin{tabular}{|c|c|c|c|c|r|}
  \hline
\rule{0pt}{11pt}  $N$ & $t(N)_-$ & Fast heuristic & Time (s) & Standard exhaustive & Time (s) \\
  \hline
\rule{0pt}{12pt}$1 \times 10^5$ & $\num{33642}$ & $t(N)_- - \num{458}$ & 0.002 & $t(N)_- - \num{70}$ & 0.013 \\
$2 \times 10^5$ & $\num{67703}$ & $t(N)_- - \num{1046}$ & 0.000 & $t(N)_- - \num{90}$ & 0.015 \\
$3 \times 10^5$ & $\num{101903}$ & $t(N)_- - \num{1495}$ & 0.000 & $t(N)_- - \num{54}$ & 0.023 \\
$4 \times 10^5$ & $\num{136143}$ & $t(N)_- - \num{2610}$ & 0.001 & $t(N)_- - \num{147}$ & 0.036 \\
$5 \times 10^5$ & $\num{170456}$ & $t(N)_- - \num{3091}$ & 0.001 & $t(N)_- - \num{51}$ & 0.050 \\
$6 \times 10^5$ & $\num{204811}$ & $t(N)_- - \num{2878}$ & 0.001 & $t(N)_- - \num{214}$ & 0.058 \\
$7 \times 10^5$ & $\num{239187}$ & $t(N)_- - \num{3834}$ & 0.001 & $t(N)_- - \num{279}$ & 0.061 \\
$8 \times 10^5$ & $\num{273604}$ & $t(N)_- - \num{4216}$ & 0.001 & $t(N)_- - \num{226}$ & 0.086 \\
$9 \times 10^5$ & $\num{308029}$ & $t(N)_- - \num{4444}$ & 0.001 & $t(N)_- - \num{226}$ & 0.172 \\\hline
\rule{0pt}{12pt}$1 \times 10^6$ & $\num{342505}$ & $t(N)_- - \num{4863}$ & 0.001 & $t(N)_- - \num{202}$ & 0.238 \\
$2 \times 10^6$ & $\num{687796}$ & $t(N)_- - \num{7850}$ & 0.002 & $t(N)_- - \num{293}$ & 0.408 \\
$3 \times 10^6$ & $\num{1033949}$ & $t(N)_- - \num{10395}$ & 0.003 & $t(N)_- - \num{564}$ & 0.385 \\
$4 \times 10^6$ & $\num{1380625}$ & $t(N)_- - \num{14637}$ & 0.005 & $t(N)_- - \num{705}$ & 0.550 \\
$5 \times 10^6$ & $\num{1727605}$ & $t(N)_- - \num{20837}$ & 0.005 & $t(N)_- - \num{470}$ & 1.394 \\
$6 \times 10^6$ & $\num{2074962}$ & $t(N)_- - \num{25872}$ & 0.006 & $t(N)_- - \num{1480}$ & 2.053 \\
$7 \times 10^6$ & $\num{2422486}$ & $t(N)_- - \num{31513}$ & 0.008 & $t(N)_- - \num{829}$ & 6.479 \\
$8 \times 10^6$ & $\num{2770212}$ & $t(N)_- - \num{31401}$ & 0.008 & $t(N)_- - \num{1183}$ & 3.978 \\
$9 \times 10^6$ & $\num{3118129}$ & $t(N)_- - \num{35468}$ & 0.007 & $t(N)_- - \num{1100}$ & 2.941 \\\hline
\rule{0pt}{12pt}$1 \times 10^7$ & $\num{3466235}$ & $t(N)_- - \num{43529}$ & 0.010 & $t(N)_- - \num{1222}$ & 3.144 \\
$2 \times 10^7$ & $\num{6952243}$ & $t(N)_- - \num{73103}$ & 0.014 & $t(N)_- - \num{2730}$ & 10.825 \\
$3 \times 10^7$ & $\num{10444441}$ & $t(N)_- - \num{110137}$ & 0.011 & $t(N)_- - \num{1653}$ & 39.501 \\
$4 \times 10^7$ & $\num{13940484}$ & $t(N)_- - \num{107106}$ & 0.018 & $t(N)_- - \num{1544}$ & 22.470 \\
$5 \times 10^7$ & $\num{17439282}$ & $t(N)_- - \num{186318}$ & 0.025 & $t(N)_- - \num{1911}$ & 79.011 \\
$6 \times 10^7$ & $\num{20940210}$ & $t(N)_- - \num{263788}$ & 0.031 & $t(N)_- - \num{1787}$ & 72.292 \\
$7 \times 10^7$ & $\num{24442818}$ & $t(N)_- - \num{286343}$ & 0.028 & $t(N)_- - \num{1047}$ & 273.384 \\
$8 \times 10^7$ & $\num{27946958}$ & $t(N)_- - \num{255063}$ & 0.031 & $t(N)_- - \num{3833}$ & 213.208 \\
$9 \times 10^7$ & $\num{31452431}$ & $t(N)_- - \num{335639}$ & 0.041 & $t(N)_- - \num{4121}$ & 823.168 \\\hline
\rule{0pt}{12pt}$1 \times 10^8$ & $\num{34958725}$ & $t(N)_- - \num{342699}$ & 0.027 & $t(N)_- - \num{2785}$ & 331.221 \\
$2 \times 10^8$ & $\num{70064782}$ & $t(N)_- - \num{738180}$ & 0.063 & $t(N)_- - \num{4800}$ & 4531.127 \\
$3 \times 10^8$ & $\num{105218403}$ & $t(N)_- - \num{956003}$ & 0.068 & $t(N)_- - \num{3502}$ & 2488.738 \\
$4 \times 10^8$ & $\num{140401212}$ & $t(N)_- - \num{1264714}$ & 0.073 & $t(N)_- - \num{6111}$ & 4852.155 \\
$5 \times 10^8$ & $\num{175605266}$ & $t(N)_- - \num{1645121}$ & 0.113 & $t(N)_- - \num{13029}$ & 12647.108 \\
$6 \times 10^8$ & $\num{210825848}$ & $t(N)_- - \num{1801197}$ & 0.149 & $t(N)_- - \num{7372}$ & 7154.594 \\
$7 \times 10^8$ & $\num{246059851}$ & $t(N)_- - \num{1925394}$ & 0.123 & $t(N)_- - \num{13808}$ & 12781.331 \\
$8 \times 10^8$ & $\num{281305291}$ & $t(N)_- - \num{2487332}$ & 0.147 & $t(N)_- - \num{17305}$ & 8573.188 \\
$9 \times 10^8$ & $\num{316560601}$ & $t(N)_- - \num{3137853}$ & 0.153 & $t(N)_- - \num{14555}$ & 23058.731 \\\hline
\end{tabular}
\bigskip

\caption{For sample values of $N \in [10^5, 10^9]$, the performance of the fast greedy algorithm (using heuristically chosen $t\in [t_0(N),t_1(N)]$), and the standard greedy algorithm (using  exhaustively computed $t_1(N)$), compared against the lower bound $t(N)_-$ obtained from the linear programming method of \Cref{linprog-sec}. All computations were performed on an Intel i9-13900KS CPU with 24 cores. The fast heuristic computations were single-threaded computations, while exhaustive greedy computations used 32 threads running on 24 cores.}\label{greedy-table}
\end{table}

\section{Linear programming}\label{linprog-sec}

It turns out that linear programming and integer programming methods are quite effective at bounding $M(N,t)$, both from above and below.  The starting point is the following integer program interpretation of $M(N,t)$. For any $t,N$, let $J_{t,N}$ be the collection of all $j \geq t$ that divide $N!$, and which do not have any proper factor $j' < j$ that is also greater than or equal to $t$.  For instance,
$$ J_{4,5} = \{ 4, 5, 6, 9 \}.$$

\begin{proposition}[Integer programming description of $t(N)$]\label{fip}  For any $N,t \geq 1$, $M(N,t)$ is the maximum value of
\begin{equation}\label{mj-sum}
   \sum_{j \in J_{t,N}} m_j
\end{equation}
where the $m_j$ are non-negative integers subject to the constraints
\begin{equation}\label{constraints}
  \sum_{j \in J_{t,N}} m_j \nu_p(j) \leq \nu_p(N!)
\end{equation}
for all primes $p \leq N$.
\end{proposition}

\begin{proof}  If $m_j, j \in J_{t,N}$ are non-negative integers obeying \eqref{constraints}, then clearly
  \begin{equation}\label{mj-subfac}
    \prod_{j \geq t} j^{m_j}
  \end{equation}
is a $t$-admissible subfactorization of $N!$, so that $M(N,t)$ is greater than or equal to \eqref{mj-sum}. Conversely, suppose that $M(N,t) \geq M$, thus we have a $t$-admissible subfactorization of $N!$ into $M$ factors.  Clearly, each of these factors $j$ is at least $t$, and divides $N!$.
If one of these factors $j$ has a proper factor $j' < j$ that is greater than or equal to $t$, then we can replace the factor $j$ by the factor $j'$ in the subfactorization, and still obtain a $t$-admissible subfactorization of $N!$.  Iterating this, we may assume without loss of generality that all the factors $t$ lie in $J_{t,N}$.  We can then express this subfactorization as a product \eqref{mj-subfac}, and by computing $p$-valuations we conclude the constraints \eqref{constraints}.  The claim follows.
\end{proof}

This integer program formulation can be used, when combined with standard packages such as Gurobi~\cite{gurobi} or \texttt{lp\_solve}~\cite{lpsolve}, to compute $M(N,t)$ (and hence $t(N)$) precisely for any specific $N,t$ with $N$ as large as $10^4$, though in practice it is better to first use faster methods (which we discuss below) to control these quantities first, using integer programming as a last resort when these faster methods fail to achieve the desired result.

For larger $N$, the sets $J_{t,N}$ become somewhat large, and the integer program becomes computationally expensive.  For the purposes of lower bounding $M(N,t)$, one can arbitrarily replace $J_{t,N}$ with a smaller set (effectively setting $m_j=0$ for all $j$ outside this set) to speed up the integer program; empirically we have found that the set $\{ j: t \leq j \leq N \}$ is a good choice, as it appears to give the same bounds while being significantly faster.

For upper bounds, we can relax the integer program to a linear program.  Let $M_\R(N,t)$ denote the maximum value of \eqref{mj-sum} where the $m_j, j \in J_{t,N}$ are now non-negative \emph{real} numbers obeying \eqref{constraints}.  Clearly we have the upper bound
$$ M(N,t) \leq M_\R(N,t)$$
which can be improved slightly to
\begin{equation}\label{lp-upper}
  M(N,t) \leq \lfloor M_\R(N,t)\rfloor
\end{equation}
since $M(N,t)$ is an integer.  We refer to these bounds as the \emph{linear programming upper bounds}.

The quantity $M_\R(N,t)$ can be computed by standard linear programming methods; in particular, upper bounds on $M_\R(N,t)$ can be obtained by solving a dual linear program involving some weights $w_p, p \leq N$ that obey constraints for each $j \in J_{t,N}$.  In fact we can restrict attention to those constraints with $j$ in the range $t \leq j \leq N$:

\begin{proposition}[Dual description of $M_\R(N,t)$]\label{dual-desc} For any $N,t \geq 1$ with $t \leq N/2$, $M_\R(N,t)$ is the minimum value of
\begin{equation}\label{hyp}
    \sum_{p \leq N} w_p \nu_p(N!)
\end{equation}
where $w_p$ are non-negative reals for primes $p \leq N$
subject to the constraints that the $w_p$ are weakly increasing, thus
\begin{equation}\label{wp-decrease}
w_{p_2} \geq w_{p_1}
\end{equation}
whenever $p_2 \geq p_1$, and
\begin{equation}\label{pj}
  \sum_{p \leq N} w_p \nu_p(j) \geq 1
 \end{equation}
for all $t \leq j \leq N$.  In particular, if
\begin{equation}\label{hyp-low}
  \sum_{p \leq N} w_p \nu_p(N!) < N
\end{equation}
then $t(N) < t$.
\end{proposition}

The requirement $t \leq N/2$ applies in practice, since by \eqref{trivial} we have $t(N) \leq N/2$ except possibly for the small cases $N \leq 5$.

\begin{proof}  Suppose first that $w_p$ are non-negative reals obeying \eqref{pj} for all $t \leq j \leq N$.
We claim that \eqref{pj} in fact holds for all $j \geq t$, not just for $t \leq j \leq N$.  Indeed, if this were not the case, consider the first $j > N$ where \eqref{pj} fails.  Take a prime $p$ dividing $j$ and replace it by a prime in the interval $[p/2,p)$ which exists by Bertrand's postulate (or remove $p$ entirely, if $p=2$); this creates a new $j'$ in $[j/2,j)$ which is still at least $t$.  By the weakly increasing hypothesis on $w_p$, we have
  $$ \sum_p w_p \nu_p(j) \geq \sum_p w_p \nu_p(j')$$
  and hence by the minimality of $j$ we have
  $$ \sum_p w_p \nu_p(j) > 1,$$
  a contradiction.

Now let $m_j, j \in J_{t,N}$ be non-negative reals obeying \eqref{constraints}.  Multiplying each constraint in \eqref{constraints} by $w_p$ and summing, we conclude from \eqref{pj} that
$$ \sum_{j \in J_{t,N}} m_j  \leq \sum_{j \in J_{t,N}} m_j \sum_{p \leq N} w_p \nu_p(j) \leq \sum_{p \leq N} w_p \nu_p(N!)$$
and hence \eqref{hyp} is an upper bound for $M_\R(N,t)$.

In the opposite direction, we need to locate weakly increasing non-negative weights $w_p$ obeying \eqref{pj} for $t \leq j \leq N$ for which
\begin{equation}\label{wpm}
  \sum_{p \leq N} w_p \nu_p(N!) = M_\R(N,t).
\end{equation}
To do this, we first make the technical observation that in the definition of $M_\R(N,t)$, we can enlarge the index set $J_{t,N}$ to the larger set $J'_{t,N}$ of natural numbers $j \geq t$ that divide $N!$.  This follows by repeating the proof of \Cref{fip}: if $m_j$ were non-zero for some $j \geq t$ dividing $N!$ that had a proper factor $j' \geq t$, then one could transfer the mass of $m_j$ to $m_{j'}$ (i.e., replace $m_{j'}$ with $m_{j'}+m_j$ and then set $m_j$ to zero) without affecting \eqref{constraints}.

If we then invoke the duality theorem of linear programming, we can find weights $w_p \geq 0$ for $p \leq N$ obeying \eqref{wpm} as well as \eqref{pj} for all $j \in J'_{t,N}$ (not just $j \in J_{t,N}$).  To conclude the proof, it suffices to show that the $w_p$ are weakly increasing. Suppose for contradiction that there are primes $p_1 < p_2 \leq N$ such that $w_{p_1} > w_{p_2}$.  Let $\eps>0$ be a sufficiently small quantity, and define the modification $\tilde w_p$ to $w_p$ by decreasing $w_{p_1}$ by $\eps$ and leaving all other $w_p$ unchanged.  This decreases the left-hand side of \eqref{wpm}, so to get a contradiction with the already-obtained lower bound, it suffices to show that
  $$ \sum_p \tilde w_p \nu_p(j) \geq 1$$
  for all $j \geq t$ dividing $N!$.  If $j$ has a proper factor $j'$ that is still at least $t$, the condition for $j$ would follow from that of $j'$, so we may restrict attention to the case where $j$ has no proper factor greater than or equal to $t$.  We can assume that $j$ is divisible by $p_1$, otherwise the claim follows from \eqref{pj}.  For $\eps$ small enough, one has
  $$ \sum_p \tilde w_p \nu_p(j) \geq \sum_p w_p \nu_p\left(\frac{p_2}{p_1} j\right);$$
since $\frac{p_2}{p_1} j \geq j \geq t$, we are done unless $\frac{p_2}{p_1} j$ is not divisible by $N!$. This only occurs when $\nu_{p_2}(j) = \nu_{p_2}(N!)$, but then
$$ \frac{j}{p_1} \geq p_2^{\nu_{p_2}(N!)} \geq p_2^{\lfloor N/p_2\rfloor}.$$
We claim that the right-hand side is at least $N/2$.  This is clear for $p_2 \geq N/2$, and also for $\sqrt{N/2} \leq p_2 < N/2$ since $\lfloor N/p_2\rfloor \geq 2$ in this case.  For $p_2 < \sqrt{N/2}$ one has
$$  p_2^{\lfloor N/p_2\rfloor} \geq 3^{\lfloor \sqrt{2N} \rfloor} \geq \frac{N}{2}$$
for all $N$ (here we use that $3^k \geq \frac{(k+1)^2}{4}$ for $k \geq 1$).  Thus in all cases we have $j/p_1 \geq N/2 \geq t$, contradicting the hypothesis that $j$ has no proper factor that is at least $t$.
\end{proof}

\Cref{dual-desc} allows for a fast method to compute $M_\R(N,t)$ by a linear program.  In practice, we have found that even if we drop the explicit constraint \eqref{wp-decrease} that the $w_p$ are weakly decreasing (or equivalently, if we return to the primal problem of optimizing \eqref{mj-sum} for real $m_j \geq 0$ obeying \eqref{constraints}, but now with $j$ restricted to $t \leq j \leq N$), the optimal weights $w_p$ produced by the resulting linear program will be weakly decreasing anyway, although we could not prove this empirically observed fact rigorously.  For instance, when $N = 3 \times 10^5$ and $t = N/3$, this linear program produces non-decreasing weights which certify that
$$ M_\R(N,t) = N + 445.83398\dots$$
and hence by \eqref{lp-upper}
\begin{equation}\label{mnt}
  M(N,t) \leq N + 445
\end{equation}
for this choice of $N,t$.  In fact, as discussed later in this section, we know that equality holds in this particular case.  For $N \leq 10^4$, we found that the linear programming upper bound on $t(N)$ is tight except for
$N=155$, $765$, $1528$, $1618$, $1619$, $2574$, $2935$, $3265$, $5122$, $5680$, and $9633$, for which integer programming was needed to precisely compute $t(N)$.  The values of $t(N)$ thus computed are  plotted in \Cref{fig-long}.

The linear programming upper bound is sufficiently tight to establish that $t(N)<N/3$ for $N=43631$, which proves that the threshold $43632$ of \Cref{main}(iii) is best possible.
The dual certificate for this computation\footnote{\url{https://github.com/teorth/erdos-guy-selfridge/blob/main/src/python/verification/prove43631.py}} was verified in exact arithmetic.
The exact bound obtained is $M(43631,14544) \le 43631 - \frac{47}{1257}$.

\begin{figure}[!ht]
  \centering
  \includegraphics[width=0.8\textwidth]{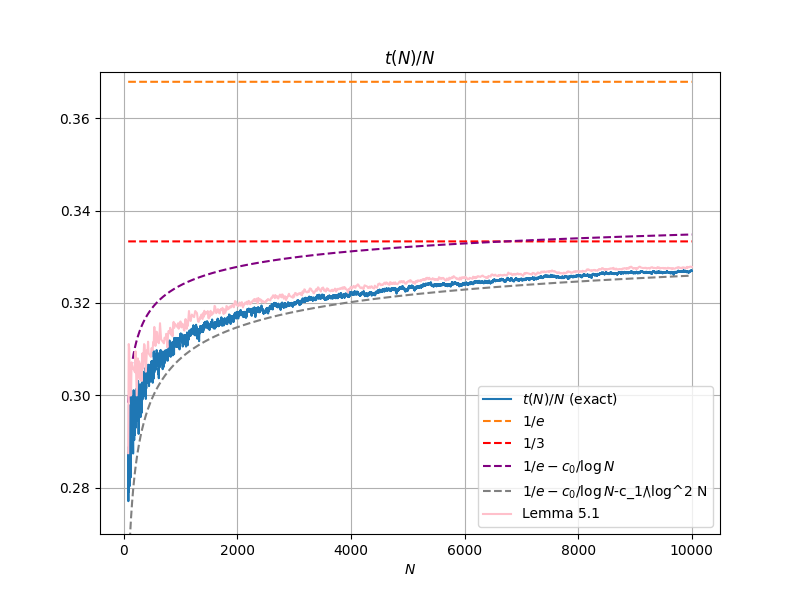}
  \vspace{-8pt}
  \caption{Exact values of $t(N)/N$ for $80 \leq N \leq 10^4$, obtained via integer programming.  The upper bound from \Cref{upper-crit} is surprisingly sharp, as is the refined asymptotic $1/e - c_0/\log N - c_1/\log^2 N$, though the cruder asymptotics $1/e$ or $1/e - c_0/\log N$ are significantly poorer approximations.}
  \label{fig-long}
  \end{figure}

With integer programming, we could also establish\footnote{Explicit factorizations in this range can be found at \url{https://github.com/teorth/erdos-guy-selfridge/tree/main/Data/factorizations}.} $t(N) \geq N/3$ for all $43632 \leq N \leq 8 \times 10^4$.
  In particular, when combined with the greedy algorithm computations from the previous section, this resolves \Cref{main}(iii) except in the asymptotic range $N > 10^{14}$, where it suffices to establish the lower bound $t(N) \geq N/3$.

For $N \geq 10^4$, the integer programming method to lower bound $M(N,t)$ becomes slow.  We found two faster methods to give slightly weaker lower bounds on this quantity, which we call the ``floor+residuals'' method, and the ``smooth factorization'' method.

The ``floor+residuals'' method proceeds by first running the primal linear program to find the real $m_j \geq 0$ for $t \leq j \leq N$ that maximize \eqref{mj-sum} subject to \eqref{constraints}.  The integer parts\footnote{This procedure turns out to be subtly dependent on the specific linear programming implementation, both due to roundoff errors, and also because the extremizer of the linear program can be non-unique, with different LP solvers arriving at different extremizers.} $\lfloor m_j \rfloor$ will then of course also obey \eqref{constraints} and thus form a subfactorization; but this subfactorization is somewhat inefficient because there can be a $p$-surplus of $\nu_p(N!) - \sum_{j \geq t} \lfloor m_j \rfloor \nu_p(j)$ at various primes $p \leq N$.  We then apply the greedy algorithm of the previous section to fashion as many factors greater than or equal to $t$ from these residual primes, to obtain our final subfactorization that provides a lower bound on $M(N,t)$.

The floor+residuals method is fast and highly accurate for small and medium $N$ (e.g., $N \leq 3 \times 10^5$).  For instance:
\begin{itemize}
  \item The method computes $t(N)$ exactly for all $N \leq 600$, with the sole exception of $N=155$; see \Cref{fig-zoom}.  When $N=155$, the floor+residuals method provides a subfactorization that certifies $t(155) \geq 45$, while the linear programming upper bound \eqref{lp-upper} gives $t(155) \leq 46$.  Integer programming can then be deployed to confirm $t(155)=45$.
  \item The method establishes $t(N) \geq N/3$ for all $43632 \leq N \leq 4.5 \times 10^4$; see \Cref{fig-surplus}.
  \item The method also verifies the lower bound $t(N)\geq N/3$ for $N=41006$, while the linear programming upper bound \eqref{lp-upper} shows that $t(N) \geq N/3$ fails for all smaller $N$ except for $N=1,2,3,4,5,6,9$; see \Cref{fig-surplus}.
  \item The method establishes the matching lower bound
\begin{equation}\label{mnt-lower}
M(N,t) \geq N + 455
\end{equation}
to \eqref{mnt} when $N = 3 \times 10^5$ and $t=N/3$.
\end{itemize}

For larger $N$ (e.g., $3 \times 10^5 \leq N \leq 9 \times 10^8$), the floor+residuals method becomes slow due to the large number $\pi(N)$ of variables $w_p$ that are involved in the linear program.  We developed a \emph{smooth factorization lower bound} method\footnote{See \url{https://github.com/teorth/erdos-guy-selfridge/tree/main/src/mojo} for details.} to handle this range, by first using the greedy approach from the previous section to allocate all factors involving $p \geq \sqrt{N}$, and then using a version of the floor+residuals method to handle the smaller primes $p < \sqrt{N}$ (with $j$ now restricted to ``smooth'' numbers - numbers whose prime factors are less than $\sqrt{N}$).  Thus, the linear program now involves only $\pi(\sqrt{N})$ variables $w_p$, and runs considerably faster in ranges such as $10^4 < N \leq 9 \times 10^8$.  The lower bounds obtained by this method remain quite close to the linear programming upper bound \eqref{lp-upper} (or the floor+residuals method), and outperforms the greedy algorithm; see Table \ref{long-table} and \Cref{fig-longer}.

  \begin{figure}
    \centering
    \includegraphics[width=0.75\textwidth]{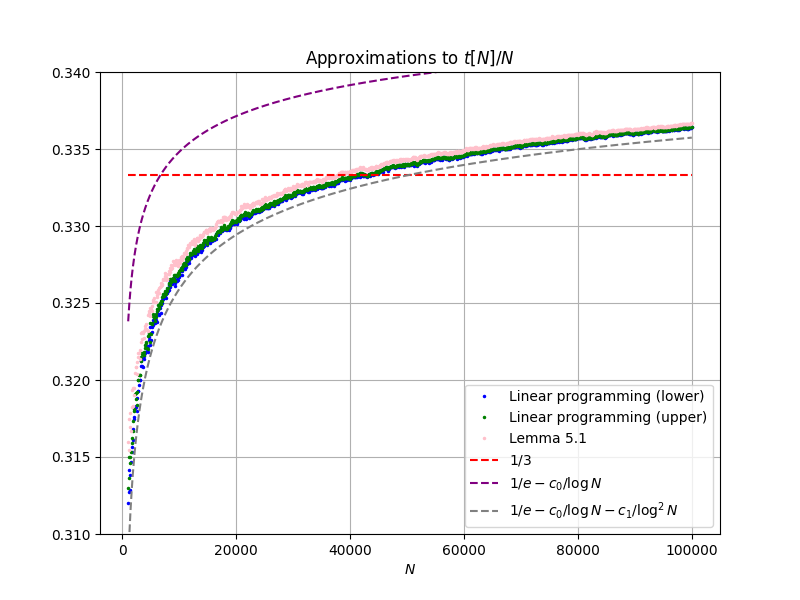}
    \vspace{-12pt}
    \caption{The linear programming upper bound and smooth factorization lower bound on $t(N)/N$ for $10^3 \leq N \leq 10^5$ that are multiples of $100$. The refined asymptotic $1/e - c_0/\log N - c_1/\log^2 N$ is now a slight underestimate, hinting at further terms in the asymptotic expansion. For another view of the situation near the crossover point $N=43632$ for \Cref{main}(iii), see \Cref{fig-surplus}.}
    \label{fig-longer}
    \end{figure}

    \begin{figure}
      \centering
      \includegraphics[width=0.8\textwidth]{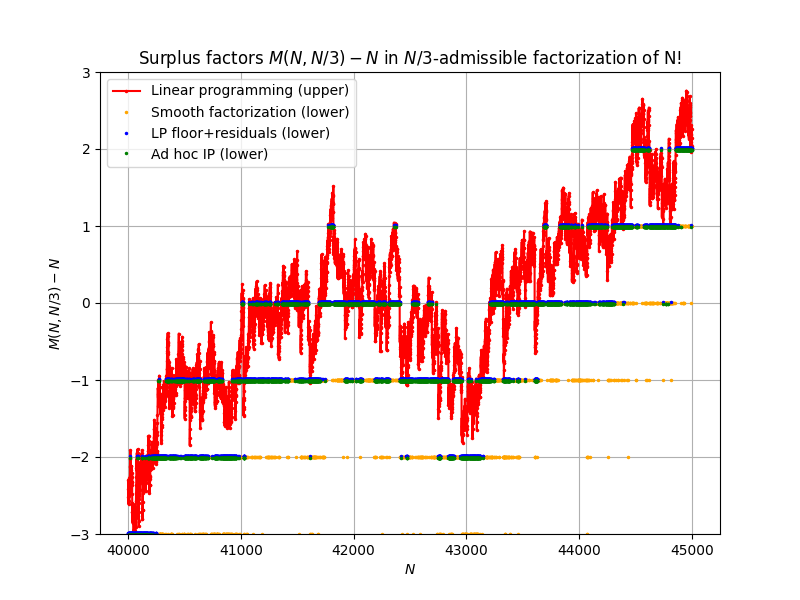}
      \vspace{-12pt}
      \caption{Bounds on $M(N,N/3)-N$ for $4 \times 10^4 \leq N \leq 4.5 \times 10^4$.  The linear programming upper bound (red) can be rounded down to the nearest integer, as per \eqref{lp-upper}.  In this range, at least one of the integer programming (green, implemented in an \emph{ad hoc} fashion) and the floor+residuals (blue) methods turn out to match this bound exactly, with the smooth factorization method (orange) not far behind.}
      \label{fig-surplus}
      \end{figure}

      \begin{figure}
        \centering
        \includegraphics[width=0.8\textwidth]{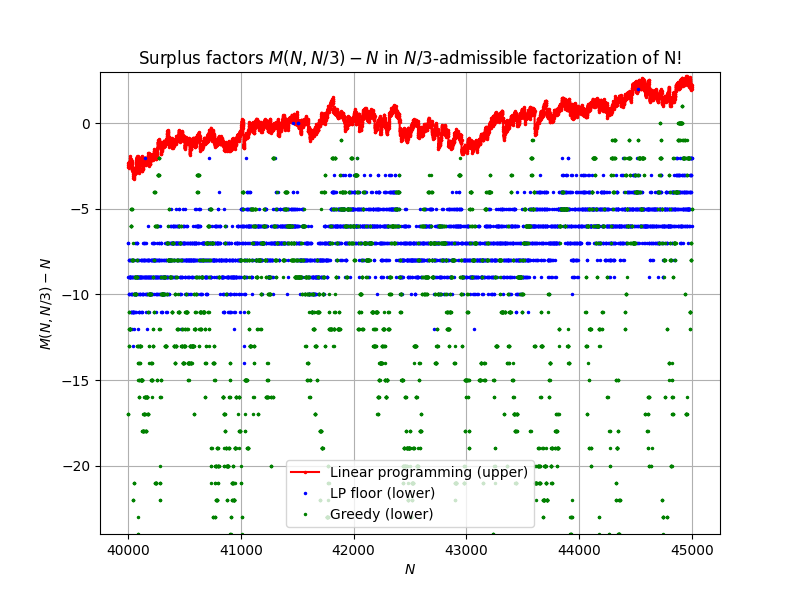}
        \vspace{-16pt}
        \caption{Two weaker lower bounds on $M(N,N/3)-N$ for $4 \times 10^4 \leq N \leq 4.5 \times 10^4$ beyond those depicted in \Cref{fig-surplus}: the floor of a linear programming method (blue) without gathering residuals, and the greedy method (green), which performs erratically but usually gives worse bounds than any of the other methods in this range.}
        \label{fig-surplus2}
        \end{figure}

\begin{figure}
  \centering
  \includegraphics[width=0.8\textwidth]{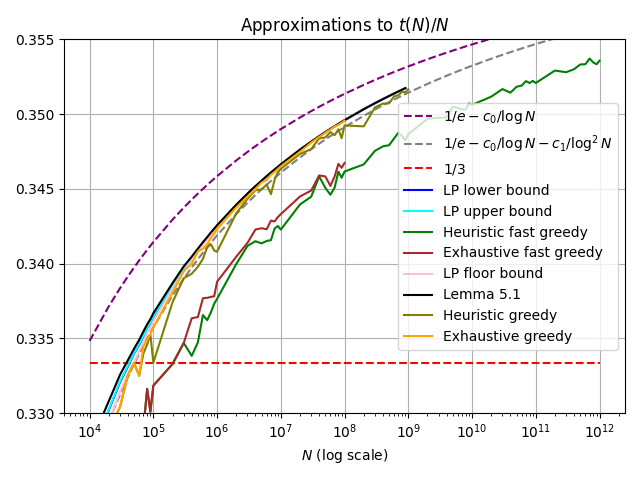}
  \vspace{-16pt}
  \caption{Lower bounds on $t(N)/N$ coming from various linear programming and greedy methods for sample values of $N$ in the range $10^4 \leq N \leq 10^{12}$, as well as upper bounds coming from linear programming and \Cref{upper-crit}.  Even if one simply takes floors from the linear program and discards residuals, the resulting lower bound is asymptotically almost indistinguishable from the upper bound. }\label{fig-verylarge}
\end{figure}

    \begin{table}[ht]
      \centering
      \begin{tabular}{|c|c|c|c|c|}
      \hline
\rule{0pt}{14pt} $N$ & $t(N)_-$ & $t(N)_+$ & \Cref{upper-crit} & $\frac{N}{e}-\frac{c_0 N}{\log N} - \frac{c_1 N}{\log^2 N}$  \\[5pt]
      \hline
\rule{0pt}{12pt}$1 \times 10^5$ & $\num{33642}$ & $t(N)_- + 4$ & $t(N)_- + 26$ & $t(N)_- - \num{69}$ \\
$2 \times 10^5$ & $\num{67703}$ & $t(N)_- + 1$ & $t(N)_- + 36$ & $t(N)_- - \num{130}$ \\
$3 \times 10^5$ & $\num{101903}$ & $t(N)_- + 3$ & $t(N)_- + 42$ & $t(N)_- - \num{206}$ \\
$4 \times 10^5$ & $\num{136143}$ & $t(N)_- + 6$ & $t(N)_- + 43$ & $t(N)_- - \num{248}$ \\
$5 \times 10^5$ & $\num{170456}$ & $t(N)_- + 3$ & $t(N)_- + 46$ & $t(N)_- - \num{310}$ \\
$6 \times 10^5$ & $\num{204811}$ & $t(N)_- + 4$ & $t(N)_- + 47$ & $t(N)_- - \num{373}$ \\
$7 \times 10^5$ & $\num{239187}$ & $t(N)_- + 9$ & $t(N)_- + 54$ & $t(N)_- - \num{425}$ \\
$8 \times 10^5$ & $\num{273604}$ & $t(N)_- + 6$ & $t(N)_- + 64$ & $t(N)_- - \num{490}$ \\
$9 \times 10^5$ & $\num{308029}$ & $t(N)_- + 13$ & $t(N)_- + 70$ & $t(N)_- - \num{539}$ \\
\hline
\rule{0pt}{12pt}$1 \times 10^6$ & $\num{342505}$ & $t(N)_- + 3$ & $t(N)_- + 62$ & $t(N)_- - \num{619}$ \\
$2 \times 10^6$ & $\num{687796}$ & $t(N)_- + 4$ & $t(N)_- + 87$ & $t(N)_- - \num{1180}$ \\
$3 \times 10^6$ & $\num{1033949}$ & $t(N)_- + 11$ & $t(N)_- + 107$ & $t(N)_- - \num{1736}$ \\
$4 \times 10^6$ & $\num{1380625}$ & $t(N)_- + 12$ & $t(N)_- + 122$ & $t(N)_- - \num{2286}$ \\
$5 \times 10^6$ & $\num{1727605}$ & $t(N)_- + 4$ & $t(N)_- + 126$ & $t(N)_- - \num{2763}$ \\
$6 \times 10^6$ & $\num{2074962}$ & $t(N)_- + 21$ & $t(N)_- + 152$ & $t(N)_- - \num{3326}$ \\
$7 \times 10^6$ & $\num{2422486}$ & $t(N)_- + 22$ & $t(N)_- + 165$ & $t(N)_- - \num{3819}$ \\
$8 \times 10^6$ & $\num{2770212}$ & $t(N)_- + 29$ & $t(N)_- + 177$ & $t(N)_- - \num{4316}$ \\
$9 \times 10^6$ & $\num{3118129}$ & $t(N)_- + 24$ & $t(N)_- + 173$ & $t(N)_- - \num{4834}$ \\
\hline
\rule{0pt}{12pt}$1 \times 10^7$ & $\num{3466235}$ & $t(N)_- + 12$ & $t(N)_- + 179$ & $t(N)_- - \num{5392}$ \\
$2 \times 10^7$ & $\num{6952243}$ & $t(N)_- + 18$ & $t(N)_- + 234$ & $t(N)_- - \num{10284}$ \\
$3 \times 10^7$ & $\num{10444441}$ & $t(N)_- + 13$ & $t(N)_- + 253$ & $t(N)_- - \num{14975}$ \\
$4 \times 10^7$ & $\num{13940484}$ & $t(N)_- + 64$ & $t(N)_- + 354$ & $t(N)_- - \num{19582}$ \\
$5 \times 10^7$ & $\num{17439282}$ & $t(N)_- + 33$ & $t(N)_- + 356$ & $t(N)_- - \num{24124}$ \\
$6 \times 10^7$ & $\num{20940210}$ & $t(N)_- + 23$ & $t(N)_- + 381$ & $t(N)_- - \num{28610}$ \\
$7 \times 10^7$ & $\num{24442818}$ & $t(N)_- + 37$ & $t(N)_- + 415$ & $t(N)_- - \num{32996}$ \\
$8 \times 10^7$ & $\num{27946958}$ & $t(N)_- + 43$ & $t(N)_- + 445$ & $t(N)_- - \num{37417}$ \\
$9 \times 10^7$ & $\num{31452431}$ & $t(N)_- + 23$ & $t(N)_- + 428$ & $t(N)_- - \num{41882}$ \\
\hline
\rule{0pt}{12pt}$1 \times 10^8$ & $\num{34958725}$ & $t(N)_- + 48$ & $t(N)_- + 482$ & $t(N)_- - \num{46039}$ \\
$2 \times 10^8$ & $\num{70064782}$ & $t(N)_- + 45$ & $t(N)_- + 644$ & $t(N)_- - \num{87837}$ \\
$3 \times 10^8$ & $\num{105218403}$ & $t(N)_- + 41$ & $t(N)_- + 752$ & $t(N)_- - \num{128227}$ \\
$4 \times 10^8$ & $\num{140401212}$ & $t(N)_- + 80$ & $t(N)_- + 887$ & $t(N)_- - \num{167495}$ \\
$5 \times 10^8$ & $\num{175605266}$ & $t(N)_- + 98$ & $t(N)_- + 972$ & $t(N)_- - \num{206175}$ \\
$6 \times 10^8$ & $\num{210825848}$ & $t(N)_- + 68$ & $t(N)_- + 1058$ & $t(N)_- - \num{244391}$ \\
$7 \times 10^8$ & $\num{246059851}$ & $t(N)_- + 89$ & $t(N)_- + 1147$ & $t(N)_- - \num{282167}$ \\
$8 \times 10^8$ & $\num{281305291}$ & $t(N)_- + 92$ & $t(N)_- + 1158$ & $t(N)_- - \num{319607}$ \\
$9 \times 10^8$ & $\num{316560601}$ & $t(N)_- + 101$ & $t(N)_- + 1238$ & $t(N)_- - \num{356927}$ \\
\hline
\end{tabular}
\bigskip

\caption{For sample values of $N \in [10^5, 9 \times 10^8]$, the (remarkably precise) lower and upper bounds $t(N)_- \leq t(N) \leq t(N)_+$ obtained by smooth factorization and linear programming respectively, the (slightly weaker) upper bound on $t(N)$ from \Cref{upper-crit}, and the conjectural approximation $\frac{N}{e} - \frac{c_0 N}{\log N} - \frac{c_1N}{\log^2 N}$ (rounded to the nearest integer). }\label{long-table}
\end{table}

\FloatBarrier

\section{Some upper bounds}\label{upper-sec}

It is easy to check using \eqref{ftoa} that the weights $w_p \coloneqq \frac{\log p}{\log t}$ will obey the conditions \eqref{hyp-low}, \eqref{pj} as long as $0 > \log N! - N \log t$.  This recovers the trivial upper bound \eqref{trivial}.  By adjusting these weights at large primes, one can improve this bound as follows:

\begin{lemma}[Upper bound criterion]\label{upper-crit}  Suppose that $1 \leq t \leq N$ are such that
  \begin{equation}\label{contra}
     \sum_{\frac{t}{\lfloor\sqrt{t}\rfloor} < p \leq N} f_{N/t}(p/N) > \log N! - N \log t,
  \end{equation}
  where $f_{N/t}$ was defined in \eqref{falpha-def}.
  Then $t(N) < t$.
  \end{lemma}

\begin{proof} We introduce the weights
  $$
w_p \coloneqq  \begin{cases}
  \frac{\log p}{\log t}, & p  \leq \frac{t}{\lfloor \sqrt{t} \rfloor}, \\
  \frac{\log p}{\log t} - \frac{\log \frac{\lceil t/p \rceil}{t/p}}{\log t} = 1 - \frac{\log \lceil t/p \rceil}{\log t}, & p  > \frac{t}{\lfloor \sqrt{t} \rfloor}.
\end{cases}
$$
Clearly the $w_p$ are non-negative.  It will suffice to verify the conditions \eqref{pj}, \eqref{hyp-low}.  If $j \in J_{t,N}$ contains no prime factor $p > \frac{t}{\lfloor \sqrt{t} \rfloor}$, then from \eqref{ftoa} we have
$$ \sum_p w_p \nu_p(j) = \frac{\sum_p \nu_p(j) \log p}{\log t} = \frac{\log j}{\log t} \geq 1.$$
If $j \in J_{t,N}$ is of the form $j = mp_1$ where $p_1 > \frac{t}{\lfloor \sqrt{t} \rfloor}$ and $m$ contains no prime factor exceeding $\frac{t}{\lfloor \sqrt{t} \rfloor}$, then $m \geq \lceil t/p_1 \rceil$, and we have
\begin{align*}
\sum_p w_p \nu_p(j) &= \frac{\sum_p \nu_p(j) \log p}{\log t}
- \frac{\log \frac{\lceil t/p_1 \rceil}{t/p_1}}{\log t}\\
&= \frac{\log(mp_1)}{\log t} -  \frac{\log \frac{m}{t/p_1}}{\log t} = 1.
\end{align*}
Finally, if $j \in J_{t,N}$ is divisible by two primes $p_1, p_2 > {t}{\lfloor \sqrt{t} \rfloor}$ (possibly equal), then
\begin{align*}
  \sum_p w_p \nu_p(j) &\geq
  1 - \frac{\log \lceil t/p_1 \rceil}{\log t} + 1 - \frac{\log \lceil t/p_1 \rceil}{\log t} \\
  &\geq
  1 - \frac{\log \sqrt{t}}{\log t} + 1 - \frac{\log \sqrt{t}}{\log t} = 1.
\end{align*}
Thus we have verified \eqref{pj} for all $j \in J_{t,N}$.  Finally, from \eqref{ftoa}, \eqref{legendre}, \eqref{contra} we have
\begin{align*}
  \sum_p w_p \nu_p(N!) &= \frac{\sum_p \nu_p(N!) \log p}{\log t} - \sum_{p > \frac{t}{\lfloor \sqrt{t} \rfloor}} \frac{\nu_p(N!) \log \frac{\lceil t/p \rceil}{t/p}}{\log t} \\
  &\leq \frac{\log N!}{\log t} -  \frac{\sum_{p > \frac{t}{\lfloor \sqrt{t} \rfloor}} \lfloor \frac{N}{p} \rfloor \log \frac{\lceil t/p \rceil}{t/p}}{\log t}\\
  &= \frac{\log N!}{\log t} - \frac{\sum_{p > \frac{t}{\lfloor \sqrt{t} \rfloor}} f_{N/t}(p/N)}{\log t} \\
  &< N,
\end{align*}
giving \eqref{hyp}.  The claim follows.
\end{proof}

In practice, \Cref{upper-crit} gives quite good upper bounds on $N$, especially when $N$ is large, although for medium $N$ the linear programming method is superior: see \Cref{fig1}, \Cref{fig1-alt}, and \Cref{fig-zoom}.

\begin{figure}
  \centering
  \includegraphics[width=0.8\textwidth]{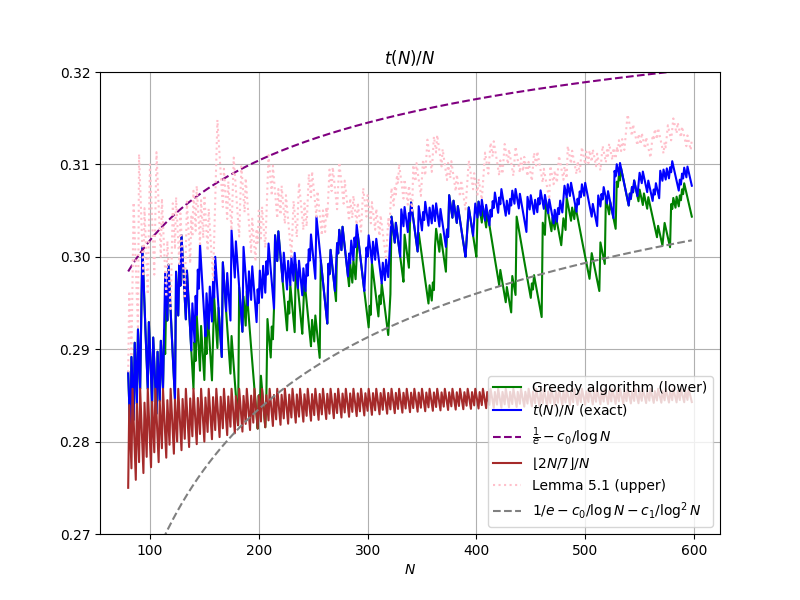}
  \vspace{-8pt}
  \caption{An enlarged version of \Cref{fig1-alt}, displaying the lower bound from the greedy algorithm and the upper bound from \Cref{upper-crit}.  The linear programming upper bound and floor+residual bounds are exact in this region, except for $N=155$ in which the upper bound is off by one.}\label{fig-zoom}
\end{figure}

\begin{figure}
  \centering
  \includegraphics[width=0.8\textwidth]{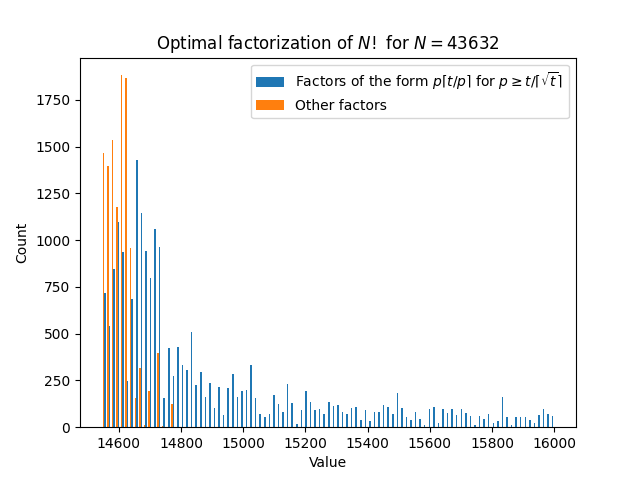}
  \caption{A histogram of an optimal factorization of $N!$ for $N=43632$, demonstrating that $t(N) = N/3+1 = 14545$.  Of the $N$ factors, most ($\num{32147}$) of them are of the form $p \lceil t/p \rceil$ for $p > \frac{t}{\lfloor \sqrt{t} \rfloor}$, thus directly contributing to the left-hand side of \Cref{upper-crit}.  (The factors arising from very large primes lie to the right of the displayed graph, as a long tail of the histogram.)  The remaining $\num{11485}$ factors stay close to $t(N)$, with the largest being $\num{14941}$.
  }\label{fig-factor}
  \end{figure}

We can now prove the upper bound portion of \Cref{main}(iv):

  \begin{proposition}\label{upper-bound}  For large $N$, one has
    $$ \frac{t(N)}{N} \leq \frac{1}{e} - \frac{c_0}{\log N} - \frac{c_1 + o(1)}{\log^2 N}$$
    where $c_0$ is given by \eqref{c0-def} and $c_1$ is given by
    \begin{align}
      c'_1 &\coloneqq \frac{1}{e} \int_0^1 f_e(x) \log \frac{1}{x}\ dx = 0.3702015\dots \label{c1p-def}\\
      c''_1 &\coloneqq \sum_{k=1}^\infty \frac{1}{k}  \log\left( \frac{e}{k} \left\lceil \frac{k}{e} \right\rceil \right) = 1.679578996\dots\label{c1pp-def}\\
      c_1 &\coloneqq c'_1 + c_0 c''_1 - ec_0^2/2 =0.75554808\dots.\label{c1-def}
    \end{align}
  \end{proposition}

We discuss the numerical evaluation of these constants in \Cref{c0-app}.

  Numerically, this bound is a reasonably good approximation for medium-sized $N$, see \Cref{fig-long}, \Cref{fig-longer}, although it may be possible to improve the approximation further with additional terms.  Based on these numerics it seems natural to conjecture that one in fact has
  $$\frac{t(N)}{N} = \frac{1}{e} - \frac{c_0}{\log N} - \frac{c_1 + o(1)}{\log^2 N}$$
as $N \to \infty$.

    \begin{proof}[Proof of \Cref{upper-bound}]
    We apply \Cref{upper-crit} with
      $$ t \coloneqq \frac{1}{e} - \frac{c_0}{\log N} - \frac{c_1-\eps}{\log^2 N}$$
    for a given small constant $\eps>0$.  From the Taylor expansion of the logarithm and the Stirling approximation \eqref{stirling} one sees that
    $$ \log N! - N \log t = ec_0 \frac{N}{\log N} + (ec_1 - \frac{1}{2} e^2 c_0^2 - e\eps+o(1)) \frac{N}{\log^2 N}$$
    so it will suffice to establish the lower bound
    \begin{equation}\label{targ-sum}
      \sum_{\frac{t}{\lfloor\sqrt{t}\rfloor} < p \leq N} f_{N/t}(p/N) \geq ec_0 \frac{N}{\log N} + (ec_1 - \frac{1}{2} e^2 c_0^2 - e \eps + o(1)) \frac{N}{\log^2 N}
    \end{equation}
    for $N$ sufficiently large depending on $\eps$.

    For $N$ large enough, we have $\frac{t}{\lfloor\sqrt{t}\rfloor} \leq \frac{N}{\log^3 N}$.
    On the interval $[1/\log^3 N,1]$, the piecewise smooth function $f_{N/t}$ is bounded by $O(1)$ thanks to \eqref{falpha-bound}, and has an (augmented) total variation of $O(\log^3\! N)$; the same is true of the rescaled function $x \mapsto f_{N/t}(x/N)$ on $[N/\log^3\! N,N]$.   This implies that $x \mapsto \frac{1}{\log x} f_{N/t}(x/N)$ has an (augmented) total variation of $O(\log^2 N)$.
   By \Cref{osc-lemma} (with classical error term), we conclude that the left-hand side of \eqref{targ-sum} is at least
    $$ \int_{N/\log^3 N}^N f_{N/t}(x/N) \frac{dx}{\log x} + O\left( N \exp(-c\sqrt{\log N}) \right)$$
    for some $c>0$.  Performing a change of variable, it suffices to show that
    $$  \int_{1/\log^3 N}^1 f_{N/t}(x) \frac{\log N}{\log(Nx)}\ dx
    \geq  ec_0 + \frac{ec_1 - \frac{1}{2} e^2 c_0^2 - e \eps + o(1)}{\log N}.$$
    By Taylor expansion, we have
    $$ \frac{\log N}{\log(Nx)} = 1 + \frac{\log \frac{1}{x}}{\log N}  + o\left(\frac{1}{\log N}\right)$$
    and from dominated convergence we have
  $$
  \int_{1/\log^3 N}^1 f_{N/t}(x) \log \frac{1}{x}\ dx = ec'_1 + o(1)$$
and hence by definition of $c_1$, it suffices to show that
$$  \int_{1/\log^3 N}^1 f_{N/t}(x)\ dx
\geq  ec_0 + \frac{ec_0 c''_1 - e^2 c_0^2 - e\eps + o(1)}{\log N}.$$
By performing a rescaling by $N/et = 1 + \frac{ec_0+o(1)}{\log N}$, the left-hand side may be written as
$$ \left(1 - \frac{ec_0+o(1)}{\log N}\right) \int_{N/et\log^3 N}^{N/et}
\left\lfloor \frac{N/et}{x} \right\rfloor \log\left(ex \left\lceil \frac{1}{ex} \right\rceil \right)\ dx$$
so it will suffice to show that
$$
\int_{N/et\log^3 N}^{N/et}
\left\lfloor \frac{N/et}{x} \right\rfloor \log\left(ex \left\lceil \frac{1}{ex} \right\rceil \right)\ dx \geq ec_0 + \frac{ec_0 c''_1 - e\eps + o(1)}{\log N}.$$
From \eqref{c0-def}, \eqref{falpha-bound} we have
$$ \int_{1/\log^2 N}^1 \left\lfloor \frac{1}{x} \right\rfloor \log\left(ex \left\lceil \frac{1}{ex} \right\rceil \right) = ec_0 - \frac{o(1)}{\log N},$$
so it suffices to show that
$$
\int_{1/log^2 N}^{N/et}
\left(\left\lfloor \frac{N/et}{x} \right\rfloor - \left\lfloor \frac{1}{x} \right\rfloor\right) \log\left(ex \left\lceil \frac{1}{ex} \right\rceil \right)\ dx \geq \frac{ec_0 c'' - e\eps + o(1)}{\log N}.$$
Let $K$ be sufficiently large depending on $\eps$, then for $N$ sufficiently large depending on $K$ we can lower bound the left-hand side by
$$ \sum_{k=1}^K \int_{1/k}^{N/etk} \log\left(ex \left\lceil \frac{1}{ex} \right\rceil \right)\ dx;$$
since $\frac{N}{etk} = \frac{1}{k} + \frac{ec_0}{k \log N}$, we can lower bound this (using the irrationality of $e$) by
$$ \frac{ec_0+o(1)}{\log N} \sum_{k=1}^K \frac{1}{k} \log\left(\frac{e}{k} \left\lceil \frac{k}{e} \right\rceil\right)$$
for sufficiently large $N$.  Since the sum here can be made arbitrarily close to $c''_0$ by increasing $K$, we obtain the claim.
\end{proof}

We can now establish \Cref{main}(i):

\begin{proposition}\label{tne} One has $t(N)/N < 1/e$ for $N \neq 1,2,4$.
\end{proposition}

\begin{proof}  From existing data on $t(N)$ (or the linear programming method) one can verify this claim for $N < 80$ (see \Cref{fig1}), so we assume that $N\geq 80$.

Applying \Cref{upper-crit} and \eqref{stirling}, it suffices to show that
\begin{equation}\label{base-test-ineq}
   \sum_{p \geq \frac{N/e}{\lfloor\sqrt{N/e}\rfloor}} f_{e}(p/N) > \frac{1}{2} \log(2\pi N) + \frac{1}{12N}.
\end{equation}
This may be easily verified numerically in the range $80 \leq N \leq 5000$ (see \Cref{fig2}).
We will discard the $\lfloor\sqrt{N/e}\rfloor$ denominator, so it suffices to show
\begin{equation}\label{test-ineq}
  \sum_{N/e < p \leq N} f_{e}(p/N) > \frac{1}{2} \log(2\pi N) + \frac{1}{12N}
\end{equation}
for $N > 5000$.  On $[1/e,1]$, one can compute
$$ \|f_e\|_{\mathrm{TV}^*(1/e,1]}
= 4 - 2 \log 2$$
so by \Cref{osc-lemma} (noting that $5000/e > 1423$) we have
$$ \sum_{N/e < p \leq N} f_{e}(p/N)
\geq \frac{N \left(1-\frac{2}{\sqrt{N/e}}\right)}{\log N} \int_{1/e}^1 f_e(x)\ dx - (4 - 2 \log 2) \frac{E(N)}{\log N}
$$
and so it suffices to show that
$$ \left(1 - \frac{2}{\sqrt{N/e}}\right) \int_{1/e}^1 f_e(x)\ dx \geq
(4 - 2 \log 2) \frac{E(N)}{N}
+ \frac{\log(2\pi N) \log N}{2N} + \frac{\log N}{12N^2}.$$
The right-hand side is increasing in $N$ and the left-hand side is decreasing for $N \geq 5000$, so it suffices to verify this claim for $N=5000$; but this is a routine calculation (with plenty of room to spare; see \Cref{fig2}).
\end{proof}

\begin{figure}
  \centering
  \includegraphics[width=0.8\textwidth]{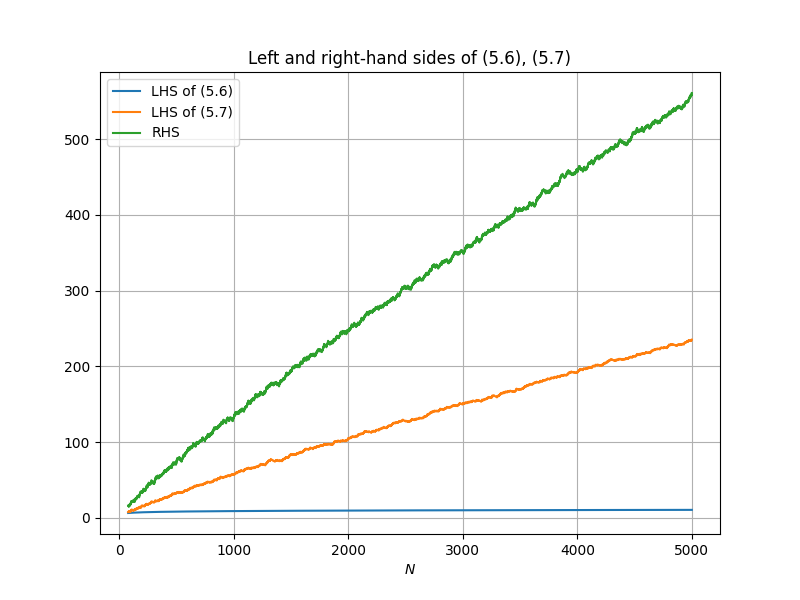}
  \vspace{-8pt}
  \caption{A plot of the left and right sides of \eqref{base-test-ineq}, \eqref{test-ineq} for $80 \leq N < 5000$.}\label{fig2}
\end{figure}

\section{Rearranging the standard factorization}\label{rearrange-sec}

In this section we describe an approach to establishing lower bounds on $t(N)$ by starting with the standard factorization $\{1,\dots,N\}$, dividing out some small prime factors from some of the terms, and then redistributing them to other terms.  This approach was introduced in \cite{guy-selfridge} to give lower bounds of the shape $\frac{t(N)}{N} \geq \frac{3}{16} + o(1)$ (by redistributing powers of two only); \cite{guy-selfridge} also claimed one could show $\frac{t(N)}{N} \geq \frac{1}{4} + o(1)$ by redistributing powers of two and three, but we show in \Cref{onefourth} that this does not work.  With computer assistance, we are also able to show that $\frac{t(N)}{N} \geq \frac{1}{3}+o(1)$ for sufficiently large $N$, in a simpler fashion than the method used to prove \Cref{main}(iv) in the next section.  Finally, we give an alternate proof that $\frac{t(N)}{N} = \frac{1}{e}+o(1)$.

We need some notation.  Define a \emph{downset} to be a finite set ${\mathcal D}$ of natural numbers, obeying the following axioms:
\begin{itemize}
  \item $1 \in {\mathcal D}$.
  \item If $d \in {\mathcal D}$ then all factors of $d$ lie in ${\mathcal D}$.
  \item If $pd \in {\mathcal D}$ for some prime $p$, then $p'd \in {\mathcal D}$ for all primes $p'<p$.
\end{itemize}
For instance, $\{1,2,3,4,6\}$ is a downset.

If $d$ is an element of a downset ${\mathcal D}$, we define $A_{d,\mathcal{D}}$ to be the set of natural numbers not divisible by any prime $p$ with $p < P_+(d)$ or $pd \in \mathcal{D}$.  For instance, if ${\mathcal D} = \{1,2,3,4\}$, then
\begin{itemize}
\item $A_{1,{\mathcal D}}$ is the set of $3$-rough numbers (numbers coprime to $6$);
\item $A_{2,{\mathcal D}} = A_{3,{\mathcal D}}$ is the set of odd numbers; and
\item $A_{4,{\mathcal D}}$ is the set of all natural numbers.
\end{itemize}

The fundamental theorem of arithmetic asserts that every natural number can be uniquely factored as $n = p_1 \dots p_r$ for some primes $p_1 \leq \dots \leq p_r$.  If we let $d \coloneqq p_1 \dots p_k$ be the largest initial segment of this factorization that lies in ${\mathcal D}$, then we have $n=dm$ with $m \in A_{d,{\mathcal D}}$.  Conversely, if $d \in {\mathcal D}$ and $m \in A_{d,{\mathcal D}}$, then applying this procedure to $n=dm$ will recover precisely these factors.  We thus have a partition
\begin{equation}\label{partition}
 \N = \biguplus_{d \in {\mathcal D}} d \cdot A_{d,{\mathcal D}}
\end{equation}
of the natural numbers into various multiples of $d$, for various $d \in {\mathcal D}$.  For instance, in the example ${\mathcal D} = \{1,2,3,4\}$ appearing above, then \eqref{partition} partitions the natural numbers into four classes: the $3$-rough numbers, the odd numbers multiplied by two, the odd numbers multiplied by three, and the multiples of four.

We can use this partition to rearrange the standard factorization $\{1,\dots,N\}$ of $N!$ into a new factorization by extracting out the factors of $d$ arising from \eqref{partition} and reallocating them efficiently to the smaller elements of the resulting subfactorization to make it $t$-admissible.  Specifically, we have

\begin{proposition}[Criterion for lower bound]\label{rearrange-crit}  Let ${\mathcal D}$ be a downset, let $0 < \alpha < 1$, let $N \geq 1$, and suppose we have non-negative reals $a_\ell$ for all natural numbers $\ell$ obeying the following conditions:
  \begin{itemize}
  \item[(i)]  For all primes $p$, one has
  \begin{equation}\label{i-eq-1}
     \sum_{\ell=1}^\infty \nu_p(\ell) a_\ell  \leq \sum_{d \in {\mathcal D}} \frac{\nu_p(d)}{d} \frac{\# (A_{d,{\mathcal D}} \cap [1,N/d])}{N/d}.
  \end{equation}
  \item[(ii)]  For every natural number $\ell$, we have
  \begin{equation}\label{l-eq-1} \sum_{\ell' > \ell} a_{\ell'} \geq \sum_{d \in {\mathcal D}} \frac{\# (A_{d,{\mathcal D}} \cap [1,\min(N/d, \alpha N/\ell)])}{N}.
  \end{equation}
  \end{itemize}
  Suppose further that the $a_\ell N$ are integers for all $\ell$.   Then $t(N) \geq \alpha N$.
  \end{proposition}

  Informally, ${\mathcal D}$ represents the factors that one removes (as greedily as possible) from the standard factorization $\{1,\dots,N\}$ of $N!$ to free up some prime factors, and $a_\ell$ is the proportion of elements in the resulting subfactorization that are to be multiplied by $\ell$ (again, in a greedy fashion) to (hopefully) bring the subfactorization into $t$-admissibility.  The condition \eqref{i-eq-1} asserts that enough primes are freed up by the first step to ``afford'' the second step, while \eqref{l-eq-1} is the assertion that the $a_\ell$ have enough "mass" at large $\ell$ to make even the smallest elements of the subfactorization $t$-admissible.

  \begin{proof}[Proof of \Cref{rearrange-crit}]  Restricting \eqref{partition} to $[1,N]$ and multiplying, we obtain a factorization
  $$ N! = \prod_{d \in {\mathcal D}} d^{|A_{d,{\mathcal D}} \cap [1,N/d]|} \times \prod \tuple$$
  where $\tuple$ is the multiset
  $$ \tuple \coloneqq \biguplus_{d \in {\mathcal D}} \left(A_{d,{\mathcal D}} \cap \left[1,\frac{N}{d}\right]\right).$$
  Thus $\tuple$ has cardinality $|\tuple| = N$, and is a subfactorization of $N!$ with surplus
  \begin{equation}\label{p-surplus-1}
   \nu_p\left( \frac{N!}{\prod \tuple} \right) = \sum_{d \in {\mathcal D}} \nu_p(d) \left|A_{d,{\mathcal D}} \cap \left[1,\frac{N}{d}\right]\right| \geq      \sum_{\ell=1}^\infty \nu_p(\ell) a_\ell N
  \end{equation}
  thanks to \eqref{i-eq-1}.  In particular this multiset is in balance for all primes $p \not \in {\mathcal D}$.

  The multiset $\tuple$ will contain elements that are smaller than $\alpha N$; but we can compute the number of such elements  precisely.  Indeed, for any natural number $\ell$, the number of elements of $\tuple$ that are less than $\alpha N/\ell$ is
  \begin{equation}\label{N-sum-1}
    \sum_{d \in {\mathcal D}} \left|A_{d,{\mathcal D}} \cap \left[1,\frac{N}{d}\right] \cap \left[1,\frac{\alpha N}{\ell}\right)\right|
    \leq  \sum_{\ell' > \ell} a_{\ell'} N
  \end{equation}
  by \eqref{l-eq-1}.

  We now form a modification $\tuple'$ of the multiset $\tuple$ by multiplying each element of $\tuple$ by an appropriate natural number to make it at least $\alpha N$.  More precisely, we perform the following algorithm.
  \begin{itemize}
    \item Initialize $\tuple'$ to be empty, and initialize $\ell$ to be the largest natural number for which $a_\ell > 0$.  (From \eqref{i-eq-1} and the hypothesis that the $a_\ell N$ are integers, it is easy to see that $\ell$ exists.)
    \item For the $a_\ell N$ smallest elements $m$ of $\tuple$ not already selected, add $\ell m$ to $\tuple'$ (counting multiplicity).
    \item Decrement $\ell$ by one, and repeat the previous step until $\ell=1$, or until all elements of $\tuple$ have been selected.
    \item For any remaining elements $m$ of $\tuple$ that have not been involved in any previous step, add $m$ to $\tuple'$.
  \end{itemize}
  It is clear that $\tuple'$ has the same cardinality $N$ as $\tuple$.  If $p \not \in {\mathcal D}$, the hypothesis \eqref{i-eq-1} forces $a_\ell=0$ for all $\ell$ that are divisible by $p$; because of this, $\tuple'$ remains in balance at those primes.  For the primes $p$ in ${\mathcal D}$, we see from construction that the $p$-surplus of $\tuple'$ has decreased from $\tuple$ by at most
  $\sum_\ell \nu_p(\ell) a_\ell N$, and so from \eqref{p-surplus-1}, $\tuple'$ is either in $p$-surplus or in $p$-balance, and is thus a subfactorization of $N!$.

It remains to verify that $\tuple'$ is $\alpha N$-admissible.  From an inspection of the algorithm, we see that the only way this can fail to be the case is if, for some $\ell$, the number of elements of $\tuple$ in $[1, \alpha N/\ell)$ exceeds the quantity $\sum_{\ell' > \ell} a_{\ell'} N$.  But this is ruled out by \eqref{N-sum-1}.
\end{proof}

One advantage of this criterion is that it is amenable to taking asymptotic limits as $N \to \infty$, holding the downset ${\mathcal D}$ and the ratio $\alpha$ fixed.  From the Chinese remainder theorem we observe that the sets $A_{d,{\mathcal D}}$ have density
$$ \sigma_{d,{\mathcal D}} \coloneqq \prod_{p < P_+(d) \text{ or } pd \in {\mathcal D}} \left(1-\frac{1}{p}\right),$$
in the sense that
\begin{equation}\label{crt}
 |A_{d,{\mathcal D}} \cap I| = \sigma_{d,{\mathcal D}} |I| + O(1)
\end{equation}
for any interval $I$ (where the implied constant is permitted to depend on ${\mathcal D}$).  From \eqref{partition} we then have the identity
\begin{equation}\label{sigma-id}
  \sum_{d \in {\mathcal D}} \frac{\sigma_{d,{\mathcal D}}}{d} = 1.
\end{equation}

\begin{example}  The set ${\mathcal D} = \{1,2,4\}$ is a downset, with $\sigma_{1,{\mathcal D}} = \sigma_{2,{\mathcal D}} = \frac{1}{2}$ and $\sigma_{4,\mathcal D} = 1$.  The set ${\mathcal D}' = \{1,2,3,4\}$ is also a downset with $\sigma_{1,{\mathcal D}} = \frac{1}{3}$, $\sigma_{2,{\mathcal D}} = \frac{1}{2}$, $\sigma_{3,{\mathcal D}} = \frac{1}{2}$, $\sigma_{4,{\mathcal D}} = 1$.  The identity \eqref{sigma-id} becomes
$$ \frac{1/2}{1} + \frac{1/2}{2} + \frac{1}{4} = 1$$
for the former downset and
$$ \frac{1/3}{1} + \frac{1/2}{2} + \frac{1/2}{3} + \frac{1}{4} = 1$$
for the latter downset.
\end{example}

As an aside, we make the constant in \eqref{crt} explicit in certain situations that will arise later in \Cref{two-sevenths}:
\begin{lemma} \label{lit7}
  Suppose $p\le 7$ is prime.  Let $A_{p}$ denote the $p$-rough numbers, that is, integers not divisible by any prime $p'\le p$.  Then
  $$ \#(A_{p}\cap[1,x]) = x\prod_{p'\le p} \left(1-\frac{1}{p'}\right) + O_{\le}(53/35). $$
\end{lemma}
\begin{proof}
  As in the proof of \Cref{lit}, we consider the functions $$\#(A_{p}\cap[1,x]) - x\prod_{p'\le p} \left(1-\frac{1}{p'}\right)$$
  for each prime $p\le 7$.
  These functions are periodic with period $\prod_{p'\le p} p'$ respectively (which is at most $2\cdot 3\cdot 5\cdot 7=210$), and we can check that the maximum (absolute) value attained across all of them is for $p=7$ at $x=210-11$ or as $x\to 11^{-}$.
\end{proof}

We now have an asymptotic version of \Cref{rearrange-crit}:

\begin{proposition}[Criterion for asymptotic lower bound]\label{asym-crit}  Let ${\mathcal D}$ be a downset that contains at least one prime $p_0$, let $0  < \alpha < 1$, and suppose we have non-negative reals $a_\ell$ for all natural numbers $\ell$ obeying the following conditions:
\begin{itemize}
\item[(i)]  For all primes $p$, one has
\begin{equation}\label{i-eq}
   \sum_{\ell=1}^\infty \nu_p(\ell) a_\ell  \leq \sum_{d \in {\mathcal D}} \sigma_{d,{\mathcal D}} \frac{\nu_p(d)}{d}.
\end{equation}
\item[(ii)]  For every natural number $\ell$, we have
\begin{equation}\label{l-eq} \sum_{\ell' > \ell} a_{\ell'} > \sum_{d \in {\mathcal D}} \sigma_{d,{\mathcal D}} \min\left( \frac{1}{d}, \frac{\alpha}{\ell}\right).
\end{equation}
\end{itemize}
Then $t(N) \geq \alpha N$ for all sufficiently large $N$.
\end{proposition}

\begin{proof}  We first make some small technical modifications to the sequence $a_\ell$. If $p \in {\mathcal D}$, then the right-hand side of \eqref{i-eq} is positive.  If equality holds here, then one of the $a_\ell$ with $\ell$ divisible by $p$ is positive; but one can reduce this quantity slightly without violating \eqref{l-eq}, since this $a_\ell$ only impacts finitely many cases of these strict inequalities.  Thus, we may assume without loss of generality that the inequality \eqref{i-eq} is strict for all $p \in {\mathcal D}$.

From \eqref{i-eq} and \eqref{ftoa} we see that
$$ \sum_{\ell=1}^\infty a_\ell \log \ell \leq \sum_{d \in {\mathcal D}} \frac{\sigma_{d,{\mathcal D}}}{d} \log d.$$
In particular, we have the decay bound
\begin{equation}\label{tail}
  \sum_{\ell \geq \sqrt{N}} a_\ell \ll \frac{1}{\log N}
\end{equation}
(we allow implied constants to depend on ${\mathcal D}$).

Now let $\ell_0$ be a sufficiently large natural number, and assume that $N$ is sufficiently large depending on ${\mathcal D}$ and $\ell_0$.  We introduce the modified sequence
$$ a'_\ell \coloneqq \left\lfloor \left(a_\ell + \frac{1_{\ell \in S}}{\sqrt{\ell}}\right) N \right\rfloor / N,$$
where $S$ is the set of all natural numbers $\ell > \ell_0$ that are only divisible by primes in ${\mathcal D}$.
Clearly the $N a'_\ell$ are integers, so by \Cref{rearrange-crit} it will suffice to verify the hypotheses \eqref{i-eq-1}, \eqref{l-eq-1}.

We begin with \eqref{i-eq-1}.  By construction, both sides vanish for $p \not \in {\mathcal D}$, so we only need to verify the finite number of cases when $p \in {\mathcal D}$.  By \eqref{crt}, the right-hand side of \eqref{i-eq-1} is
$$ \sum_{d \in {\mathcal D}} \sigma_{d,{\mathcal D}} \frac{\nu_p(d)}{d} + O\left(\frac{1}{N} \right).$$
The left-hand side is at most
$$ \sum_{\ell=1}^\infty \nu_p(\ell) a_\ell + \sum_{\ell \in S} \frac{\nu_p(\ell)}{\sqrt{\ell}}.$$
By the dominated convergence theorem and Euler products, the second term goes to zero as $\ell_0 \to \infty$.  Since we are assuming that \eqref{i-eq} holds with strict inequality, we thus conclude \eqref{i-eq-1} for all sufficiently large $N$.

Now we turn to \eqref{l-eq-1}.  For the finite number of cases $\ell \leq \ell_0$, we use \eqref{crt} again to write the right-hand side as
$$ \sum_{d \in {\mathcal D}} \sigma_{d,{\mathcal D}} \min\left( \frac{1}{d}, \frac{\alpha}{\ell}\right) + O\left(\frac{1}{N} \right).$$
Using the lower bound $a'_\ell \geq a_\ell - O(1/N)$ for $\ell \leq \sqrt{N}$, as well as \eqref{tail}, the left-hand side is at least
$$ \sum_{\ell' > \ell} a_{\ell'} - O\left(\frac{1}{\sqrt{N}} \right) - O\left(\frac{1}{\log N} \right).$$
Since we have strict inequality in \eqref{l-eq}, we obtain \eqref{l-eq-1} for all sufficiently large $N$ under the regime $\ell \leq \ell_0$.

In the case $\ell>N$ the claim \eqref{l-eq-1} is trivial because the right-hand side vanishes, so it remains to consider the case $\ell_0 < \ell \leq N$. Here the right-hand side can be crudely bounded by $O(1/\ell)$.  On the other hand, by using the powers of $p_0$ in $S$ one can lower bound the left-hand side by $\gg 1/\sqrt{\ell}$.  For $\ell_0$ large enough, the claim follows.
\end{proof}

\begin{example}  Let $0 < \alpha < 3/16$ and ${\mathcal D} = \{1,2,4\}$.  If we set $a_{2^r} = \frac{3}{2^{r+3}}$ for $r \geq 1$, and $a_\ell=0$ for all other $\ell$, then one can calculate that
$$ \sum_{\ell=1}^\infty a_\ell \nu_2(\ell) = \frac{3}{4} = \sum_{d \in {\mathcal D}} \nu_2(d) \frac{\sigma(d)}{d}$$
and
$$ \sum_{\ell > 2^r} a_{2^r} = \frac{3}{2^{r+3}} > \frac{2\alpha}{2^r} = \sum_{d \in {\mathcal D}} \sigma_{d,{\mathcal D}} \min\left( \frac{1}{d}, \frac{\alpha}{2^r} \right)$$
for any $r \geq 1$.  From this one can readily check that the hypotheses of \Cref{asym-crit} are satisfied, and so we recover the bound $\frac{t(N)}{N} \geq \frac{3}{16}-o(1)$ from \cite{guy-selfridge}.
\end{example}

\begin{figure}
  \centering
  \includegraphics[width=0.8\textwidth]{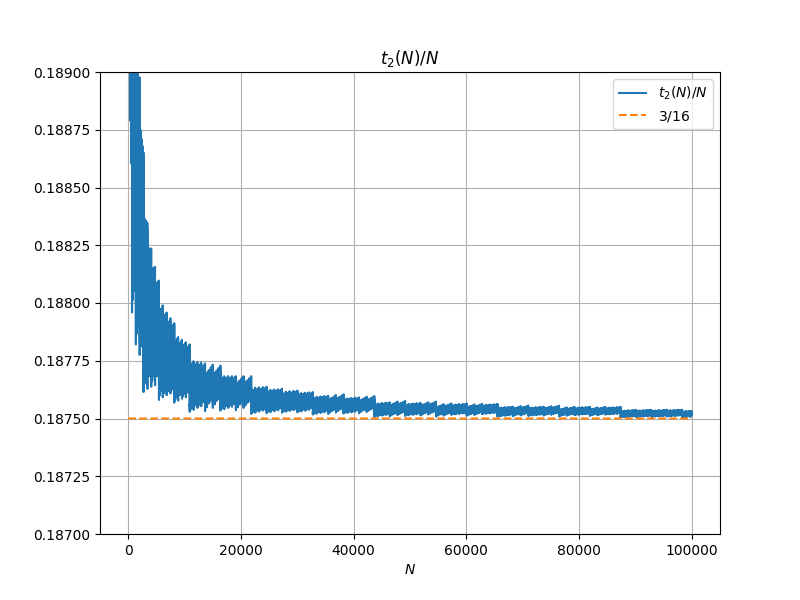}
  \vspace{-8pt}
  \caption{The function $t_2(N)/N$ was computed exactly in \cite{guy-selfridge}, and converges quickly to $3/16$.}\label{fig-t2}
\end{figure}

Unfortunately, our further applications of Propositions \ref{rearrange-crit} and \ref{asym-crit} are not as human-readable as this example.
However, these two criteria are (infinite) linear programs, so they are very amenable to computer-assisted proofs.
Furthermore, the solutions to the linear programs can be verified with relatively simple code and in exact arithmetic, thereby avoiding some potential pitfalls of computer-assisted proofs.

\begin{proposition}[\Cref{main}(iii) for sufficiently large $N$ only] \label{onethird-large} Let $N$ be sufficiently large.  Then $t(N) \ge N/3$.
\end{proposition}
\begin{proof}[Computer-assisted proof]
If one makes the ansatz $a_{2^r} = c / 2^r$ for some $c>0$ and all $r \geq r_0$, with $a_\ell = 0$ for all other $\ell > 2^{r_0}$, then the task of locating weights $a_r$ obeying the hypotheses of \Cref{asym-crit} becomes a \emph{finite} linear programming problem.
Numerically\footnote{\url{https://github.com/teorth/erdos-guy-selfridge/blob/main/src/dnup/data.py}}, we were able to locate such weights for $\alpha=1/3$, ${\mathcal D} = \{d: 1 \leq d \leq 2^{11}\}$, and $r_0=11$.
Once the weights were found, they were converted into rational numbers and the linear program was verified\footnote{\url{https://github.com/teorth/erdos-guy-selfridge/blob/main/src/dnup/verify.py}} in exact arithmetic, thus establishing that $t(N) \geq N/3$ for sufficiently large $N$.
\end{proof}

We do not explicitly compute a bound on $N$ which is necessary for the previous proposition to hold because we expect it to be significantly worse than that obtained by the modified approximate factorization strategy later in \Cref{10^{11}}, and in particular, far worse than what would be necessary to link up with the calculations using the greedy method of \Cref{greedy-sec}.
However, we can use \Cref{rearrange-crit} to prove \Cref{main}(ii):
\begin{proposition}\label{two-sevenths} Suppose $N\ne 56$.  Then $t_{2,3,5,7}(N) \ge \lfloor 2N/7\rfloor$.
\end{proposition}
\begin{proof}[Computer-assisted proof]
For small $N$, we are able to compute $t_{2,3,5,7}(N)$ exactly using integer programming techniques similar to those mentioned earlier.
This verifies that $t_{2,3,5,7}(N)\ge \lfloor 2N/7\rfloor$ for all $N\le 1.2\times 10^7$, except $N=56$ of course.
As a result, in the remainder of the proof we may assume that $N$ exceeds this bound.

In the asymptotic regime, we can apply \Cref{asym-crit} directly.
The simplest setup that we found has $\mathcal{D}=\{1,2,3,4,5,6,7,8,9,10,12,14,15,16,18,20,21,24,25,27,28\}$, the $7$-smooth numbers (that is, the numbers only divisible by primes up to $7$) up to $28$.
Note that since we only use $7$-smooth numbers, we are indeed rearranging only prime factors up to $7$.

Because the linear program for the $a_\ell$ is infinite, we used the ansatz that $a_\ell = a_{\ell/2}/2$ if $\ell\ge \ell_0$ is sufficiently large to simplify the setup.
(This $\ell_0$ is not related to the $\ell_0$ in the \emph{proof} of \Cref{asym-crit}, which we do not use.)
We found a solution to the linear program that produced decent bounds with $\ell_0 = 52$, which gives explicit values for $a_\ell$ for $28$ small indices $\ell<52$ and then exponential decay afterwards as above, so for example $\ell_{54}=\ell_{27}/2$.
In this setup, the terms in \eqref{i-eq}~and~\eqref{l-eq} are manageable infinite series and it is fairly straightforward to verify the conditions are satisfied.
This proves the result for sufficiently large $N$, but it is more difficult to compute precisely how large $N$ should be.

For a given $N$, we produce an explicit sequence $a_\ell$ that satisfies the conditions of \Cref{rearrange-crit}, but we use a different strategy from the proof of \Cref{asym-crit}.
Specifically, we define an alternative modified sequence $ a'_\ell \coloneqq \lceil a_\ell N \rceil/N$ if $\ell < 2^L N$ and $0$ otherwise.
Because $a'_\ell N$ are integers, it suffices to verify the conditions \eqref{i-eq-1}~and~\eqref{l-eq-1}.
There are two effects on the terms on the left-hand sides of these inequalities, which is that they potentially get slightly larger due to the ceiling, but also the infinite sums get smaller because we truncate their tails.

In the left-hand sum $\sum_{\ell=1}^\infty \nu_p(\ell) a'_\ell$ of \eqref{i-eq-1}, we ignore the effect of the truncation (as it only makes the situation better for us) and bound the effect of the ceilings.
Because $\lceil a_\ell N\rceil < a_\ell N+1$, each factor $a'_\ell$ increases by at most $1/N$.
If $p=2$, then each term featuring a ceiling (those with $\ell \le 2^L N$) is at most $\nu_p(\ell) \le \nu_p(2^L N) \le L+\log_2(N)$.
If $p>2$, then each term is at most $\nu_p(\ell)$ for a ``small'' $\ell$, namely those less than $52$.
(We have $\nu_3(\ell) \le 3$, $\nu_5(\ell) \le 2$, and $\nu_7(\ell) \le 1$; note that $a_{49}=0$ in our table of explicit values.)
Each $\ell\le 2^L N$ is a power of $2$ times an odd part, so we can bound the number of terms by simply counting the number of possible odd parts $\ell$ less than $52$ with $a_\ell>0$ and considering that each can appear at most with at most $\log_2 N$ factors of two.
(The odd part $1$ can appear one more time, because of fencepost-counting reasons, but the effect is more than compensated by the larger odd parts appearing fewer times.)
All in all, the total increase to the left-hand side of \eqref{i-eq-1} relative to \eqref{i-eq} is at most $(L+\log_2(N))\cdot (\#\text{odd $\ell$s})\cdot \log_2(N)/N$ for $p=2$ and a similar but simpler expression if $p>2$.
Note that this expression is decreasing for $N\ge\exp(2)$.

Now consider \eqref{l-eq-1}.
This inequality is automatically true when $\ell > \alpha N$, so we may assume $\ell \le \alpha N$.
In the left-hand sum $\sum_{\ell' > \ell} a'_{\ell'}$, we ignore the effects of the ceilings (as it only makes the situation better for us) and bound the effect of the truncation.
We first separately compute the smallest $A$ so that $\sum_{\ell' > \ell} a_{\ell'} \le A/\ell$, by explicitly computing these sums for our values~$a_{\ell}$.
Then since $\ell \le \alpha N$, $2^L N \ge (2^L/\alpha)\ell$.
If $\sum_{\ell' > \ell} a_{\ell'} \le A/\ell$ for all $\ell$, it follows $\sum_{\ell' > 2^L N} a_{\ell'} \le \sum_{\ell' > (2^L/\alpha)\ell} a_{\ell'} \le (A \alpha/2^L)/\ell$.
So (again, ignoring the ceiling) the truncated sum $\sum_{\ell' > \ell} a'_{\ell'} = (\sum_{\ell' > \ell} a_{\ell'}) - (\sum_{\ell' > 2^L N} a_{\ell'})$ goes down by at most this much.

We must also account for the deviation of the right-hand sides from their asymptotic values.
Because our downset~$\mathcal{D}$ consists of $7$-smooth numbers, the sets $A_{d,{\mathcal D}}$ consist of $p$-smooth numbers for some $p\le 7$ (or simply of all integers).
We can thus make use of \Cref{lit7} to obtain
$$ \sum_{d \in {\mathcal D}} \frac{\nu_p(d)}{d} \frac{\# (A_{d,{\mathcal D}} \cap [1,N/d])}{N/d} = \sum_{d \in {\mathcal D}} \nu_p(d)\left(\frac{\sigma_{d,{\mathcal D}}}{d} + O_{\le}\left(\frac{53}{35N}\right)\right) $$
and
$$ \sum_{d \in {\mathcal D}} \frac{\# (A_{d,{\mathcal D}} \cap [1,\min(N/d, \alpha N/\ell)])}{N} = \sum_{d \in {\mathcal D}} \left(\sigma_{d,{\mathcal D}} \min\left( \frac{1}{d}, \frac{\alpha}{\ell}\right) + O_{\le}\left(\frac{53}{35N}\right)\right) .$$

Combining all of these error estimates, we can now programmatically compute the terms of \eqref{i-eq-1} and \eqref{l-eq-1} and verify that the inequalities hold for a particular~$N$, even if all terms $O_{\le}(C/N)$ are as extreme as possible.
We performed this verification\footnote{\url{https://github.com/teorth/erdos-guy-selfridge/blob/main/src/dnup/two_sevenths.py}} (in exact arithmetic as with all other proofs in this section), for $N=8.2\times 10^6$, which is better than the results of our small $N$ computation above required, so we are done.
\end{proof}

By a modification of the above techniques, we can now establish \Cref{main}(v) for large $N$.  Let $t_{2,3}(N)$ denote the largest $t$ for which it is possible to create a $t$-admissible factorization (or subfactorization) of $N!$ purely by rearranging powers of $2$ and $3$ in the standard factorization $\{1,\dots,N\}$.  For instance, the lower bounds produced above on $t(N)$ in fact will bound the smaller quantity $t_{2,3}(N)$ as long as ${\mathcal D}$ consists solely of $3$-smooth numbers. On the other hand, there is a limit to this method (see also \Cref{fig-t3}):

\begin{figure}
  \centering
  \includegraphics[width=0.8\textwidth]{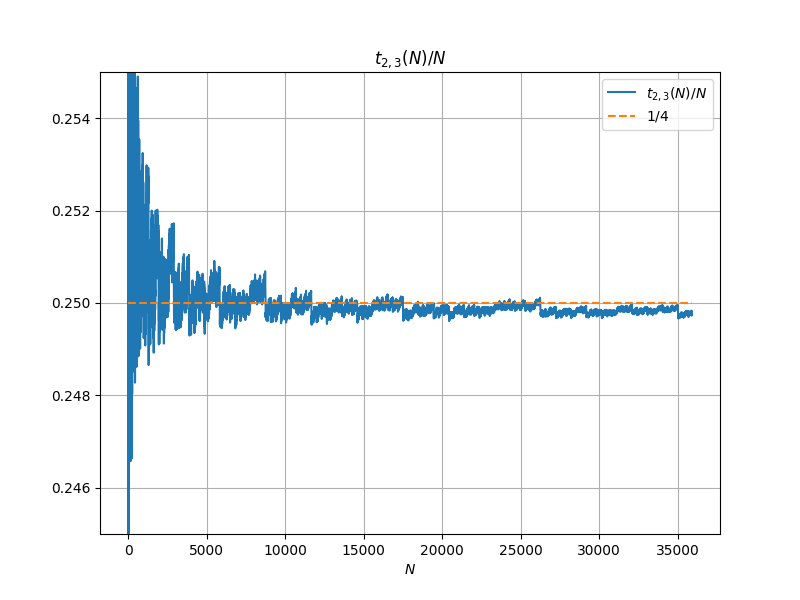}
  \vspace{-8pt}
  \caption{The function $t_{2,3}(N)/N$ can be readily computed by integer or dynamic programming, and asymptotically just barely falls below $1/4$.}\label{fig-t3}
\end{figure}

\begin{proposition}\label{onefourth}  Let $N > 26244$.  Then $t_{2,3}(N) < N/4$.  The threshold is the best possible.
\end{proposition}

\begin{proof}[Computer-assisted proof]
For small $N$, we are able to compute $t_{2,3}(N)$ exactly using the integer programming techniques mentioned earlier.
Alternatively, specifically for this task, we also wrote a \emph{dynamic} programming implementation out of curiosity and to double-check the computations.
This verifies that $t_{2,3}(N)=N/4$ for $N=26244$, and then a combination of integer and linear programming techniques verifies that $t_{2,3}(N) < N/4$ for $26244 < N \le 5\times 10^6$.
As a result, in the remainder of this proof we may assume that $N$ exceeds this bound.

Suppose for contradiction that there is an $N/4$-admissible factorization of $N!$, where the tuple is formed by decomposing each $n \in \{1,\dots,N\}$ into the product $n=dm$ of a $3$-smooth part $d$ and a $3$-rough part $m$, and replacing $d$ by some other $3$-smooth number $\ell$.  If we let $a_\ell$ be the proportion of numbers that are multiplied by $\ell$, then to maintain balance we must have
  \begin{align*}
    \sum_\ell a_\ell &= 1, \\
    \sum_\ell \nu_2(\ell) a_\ell &= \frac{\nu_2(N!)}{N} \leq 1, \\
    \sum_\ell \nu_3(\ell) a_\ell &= \frac{\nu_3(N!)}{N} \leq \frac{1}{2}
  \end{align*}
  thanks to \eqref{legendre}, and since the $\sum_{k < \ell} a_k N$ terms that are multiplied by something less than $\ell$ must have the $m$ component at least $N/4\ell$, we also have
  $$ \sum_{k < \ell} a_k \leq \frac{1}{N} \sum_{d \in \N^{\langle 2,3 \rangle}} \sum^*_{N/4\ell \leq m \leq N/d} 1,$$
  where $\N^{\langle 2,3 \rangle}$ is the set of $3$-smooth numbers.
  The inner sum vanishes for $d > 4\ell$, and is at most $1$ for $d = 4\ell$.  For $d < 4\ell$ we can apply \Cref{lit} to conclude
  $$ \sum_{k < \ell} a_k \leq \sum_{d \in \N^{\langle 2,3 \rangle}: d < 4\ell} \left(\frac{1}{d}-\frac{1}{4\ell} + \frac{4}{3N}\right) + \frac{1}{N}.$$
  Thus, for any non-negative weights $c_2, c_3$, and $w_\ell$ for $3$-smooth $\ell$ such that
  \begin{equation}\label{ell-cond}
     c_2 \nu_2(\ell) + c_3 \nu_3(\ell) + \sum_{\ell' > \ell} w_{\ell'} \geq 1
  \end{equation}
  for all $3$-smooth $\ell$, we have
  $$  1= \sum_\ell a_\ell \leq c_2 \sum_\ell \nu_2(\ell) a_\ell + c_3 \sum_\ell \nu_3(\ell) a_\ell + \sum_{\ell} w_\ell \sum_{k < \ell} a_k \leq 1 - \eps  + \frac{C}{N}
  $$
  where
  $$ \eps \coloneqq 1 -
c_2 - \frac{c_3}{2} - \sum_{\ell} w_\ell \sum_{d \in \N^{\langle 2,3 \rangle}: d < 4\ell} \left( \frac{1}{d} - \frac{1}{4\ell} \right),$$
  $$ C \coloneqq \sum_\ell w_\ell \left(\frac{4}{3} \# \left\{ d \in \N^{\langle 2,3 \rangle}: d < 4\ell \right\} + 1\right).$$
 Thus, once one produces weights for which $\eps>0$, one obtains a contradiction for $N > C/\eps$.
  Numerical computation reveals that if one selects $c_2 = 2/32$, $c_3 = 3/32$, $w_1 = 2/32$, and $w_\ell = 1/32$ for all $3$-smooth $1 < \ell \leq 2^2 3^9$ with $\nu_2(\ell) \leq 2$, with $w_\ell=0$ otherwise, then  \eqref{ell-cond} holds for all $3$-smooth $\ell$ (note that this is automatic once $\nu_2(\ell)$ or $\nu_3(\ell)$ is large enough, so this is a finite check) with
$$\eps = \frac{218038591}{4458050224128} > 0$$
and
$$ C = \frac{1559}{24}$$
and then we obtain the desired contradiction for $N > 1328148 = \lceil\frac{C}{\eps}\rceil$. These coefficients were discovered by applying a linear program to a finite truncation of the problem, choosing the truncation so as to optimize the threshold.

As in \Cref{onethird-large}, we verified\footnote{\url{https://github.com/teorth/erdos-guy-selfridge/blob/main/src/dnup/one_fourth.py}} the linear programming certificate in exact arithmetic.
However, in this case we note that the proof is fundamentally human-checkable.
There are only a few dozen $3$-smooth numbers up to $2^2 3^9$ and the coefficients and weights are very simple.
It is possible that examining this construction in more detail would reveal some greater human understanding of the asymptotic efficiency of these rearrangement methods.
\end{proof}

\begin{figure}
  \centering
  \includegraphics[width=0.8\textwidth]{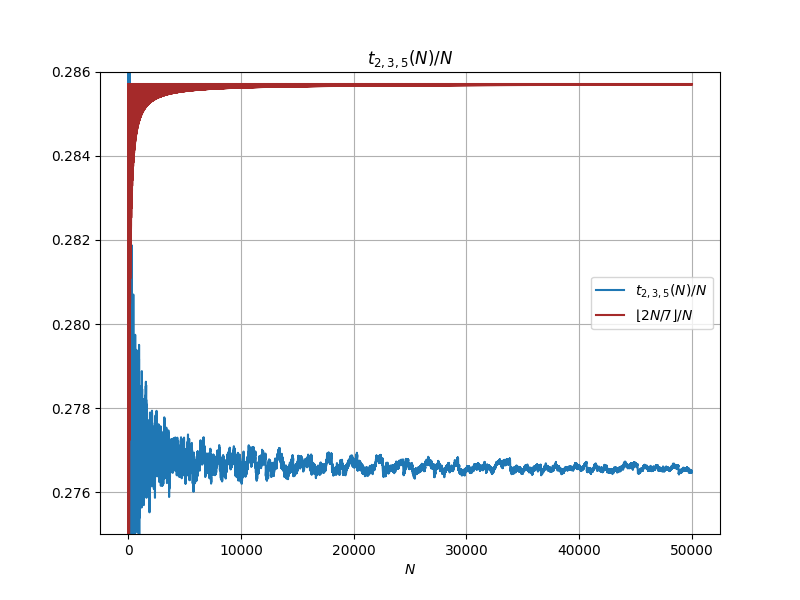}
  \vspace{-8pt}
  \caption{The function $t_{2,3,5}(N)/N$ is asymptotically larger than $1/4$, but falls well short of $\lfloor 2N/7 \rfloor/N$, let alone $1/3$ or $1/e$.}\label{fig-t5}
\end{figure}

\begin{figure}
  \centering
  \includegraphics[width=0.8\textwidth]{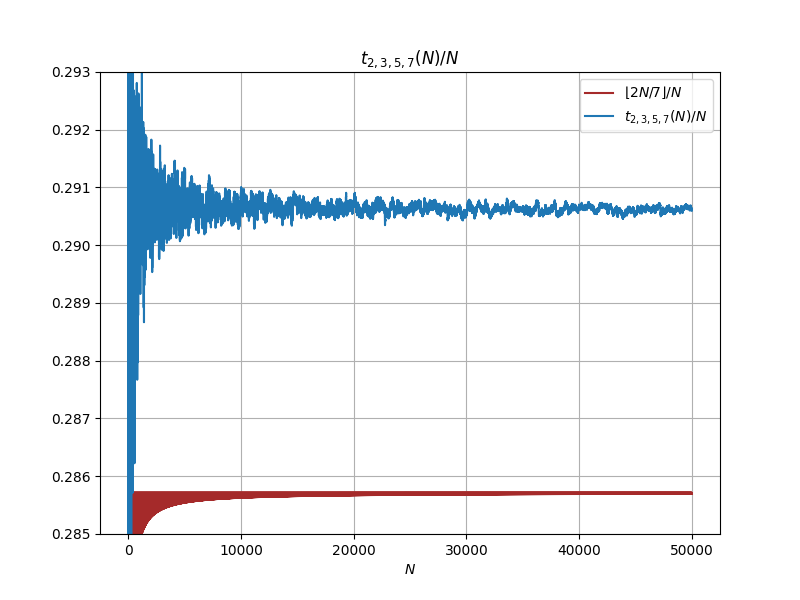}
  \vspace{-8pt}
  \caption{The function $t_{2,3,5,7}(N)/N$ numerically converges to a limit between $2/7$ and $1/3$.}\label{fig-t7}
\end{figure}

Finally, we can recover a weak version of \Cref{main}(iv) with this method:

\begin{proposition}[Asymptotic lower bound]\label{asymp}  If $0 < \alpha < 1/e$, then one has $t(N) \geq \alpha N$ for all sufficiently large $N$.
\end{proposition}

\begin{proof}  We select the following parameters:
  \begin{itemize}
  \item A sufficiently small real $c_0>0$ (which can depend on $\alpha$);
  \item A sufficiently large natural number $M$ (which can depend on $c_0$, $\alpha$);
  \item A sufficiently large natural number $C_0$ (which can depend on $M$, $c_0$, $\alpha$);
  \item A sufficiently large prime $p_-$ (which can depend on $C_0$, $M$, $c_0$, $\alpha$); and
  \item A prime $p_+$ with $\log p_+ \asymp \log^2 \log p_-$.
  \end{itemize}
Let ${\mathcal D}$ be the set of all numbers $d$ which are either of the form $2^m$ for $0 \leq m \leq M$, or $2^m p$ for $0 \leq m \leq M$ and $p_- \leq p \leq p_+$.  This is a downset, and the densities $\sigma_{d,{\mathcal D}}$ can be computed explicitly as
$$ \sigma_{2^m,{\mathcal D}} = \frac{1}{2} \mu_+; \qquad \sigma_{2^m p,{\mathcal D}} = \frac{1}{2} \mu_p$$
for $0 \leq m < M$ and
$$ \sigma_{2^M,{\mathcal D}} = \mu_+; \qquad \sigma_{2^M p,{\mathcal D}} = \mu_p,$$
where
$$ \mu_p \coloneqq \prod_{p_- \leq p' < p} \left( 1 - \frac{1}{p'} \right); \qquad \mu_+ \coloneqq \prod_{p_- \leq p' \leq p_+} \left( 1 - \frac{1}{p'} \right).$$
The identity \eqref{sigma-id} is then equivalent to the telescoping identity
\begin{equation}\label{telescope}
   \sum_{p_- \leq p \leq p_+} \frac{\mu_p}{p} + \mu_+ = 1.
\end{equation}
We can now define the weights $a_\ell$ as follows:
\begin{itemize}
\item[(i)] If $\ell = 2^m$ for some $m \geq 1$, we set $\alpha_\ell \coloneqq C_0 \frac{m}{2^m} \mu_+$.
\item[(ii)] If $\ell = 2^{C_0} p$ for some $p_- \leq p \leq 2^{C_0} p_-$, we set $\alpha_\ell \coloneqq \mu_p/p$.
\item[(iii)] If $\ell = p$ for some $2^{C_0} p_- < p < p_+/2$, we set $\alpha_\ell \coloneqq c_0 \mu_p/p$.
\item[(iv)] If $\ell = 2p$ for some $2^{C_0} p_- < p < p_+/2$, we set $\alpha_\ell \coloneqq (1-c_0) \mu_p/p$.
\item[(v)] If $\ell = 2^{C_0+m} p$ for some $p_+/2 \leq p \leq p_+$ and $m \geq 1$, we set $\alpha_\ell = \frac{m}{2^m} \mu_p/p$.
\item[(vi)] In all other cases, we set $\alpha_\ell = 0$.
\end{itemize}

By \Cref{asym-crit}, it suffices to verify the conditions \eqref{i-eq}, \eqref{l-eq}.  We begin with \eqref{i-eq}.  If $p$ is not equal to $2$ or is in the range $[p_-,p_+]$, then both sides of \eqref{i-eq} vanish.  If $p$ is in $[p_-,p_+]$, then a routine computation shows that both sides are equal to $\mu_p$.  For $p=2$, the right-hand side can be simplified using \eqref{telescope} to
$$ \sum_{0 \leq m \leq M}^* \frac{1}{2} \frac{m}{2^m} = 1 - O\left( \frac{M}{2^M} \right)$$
where the $*$ means that the term with $m=M$ is doubled (so $\sum_{0 \leq m \leq M}^* f(m) = \sum_{0 \leq m<M} f(m) + 2 f(M)$).  Meanwhile, the left-hand side can be computed to equal
\begin{align*}
  \sum_{m=1}^\infty \frac{C_0 m^2}{2^m} \mu_+ & + \sum_{p_- \leq p \leq 2^{C_0} p_-} C_0 \frac{\mu_p}{p} \\
   + \sum_{2^{C_0} p_- < p < p_+/2} (1-c_0) \frac{\mu_p}{p} & + \sum_{p_+/2 \leq p \leq p_+} \sum_{m=1}^\infty \frac{(C_0+m) m}{2^m} \frac{\mu_p}{p}
\end{align*}
which simplifies using \eqref{telescope} and summation in $m$ to
$$ 1 - c_0 + O_{C_0}\left( \mu_+ + \sum_{p_- \leq p \leq 2^{C_0} p_-} \frac{\mu_p}{p} + \sum_{p_+/2 \leq p \leq p_+} \frac{\mu_p}{p}\right).$$
From Mertens' theorem (or \Cref{osc-lemma}) one has
\begin{equation}\label{mertens}
  \mu_p = \frac{\log p_-}{\log p} \left( 1 + O\left( \frac{1}{\log^{10} p} \right) \right)
\end{equation}
and similarly
\begin{equation}\label{mertens-2}
  \mu_+ = \frac{\log p_-}{\log p_+} \left( 1 + O\left( \frac{1}{\log^{10} p_+} \right) \right).
\end{equation}
\eqref{i-eq} for $p=2$ then follows from the choice of parameters after a brief calculation.

It remains to verify \eqref{l-eq}.  We first consider the case $\ell < 2^{C_0} p_-$.  In this case, the left-hand side simplifies using \eqref{telescope} to
$$ \sum_{m: 2^m > \ell} \frac{C_0 m}{2^m} \mu_+ + (1-\mu_+)$$
and the right-hand side similarly simplifies to
$$ \sum_{0 \leq m \leq M}^* \frac{1}{2} \min\left( \frac{1}{2^m}, \frac{\alpha}{\ell}\right) \mu_+ + (1-\mu_+).$$
Thus it suffices to show that
\begin{equation}\label{lint} \sum_{0 \leq m \leq M}^* \frac{1}{2} \min\left( \frac{1}{2^m}, \frac{\alpha}{\ell}\right) < \sum_{m: 2^m > \ell} \frac{C_0 m}{2^m}.
\end{equation}
But the left-hand side is $\ll \frac{\log (2+\ell)}{\ell}$ and the right-hand side is $\gg C_0 \frac{\log (2+\ell)}{\ell}$, giving the claim \eqref{l-eq} in this case.

Next, we consider \eqref{l-eq} in the case $\ell \geq p_+$.  The left-hand side of \eqref{l-eq} is at least
$$ \sum_{m: 2^m > \ell} \frac{C_0 m}{2^m} \mu_+
+ \sum_{p_+/2 \leq p \leq p_+} \sum_{m: 2^{m+C_0} p > \ell} \frac{m}{2^m} \frac{\mu_p}{p}$$
and the right-hand side is at most
\begin{equation}\label{amp}
  \sum_{0 \leq m \leq M}^* \frac{1}{2} \min(\frac{1}{2^m}, \frac{\alpha}{\ell})\mu_+ + \sum_{0 \leq m \leq M}^* \frac{1}{2} \sum_{p_- \leq p \leq p_+} \min\left(\frac{1}{2^mp}, \frac{\alpha}{\ell}\right) \mu_p.
\end{equation}
By \eqref{lint} it suffices to show that
$$
\sum_{p_+/2 \leq p \leq p_+} \sum_{m: 2^{m+C_0} p > \ell} \frac{m}{2^m} \frac{\mu_p}{p}
> \sum_{0 \leq m \leq M}^* \frac{1}{2} \sum_{p_- \leq p \leq p_+} \min\left(\frac{1}{2^mp}, \frac{\alpha}{\ell}\right) \mu_p.$$
The left-hand side can be computed using \eqref{mertens} and the prime number theorem to be
$$ \gg 2^{C_0} \frac{\log p_-}{\log^2 p_+} \frac{\log (2+\ell/p_+)}{\ell}.$$
By \eqref{mertens} and \Cref{osc-lemma}, the right-hand side may be bounded by
$$ \ll (\log p_-) \sum_{m=0}^\infty \int_{p_-}^{p_+} \min\left(\frac{1}{2^m t}, \frac{1}{\ell}\right)\ \frac{dt}{\log^2 t}.$$
We can perform the summation over $m$ and bound this by
$$ \ll \frac{\log p_-}{\ell} \int_{p_-}^{p_+} \log\left(2 + \frac{\ell}{t}\right)\ \frac{dt}{\log^2 t},$$
and the claim \eqref{l-eq} in this case then follows from a routine computation.

Finally, we consider the case $2^{C_0} p_- < \ell < p_+$.
Note that if we redefined $a_{\ell'}$ to make rules (iii), (iv) apply for all $p_- \leq p \leq p_+$, and delete rules (ii) and (v), then this amounts to transferring the mass of $a_{\ell'}$ from larger $\ell'$ to smaller $\ell'$, so that the sum $\sum_{\ell' > \ell} a_{\ell'}$ does not increase.  From this observation, we see that we can lower bound the left-hand side of \eqref{l-eq} by
$$
\sum_{m: 2^m > \ell} \frac{C_0 m}{2^m} \mu_+ + \sum_{p_- \leq p \leq p_+} \frac{\mu_p}{p} \left( c_0 1_{p>\ell} + (1-c_0) 1_{2p>\ell}\right),
$$
while the right-hand side is at most \eqref{amp}.  By \eqref{lint}, it suffices to show that
\begin{equation}\label{test}
\sum_{p_- \leq p \leq p_+} \frac{\mu_p}{p} \left( c_0 1_{p>\ell} + (1-c_0) 1_{2p>\ell}\right)
\geq
\sum_{0 \leq m \leq M}^* \frac{1}{2} \sum_{p_- \leq p \leq p_+} \min\left(\frac{1}{2^mp}, \frac{\alpha}{\ell}\right) \mu_p.
\end{equation}
Applying \eqref{mertens} and \Cref{osc-lemma}, we can bound the right-hand side of \eqref{test} by
$$ \left( 1 + O\left(\frac{1}{\log^{10} p_-}\right)\right) \frac{\log p_-}{2} \sum_{0 \leq m \leq M}^* \int_{p_-}^{p_+} \min\left(\frac{1}{2^m t}, \frac{\alpha}{\ell}\right)\ \frac{dt}{\log^2 t}.$$
The minimum $\min(\frac{1}{2^m t}, \frac{\alpha}{\ell})$ is equal to $\frac{1}{2^m t}$ for $t \geq \ell/(2^m \alpha)$ and $\frac{\alpha}{\ell}$ for $t < \ell/(2^m \alpha)$.  Routine estimation then gives
\begin{align*}
&\int_{p_-}^{p_+} \min(\frac{1}{2^m t}, \frac{\alpha}{\ell})\ \frac{dt}{\log^2 t}\\
&\quad \leq \frac{1}{2^m \log \frac{\ell}{\alpha 2^m}} - \frac{1}{2^m \log p_+} + \frac{1}{2^m \log^2 \frac{\ell}{\alpha 2^m}} + O_M\left( \frac{1}{\log^3 \ell}\right) \\
&\quad =  \frac{1}{2^m \log \ell} - \frac{1}{2^m \log p_+} + \frac{m \log 2 - \log \frac{1}{e\alpha}}{2^m \log^2 \ell} + O_M\left( \frac{1}{\log^3 \ell}\right).
\end{align*}
Performing the $m$ summation, we conclude that the right-hand side of \eqref{test} is at most
$$ (\log p_-) \left(\frac{1}{\log \ell} - \frac{1}{\log p_+} + \frac{\log 2 - \log \frac{1}{e\alpha}}{\log^2 \ell} + O_M\left( \frac{1}{\log^3 \ell}\right) \right).$$
Meanwhile, by \eqref{mertens}, the left-hand side of \eqref{test} can be computed to be
$$
\left( 1 + O\left(\frac{1}{\log^{10} p_-}\right)\right)
(\log p_-) \left( c_0 \sum_{\ell < p \leq p_+} \frac{1}{p \log p} + (1-c_0) \sum_{\ell/2 < p \leq p_+} \frac{1}{p \log p} \right).$$
Applying \Cref{osc-lemma} and evaluating the integrals, we can bound this by
$$
\left( 1 + O\left(\frac{1}{\log^{10} p_-}\right)\right)
(\log p_-) \left( c_0 \left( \frac{1}{\log \ell} - \frac{1}{\log p_+} \right) + (1-c_0) \left( \frac{1}{\log(\ell/2)} - \frac{1}{\log p_+} \right) \right)$$
which after routine Taylor expansion simplifies to
$$ (\log p_-) \left(\frac{1}{\log \ell} - \frac{1}{\log p_+} + \frac{\log 2 - c_0 \log 2}{\log^2 \ell} + O\left( \frac{1}{\log^3 \ell}\right) \right).$$
Since $\alpha < 1/e$, we can ensure that $c_0 \log 2 < \log\frac{1}{e\alpha}$ by taking $c_0$ small enough.  The claim \eqref{test} then follows.
\end{proof}

\section{The accounting equation}\label{accounting-sec}

Given a $t$-admissible multiset $\tuple$ (which we view as an approximate factorization of $N!$), we can apply the fundamental theorem of arithmetic \eqref{ftoa} to the rational number  $N!/\prod \tuple$ and rearrange to obtain the \emph{accounting equation}
\begin{equation}\label{accounting}
  \excess_t(\tuple) + \sum_p \nu_p\left( \frac{N!}{\prod \tuple} \right) \log p = \log N! - |\tuple| \log t
\end{equation}
where we define the \emph{$t$-excess} $\excess_t(\tuple)$ of the multiset $\tuple$ by the formula
\begin{equation}\label{excess-def}
  \excess_t(\tuple) \coloneqq \sum_{a \in \tuple} \log \frac{a}{t}.
\end{equation}

\begin{example} Suppose one wishes to factorize $5! = 2^3 \times 3 \times 5$.  The attempted $3$-admissible factorization $\tuple \coloneqq \{3,4,5,5\}$ has a $2$-surplus of $\nu_2(5!/\prod \tuple) = 1$, is in balance at $3$, and has a $5$-deficit of $\nu_5(\prod \tuple/5!) = 1$, so it is not a factorization or subfactorization of $5!$.  The $3$-excess of this multiset is
  $$ \excess_3(\tuple) = \log \frac{3}{3} + \log \frac{4}{3} + \log \frac{5}{3} + \log \frac{5}{3} = 1.3093\dots$$
  and the accounting equation \eqref{accounting} becomes
  $$ 1.3093\dots + \log 2 - \log 5 = 0.3930\dots = \log 5! - 4 \log 3.$$
  If one replaces one of the copies of $5$ in ${\mathcal B}$ with a $2$, this erases both the $2$-surplus and the $5$-deficit, and creates a factorization $\tuple' = \{2,3,4,5\}$ of $5!$; the $3$-excess now drops to
  $$ \excess_3(\tuple) = \log \frac{2}{3} + \log \frac{3}{3} + \log \frac{4}{3} + \log \frac{5}{3}  = 0.3930\dots,$$
  bringing the accounting equation back into balance.
\end{example}

In view of \Cref{subfac}, one can now equivalently describe
$t(N)$ as follows:

\begin{lemma}[Equivalent description of $t(N)$]\label{t-descrip}  $t(N)$ is the largest quantity $t$ for which there exists a $t$-admissible subfactorization of $N!$ with
$$ \excess_t(\tuple) + \sum_p \nu_p\left( \frac{N!}{\prod \tuple} \right) \log p \leq \log N! - N \log t.$$
\end{lemma}

One can view $\log N! - N\log t$ as an available ``budget'' that one can ``spend'' on some combination of $t$-excess and $p$-surpluses.  For $t$ of the form $t = N/e^{1+\delta}$ for some $\delta>0$, the budget can be computed using the Stirling approximation \eqref{stirling} to be $\delta N + O(\log N)$.  The non-negativity of the $t$-excess and $p$-surpluses recovers the trivial upper bound \eqref{trivial}; but note that any prime $p > \frac{t}{\lfloor \sqrt{t} \rfloor}$ must inevitably contribute at least $\log \frac{\lceil t/p\rceil}{t/p}$ to the $t$-excess if it is to appear in the multiset $\tuple$.  By pursuing this line of reasoning, one can obtain an alternate proof of \Cref{upper-crit}; see \cite[Lemma 2.1]{tao}.

\section{Modified approximate factorizations}\label{approx-sec}

In this section we present and then analyze an algorithm that starts with an \emph{approximate} factorization $\tuple^{(0)}$ of $N!$, which is $t$-admissible but omits all tiny primes, and is approximately in balance in small and medium primes, and attempts to ``repair'' this factorization to establish a lower bound of the form $t(N) \geq t$.

To describe the criterion for the algorithm to succeed, it will be convenient to introduce the following notation.
For $a_+,a_- \in [0,+\infty]$, we define the asymmetric norm $|x|_{a_+,a_-}$ of a real number $x$ by the formula
$$
|x|_{a_+,a_-} \coloneqq  \begin{cases}
  a_+ |x| & x\geq 0 \\
  a_- |x| & x\leq 0,
\end{cases}
$$
with the usual convention $+\infty \times 0 = 0$.
If $a_+,a_-$ are finite, this function is Lipschitz with constant $\max(a_+,a_-)$.  One can think of $a_+$ as the ``cost'' of making $x$ positive, and $a_-$ as the
``cost'' of making $x$ negative.

The analysis of the algorithm is now captured by the following proposition.

\begin{proposition}[Repairing an approximate factorization]\label{repair}  Let $N, K$ be natural numbers, and let $1 \leq t \leq N$ be an additional parameter obeying the conditions
\begin{equation}\label{conditions}
    \frac{t}{K} \geq \sqrt{N}; \qquad \frac{t}{K^2} \geq K \geq 5.
\end{equation}
We also assume that there are additional parameters $\kappa_* > 0$ and $0 \leq \gamma_2, \gamma_3 < 1$, such that there exist $3$-smooth numbers
\begin{equation}\label{tlip}
  t \leq 2^{n_2} 3^{m_2}, 2^{n_3} 3^{m_3} \leq e^{\kappa_*} t
\end{equation}
such that
\begin{equation}\label{nm}
  2m_2 \leq \gamma_2 n_2; \qquad n_3 \leq 2\gamma_3 m_3.
\end{equation}
We define the ``norm'' of a pair $n,m$ of real numbers by the formula
$$ \| (n,m) \|_\gamma \coloneqq \max\left( \frac{n-2\gamma_2 m}{1-\gamma_2}, \frac{2m-\gamma_3 n}{1-\gamma_3} \right).$$
Let $\tuple^{(0)}$ be a $t$-admissible multiset of natural numbers, with all elements of $\tuple^{(0)}$ at most $(t/K)^2$, and suppose that one has the inequalities
\begin{equation}\label{delta-cond}
\sum_{i=1}^8 \delta_i \leq \delta
\end{equation}
and
\begin{equation}\label{alpha-cond}
  \sum_{i=1}^7 \alpha_i \leq 1
\end{equation}
where
\begin{align}
\delta_1 &\coloneqq \frac{1}{N} \excess_t(\tuple^{(1)}) \label{delta1-def}  \\
\delta_2 &\coloneqq \frac{1}{N} \sum_{t/K < p \leq N} f_{N/t}(p/N) \label{delta2-def}  \\
\delta_3 &\coloneqq \frac{\kappa_{4.5}}{N}  \sum_{3 < p_1 \leq t/K} \left|\nu_{p_1}\left( \frac{N!}{\prod \tuple^{(0)}} \right)\right| \label{delta3-def}  \\
\delta_4 &\coloneqq \kappa_{4.5} \sum_{K < p_1 \leq t/K} A_{p_1} \label{delta4-def}  \\
\delta_5 &\coloneqq \kappa_{4.5} \sum_{3 < p_1 \leq K} |A_{p_1} - B_{p_1}|_{\frac{\log p_1}{\log(t/K^2)},1} \label{delta5-def}  \\
\delta_6 &\coloneqq \frac{\kappa_{4.5}}{N} \label{delta6-def}  \\
\delta_7 &\coloneqq \frac{\kappa_*}{\log t} \left( \log \sqrt{12} - B_2 \log 2 - B_3 \log 3\right)\label{delta7-def}  \\
\delta_8 &\coloneqq \frac{2(\log t + \kappa_*)}{N} \label{delta8-def}\\
\delta &\coloneqq \frac{1}{N} \log N! - \log t\label{delta-def}
\end{align}
\begin{align}
\alpha_1 &\coloneqq \frac{1}{N} \left\| \left( \nu_{2}\left(\prod \tuple^{(0)}\right), \nu_{3}\left(\prod \tuple^{(0)}\right)\right) \right\|_\gamma \label{alpha1-def}  \\
\alpha_2 &\coloneqq \left\| (B_2, B_3) \right\|_\gamma \label{alpha2-def}  \\
\alpha_3 &\coloneqq \frac{\log \frac{t}{K} + \kappa_{**}}{N\log \sqrt{12}} \sum_{3 < p_1 \leq t/K} \left|\nu_{p_1}\left( \frac{N!}{\prod \tuple^{(0)}} \right)\right| \label{alpha3-def}  \\
\alpha_4 &\coloneqq \frac{1}{\log \sqrt{12}} \sum_{K < p_1 \leq t/K} \left(\log \frac{t}{p_1} + \kappa_{**}\right) A_{p_1} \label{alpha4-def}  \\
\alpha_5 &\coloneqq \frac{1}{\log \sqrt{12}} \sum_{3 < p_1 \leq K} \left|A_{p_1} - B_{p_1}\right|_{\frac{\log p_1}{\log(t/K^2)} (\log K^2 + \kappa_{**}), \log p_1 + \kappa_{**}} \label{alpha5-def}  \\
\alpha_6 &\coloneqq \frac{\log t + \kappa_{**}}{N\log \sqrt{12}}  \label{alpha6-def}  \\
\alpha_7 &\coloneqq \max\left( \frac{\log(2N)}{(1-\gamma_2)N\log 2},  \frac{\log(3N)}{(1-\gamma_3)N\log \sqrt{3}}\right)\label{alpha7-def}
\end{align}
\begin{align}
\kappa_{**} &\coloneqq \max(\kappa^{(2)}_{4.5, \gamma_2}, \kappa^{(3)}_{4.5, \gamma_3}) \label{kappastar-def}  \\
A_{p_1} &\coloneqq \frac{1}{N} \sum_m \nu_{p_1}(m) | \{ a \in \tuple^{(0)}: a = mp \hbox{ for a prime } p > t/K \}| \label{A-def} \\
B_{p_1} &\coloneqq \frac{1}{N} \sum_{m \leq K} \nu_{p_1}(m) \sum_{\frac{t}{m} \leq p < \frac{t}{m-1}} \left \lfloor \frac{N}{p} \right\rfloor,\label{B-def}
\end{align}
with the convention that the upper bound $p < \frac{t}{m-1}$ in \eqref{B-def} is vacuous when $m=1$.  Here:
\begin{itemize}
  \item The set ${\mathcal B}^{(1)}$ is defined by applying steps (a) and (b) below.
  \item The excess $\excess_t$ is defined in \eqref{excess-def}.
  \item The function $f_{N/t}$ is defined in \eqref{falpha-def}.
  \item The constant $\kappa_{4.5}$ is given by \eqref{kappa-45}.
\end{itemize}
Then $t(N) \geq t$.
\end{proposition}

In practice, the parameter $K$ will be quite small compared to $N$, and the quantities $\gamma_2, \gamma_3, \kappa_{*}$ will also be somewhat smaller than $1$.

\begin{remark} In the notation of this proposition, \Cref{upper-crit} can essentially be interpreted as a necessary condition $\delta_2 \leq \delta$ for $t(N) \leq t$ to be provable; to use the above proposition effectively, it is thus desirable to have all the other $\delta_i, i \neq 2$ terms be as small as possible.  The criterion in \Cref{t-descrip} can similarly be rewritten as $\delta_1+\delta_9 \leq \delta$, where
  $$ \delta_9 \coloneqq \frac{1}{N} \sum_p \left|\nu_{p_1}\left( \frac{N!}{\prod \tuple^{(0)}} \right)\right|_{\log p, \infty}.$$
In practice, $\delta_9$ is too large (or infinite) for this criterion to be directly useful; the algorithm below is intended to replace this large quantity with something much smaller, and in particular to utilize tiny primes to gain factors such as $\kappa_L$ for various $L$ in the bounds of the main $\delta_i$ terms besides the ``non-negotiable'' $\delta_2$.  The secondary condition \eqref{alpha-cond} can be interpreted as a requirement that ``enough'' tiny primes are available in the factorization of $N!$ to perform such adjustments.
\end{remark}

The rest of this section will be devoted to the proof of this proposition.  It will be convenient to divide the primes into four classes:
  \begin{itemize}
    \item \emph{Tiny primes} $p=2,3$.
    \item \emph{Small primes} $3 < p \leq K$.
    \item \emph{Medium primes} $K < p \leq t/K$.
    \item \emph{Large primes} $p > t/K$.
    \end{itemize}
Initially, the multiset $\tuple^{(0)}$ may have the ``wrong'' number of factors at large primes.  We fix this by applying the following modifications to $\tuple^{(0)}$:
\begin{itemize}
  \item[(a)] Remove all elements of $\tuple^{(0)}$ that are divisible by a large prime $p > t/K$ from the multiset.
  \item[(b)] For each large prime $p > t/K$, add $\nu_p(N!)$ copies of $p \lceil t/p \rceil$ to the multiset.
\end{itemize}
We let $\tuple^{(1)}$ be the multiset formed after completing both Step (a) and Step (b).  We make two simple observations:
\begin{itemize}
\item[(A)] Since the elements of $\tuple^{(0)}$ are at most $(t/K)^2$, all the elements removed in Step (a) are of the form $mp$ where $m \leq t/K$.
\item[(B)] For each large prime $p$ considered in Step (b),  one has $\nu_p(N!) = \lfloor N/p \rfloor$ by \eqref{legendre} and \eqref{conditions}, while $\lceil t/p \rceil \leq K \leq t/K$ (again by \eqref{conditions}).
\end{itemize}
From this, we see that $\tuple^{(1)}$ is automatically $t$-admissible, and is in balance at any large prime $p > t/K$:
$$ \nu_p\left(\frac{N!}{\prod \tuple^{(1)}}\right) = 0.$$
For medium primes $K < p_1 \leq t/K$, one can have some increase in the $p_1$-surplus coming from Step (a), which is described by \eqref{A-def}:
$$ \nu_{p_1}\left(\frac{N!}{\prod \tuple^{(1)}}\right) = \nu_{p_1}\left(\frac{N!}{\prod \tuple^{(0)}}\right) + NA_{p_1}.$$
For small or tiny primes $p \leq K$, one also has some possible decrease in the $p_1$-surplus coming from Step (b), which is described by \eqref{B-def}:
$$ \nu_{p_1}\left(\frac{N!}{\prod \tuple^{(1)}}\right) = \nu_{p_1}\left(\frac{N!}{\prod \tuple^{(0)}}\right) + N(A_{p_1} - B_{p_1}).$$
In particular, we have from \eqref{alpha1-def}, \eqref{alpha2-def} and the triangle inequality that
\begin{equation}\label{alpha-1}
\frac{1}{N} \left\|\left( \nu_2\left(\prod \tuple^{(1)}\right), \nu_3\left(\prod \tuple^{(1)}\right)\right)\right\|_\gamma \leq \alpha_1 + \alpha_2.
\end{equation}

Each element removed in Step (a) reduces the $t$-excess, while each element $p \lceil t/p \rceil$ added in Step (b) increases the $t$-excess by $\log \frac{\lceil t/p \rceil}{t/p}$, so each large prime $t/K < p \leq N$ contributes a net of $\lfloor \frac{N}{p} \rfloor \log \frac{\lceil t/p \rceil}{t/p} = f_{N/t}(p/N)$ to the $t$-excess.  Thus by \eqref{delta1-def}, \eqref{delta2-def} we have
\begin{equation}\label{excess-1}
  \frac{1}{N} \excess_t(\tuple^{(1)}) \leq \delta_1 + \delta_2.
\end{equation}

Now we bring the multiset $\tuple^{(1)}$ into balance at small and medium primes $3 < p \leq t/K$.  We make the following observations:
\begin{itemize}
\item[(C)]  If an element in $\tuple^{(1)}$ is divisible by some small or medium prime $3 < p \leq t/K$, and one replaces $p$ by $\lceil p \rceil^{\langle 2,3 \rangle}_{4.5}$ in the factorization of that element, then the $p$-deficit decreases by one, while (by \Cref{power-lemma}) the $t$-excess increases by at most $\kappa_{4.5}$, and the quantity $\| (\nu_2(\prod \tuple^{(1)}),\nu_3(\prod \tuple^{(1)}))\|_\gamma$ increases by at most $\frac{\log p + \kappa_{**}}{\log \sqrt{12}}$.  All other $p_1$-surpluses or $p_1$-deficits for $p_1 \neq 2,3,p$ remain unaffected.
\item[(D)]  If one adds an element of the form $m \lceil t/m \rceil^{\langle 2,3 \rangle}_{4.5}$ to $\tuple^{(1)}$ for some $m \leq t/K$ that is the product of small or medium primes $3 < p \leq t/K$, then the $p$-surpluses at small or medium primes $p$ decrease by $\nu_p(m)$, while (by \Cref{power-lemma}) the $t$-excess increases by at most $\kappa_{4.5}$, and the quantity $\| (\nu_2(\frac{N!}{\prod \tuple^{(1)}}), \nu_3(\frac{N!}{\prod \tuple^{(1)}}))\|_\gamma$ increases by at most $\frac{\log(t/m) + \kappa_{**}}{\log \sqrt{12}}$.  The $p$-surpluses or $p$-deficits at medium or large primes remain unaffected.
\end{itemize}

With these observations in mind, we perform the following modifications to the multiset $\tuple^{(1)}$.
\begin{itemize}
\item[(c)] If there is a $p_1$-deficit $\nu_{p_1}(\prod \tuple^{(1)}/N!) > 0$ at some small or medium prime $3 < p_1 \leq t/K$, then we perform the replacement of $p_1$ in one of the elements of $\tuple^{(1)}$ with $\lceil p_1 \rceil^{\langle 2,3\rangle}_{4.5}$ as per observation (C), repeated $\nu_{p_1}(\prod \tuple^{(1)}/N!)$ times, in order to eliminate all such deficits.
\item[(d)] If there is a $p$-surplus $\nu_p(\prod N!/\tuple^{(1)}) > 0$ at some medium prime $K < p \leq t/K$, we add the element
$p \lceil t/p \rceil^{\langle 2,3 \rangle}_{4.5}$ to $\tuple^{(1)}$ as per observation (D), $\nu_p(\prod N!/\tuple^{(1)})$ times, in order to eliminate all such surpluses at medium primes.
\item[(d')] If there are $p$-surpluses $\nu_p(\prod N!/\tuple^{(1)}) > 0$ at some small primes $3 < p \leq K$, we multiply all these primes together, then apply the greedy algorithm to factor them into products $m$ in the range $t/K^2 < m \leq t/K$, plus at most one exceptional product in the range $1 < m \leq t/K$.  For each of these $m$, add $m \lceil t/m \rceil^{\langle 2,3 \rangle}_{4.5}$ to $\tuple^{(1)}$ as per observation (D), to eliminate all such surpluses at small primes.
\end{itemize}
Let $\tuple^{(2)}$ be the multiset formed from $\tuple^{(1)}$ as the outcome of applying Steps (c), (d), and (d').  The product of all the primes arising in Step (d') has logarithm equal to
$$ \sum_{3 < p_1 \leq K} \left|\nu_{p_1}\left( \frac{N!}{\prod \tuple^{(1)}} \right) \right|_{\log p_1,0} = \sum_{3 < p_1 \leq K} \left|\nu_p\left( \frac{N!}{\prod \tuple^{(0)}} \right) \right|_{\log p_1,0}$$
and hence the number of non-exceptional $m$ arising in (d') is at most
$$ \sum_{3 < p_1 \leq K} \left|\nu_p\left( \frac{N!}{\prod \tuple^{(0)}} \right) \right|_{\frac{\log p_1}{\log(t/K^2)},0}.$$
The total excess of $\tuple^{(2)}$ is increased in Step (c) by at most
$$ \kappa_{4.5} \sum_{3 < p_1 \leq t/K} \left|\nu_{p_1}\left( \frac{N!}{\prod \tuple^{(1)}} \right) \right|_{0,1}
=  \kappa_{4.5} \sum_{3 < p_1 \leq t/K} \left|\nu_{p_1}\left( \frac{N!}{\prod \tuple^{(0)}} \right) + N(A_{p_1} - B_{p_1}) \right|_{0,1},$$
in Step (d) by at most
$$ \kappa_{4.5} \sum_{K < p_1 \leq t/K} \left|\nu_{p_1}\left( \frac{N!}{\prod \tuple^{(1)}} \right) \right|_{1,0}
=  \kappa_{4.5} \sum_{K < p_1 \leq t/K} \left|\nu_{p_1}\left( \frac{N!}{\prod \tuple^{(0)}} \right) + NA_{p_1} \right|_{1,0},$$
and in Step (c) by at most
$$ \kappa_{4.5} \left(1 + \sum_{3 < p_1 \leq K} \left|\nu_{p_1}\left( \frac{N!}{\prod \tuple^{(0)}} \right) \right|_{\frac{\log p_1}{\log(t/K^2)},0}\right).$$
From the triangle inequality and \eqref{excess-1}, \eqref{delta3-def}, \eqref{delta4-def}, \eqref{delta5-def}, \eqref{delta6-def}, we then have
\begin{equation}\label{excess-2}
   \frac{1}{N} \excess_t(\tuple^{(2)}) \leq \sum_{i=1}^6 \delta_i.
\end{equation}
Similarly, the quantity $\frac{1}{N} \| (\nu_2(\prod \tuple^{(1)}),\nu_3(\prod \tuple^{(1)}))\|_\gamma$ is increased in Step (c) by at most
$$\frac{1}{N\log \sqrt{12}} \sum_{3 < p_1 \leq t/K} \left|\nu_{p_1}\left( \frac{N!}{\prod \tuple^{(0)}} \right) + N(A_{p_1} - B_{p_1}) \right|_{0,\log p_1 + \kappa_{**}},$$
in Step (d) by at most
$$\frac{1}{N\log \sqrt{12}} \sum_{K < p_1 \leq t/K} \left|\nu_{p_1}\left( \frac{N!}{\prod \tuple^{(0)}} \right) + NA_{p_1} \right|_{\log(t/p_1) + \kappa_{**},0},$$
and in Step (d') by at most the sum of
$$\frac{1}{N\log \sqrt{12}} \sum_{3 < p_1 \leq K} \left|\nu_{p_1}\left( \frac{N!}{\prod \tuple^{(0)}} \right) + N(A_{p_1}-B_{p_1}) \right|_{\log(K^2) + \kappa_{**},0}$$
and
$$\frac{1}{N\log \sqrt{12}} \left( \log t + \kappa_{**} \right)$$
so by \eqref{alpha-1}, \eqref{alpha3-def}, \eqref{alpha4-def}, \eqref{alpha5-def}, \eqref{alpha6-def}, and the triangle inequality we have
\begin{equation}\label{alpha-2}
\frac{1}{N} \| (\nu_2(\prod \tuple^{(2)}), \nu_3(\prod \tuple^{(2)})) \|_\gamma \leq \sum_{i=1}^6 \alpha_i.
\end{equation}

By construction, the multiset $\tuple^{(2)}$ is $t$-admissible, and is in balance at all small, medium, and large primes $p > 3$; thus $N!/\prod \tuple^{(2)} = 2^n 3^m$ for some integers $n,m$.  From \eqref{alpha-2}, \eqref{alpha-cond}, \eqref{legendre}, \eqref{alpha7-def} we have
\begin{align*}
n - 2\gamma_2 m &= \nu_2(N!) - 2\gamma_2 \nu_3(N!) - \left(\nu_2\left(\prod \tuple^{(2)}\right) - 2\gamma_2 \nu_3\left(\prod \tuple^{(2)}\right)\right) \\
&\geq \nu_2(N!) - 2\gamma_2 \nu_3(N!) - N (1-\gamma_2) \sum_{i=1}^6 \alpha_i \\
&> N -
 \frac{\log N}{\log 2} - 1 - \gamma_2 N - N (1-\gamma_2) (1 - \alpha_7) \\
&= N (1-\gamma_2) \alpha_7 - \frac{\log(2N)}{\log 2} \\
&\geq 0
\end{align*}
and similarly
\begin{align*}
  2m - \gamma_3 n &= 2\nu_3(N!) - \gamma_3 \nu_2(N!) - \left(2\nu_3\left(\prod \tuple^{(2)}\right) - \gamma_3 \nu_2\left(\prod \tuple^{(2)}\right)\right) \\
&\geq 2\nu_3(N!) - \gamma_3 \nu_2(N!) - N (1-\gamma_3) \sum_{i=1}^6 \alpha_i \\
&> N - \frac{\log N}{\log \sqrt{3}} - 2 - \gamma_3 N - N (1-\gamma_3) (1 - \alpha_7) \\
&= N (1-\gamma_3) \alpha_7 - \frac{\log(3N)}{\log \sqrt{3}} \\
&\geq 0.
\end{align*}
From \eqref{nm} and Cramer's rule we conclude that $(n,2m)$ lies in the non-negative linear span of $(n_2, 2m_2)$, $(n_3, 2m_3)$, thus
\begin{equation}\label{comb}  (n,2m) = \beta_2 (n_2,2m_2) + \beta_3 (n_3,2m_3)\end{equation}
for some reals $\beta_2,\beta_3 \geq 0$.
We now create the multiset $\tuple^{(3)}$ by adding $\lfloor \beta_2 \rfloor$ copies of $2^{n_2} 3^{m_2}$ and $\lfloor \beta_3 \rfloor$ copies of $2^{n_3} 3^{m_3}$ to $\tuple^{(2)}$.  By \eqref{tlip}, this multiset remains $t$-admissible, and each element added increases the $t$-excess by at most $\kappa_*$.  The number of such elements can be upper bounded using \eqref{comb}, \eqref{legendre} as
\begin{align*}
  \lfloor \beta_2 \rfloor + \lfloor \beta_3 \rfloor &\leq \beta_2 + \beta_3 \\
  &\leq \frac{1}{\log t} \left( \beta_2 (n_2 \log 2 + m_2 \log 3) + \beta_3 (n_3 \log 2 + m_3 \log 3) \right) \\
  &= \frac{1}{\log t} (n \log 2 + m \log 3) \\
  &\leq \frac{1}{\log t} ((\nu_2(N!)-NB_2) \log 2 + (\nu_3(N!)-NB_3) \log 3) \\
  &\leq \frac{1}{\log t} \left(N \log 2 + \frac{N}{2} \log 3 - NB_2 \log 2 - NB_3 \log 3\right) \\
  &= \frac{N \log \sqrt{12}}{\log t} - \frac{N(B_2 \log 2 + B_3 \log 3)}{\log t}.
\end{align*}
By \eqref{excess-2}, \eqref{delta7-def}, we thus have
\begin{equation}\label{excess-3}
 \frac{1}{N} \excess_t(\tuple^{(3)}) \leq \sum_{i=1}^7 \delta_i.
\end{equation}
Meanwhile by construction we see that $\tuple^{(3)}$ is a subfactorization of $N!$ that is in balance at all non-tiny primes, with tiny prime surpluses bounded by
$$ \nu_2\left( \frac{N!}{\prod \tuple^{(3)}}\right) \leq n_2+n_3; \quad \nu_3\left( \frac{N!}{\prod \tuple^{(3)}}\right) \leq m_2+m_3.$$
and thus by \eqref{tlip}, \eqref{delta8-def}, we thus have
$$ \frac{1}{N} \sum_p \nu_p\left( \frac{N!}{\prod \tuple^{(3)}}\right) \log p \leq \frac{\log 2^{n_2} 3^{m_2} + \log 2^{n_3} 3^{m_3}}{N} \leq \delta_8$$
and thus by \eqref{excess-3}, \eqref{delta-cond} we have
$$ \excess_t(\tuple^{(3)}) + \sum_p \nu_p\left( \frac{N!}{\prod \tuple^{(3)}}\right) \log p \leq \log N! - N \log t.$$
Applying \Cref{t-descrip}, we conclude that $t(N) \geq t$ as claimed.

\FloatBarrier

\section{Estimating terms}\label{construction-sec}

In order to use \Cref{repair} for a given choice of $N,t$, we need to find a $t$-admissible multiset $\tuple^{(0)}$ and parameters $K, \kappa_*, \gamma_2, \gamma_3$ obeying \eqref{conditions} as well as good upper bounds on the quantities $\delta_i$, $i=1,\dots,8$ and $\alpha_i$, $i=1,\dots,7$, that can either be evaluated asymptotically or numerically.  Many of these terms will be straightforward to estimate; we discuss only the more difficult ones here.

We introduce a further natural number parameter $A$ and define
\begin{equation}\label{sigma-def}
  \sigma \coloneqq \frac{3N}{At}.
\end{equation}
We let $\tuple^{(0)}$ be the multiset of $3$-rough elements of the interval $(t, t(1+\sigma)]$, with each element repeated precisely $A$ times.  This is clearly $t$-admissible.  It has no presence at tiny primes, so
\begin{equation}\label{alpha1-vanish}
  \alpha_1 = 0.
\end{equation}
We will also introduce an auxiliary parameter $L$ to assist us with the estimates.  The influence of the parameters $A,K,L$ on the other parameters $\delta_i, \alpha_i$ (and $\gamma_2,\gamma_3,\kappa_{**}$) can be roughly summarized as follows:
\begin{itemize}
\item $\gamma_2, \gamma_3 \asymp \log L / \log N$; assuming this quantity is small enough, we have $\kappa_{**} \asymp 1$.
\item $\delta_1 \asymp 1/A$.
\item $\delta_2 \asymp 1/\log N$.
\item $\delta_3 \asymp A/K\log N$ and $\alpha_3 \asymp A/K$.
\item $\delta_5 \asymp \log^{O(1)} K/\log^2 N$ and $\alpha_2, \alpha_5 \asymp \log^{O(1)} K / \log N$.
\item $\delta_7 \asymp \kappa_L/\log N$.
\item $\delta_4, \delta_6, \delta_8, \alpha_4, \alpha_6, \alpha_7$ will be lower order terms.
\end{itemize}
We will quantify these relationships more precisely below, but they already suggest that one should take $A$ to only be moderately large (e.g., of logarithmic size), that $K$ should only be slightly larger than $A$, and that $L$ should be significantly smaller than $N$.

Using \Cref{lit}, we may estimate $\delta_1$:

\begin{lemma}\label{delta1-bound} We have
$$ \delta_1 \leq \frac{3N}{2tA} + \frac{4}{N}.$$
\end{lemma}

\begin{proof}  By definition, we have
$$ \excess_t(\tuple^{(1)}) = A \sum^*_{t < n \leq t(1+\sigma)} \log \frac{n}{t}.$$
By the fundamental theorem of calculus, this is
$$ A \int_0^{t\sigma} \sum^*_{t < n \leq t+h} 1\ \frac{dh}{t+h}.$$
Bounding $\frac{1}{t+h}$ by $\frac{1}{t}$ and applying \Cref{lit}, \eqref{sigma-def}, we conclude that
$$
 \excess_t(\tuple^{(1)}) \leq A \int_0^{3N/A} \left(\frac{h}{3} + \frac{4}{3}\right) \frac{dh}{t} = \frac{3N^2}{2tA} + 4.
$$
and the claim follows.
\end{proof}

To construct $\gamma_2, \gamma_3, \kappa_*, n_2, m_2, n_3, m_3$, we introduce another parameter $L \geq 1$ and assume that
\begin{equation}\label{t-lower}
  t > 3L.
\end{equation}
We define $n_2,n_3,m_2,n_3$ by setting
$$
2^{n_2} 3^{m_2} \coloneqq 2^{n_0} \lceil t/2^{n_0} \rceil^{\langle 2,3 \rangle}; \quad
2^{n_3} 3^{m_3} \coloneqq 3^{m_0} \lceil t/3^{m_0} \rceil^{\langle 2,3 \rangle}$$
where $2^{n_0}, 3^{m_0}$ are the largest powers of $2,3$ respectively that are at most $t/L$.  By construction and \eqref{kappa-def}, \eqref{tlip} holds with
\begin{equation}\label{kappas-eq}
  \kappa_* = \kappa_L.
\end{equation}
We have
$$  2m_2 \leq \frac{\log \lceil t/2^{n_0} \rceil^{\langle 2,3 \rangle}}{\log \sqrt{3}}  \leq \frac{\log(2L) + \kappa_L}{\log \sqrt{3}}
$$
and
$$n_2 \geq n_0 \geq \frac{\log t - \log(2L)}{\log 2};$$
similarly
$$ n_3 \leq \frac{\log(3L)+\kappa_L}{\log 2}$$
and
$$ 2m_3 \geq \frac{\log t - \log(3L)}{\log \sqrt{3}}.$$
We conclude that \eqref{nm} holds with
\begin{equation}\label{gammas-def}
\begin{split}
\gamma_2 &\coloneqq \frac{\log 2}{\log \sqrt{3}} \frac{\log(2L) + \kappa_L}{\log t - \log(2L)} \\
\gamma_3 &\coloneqq \frac{\log \sqrt{3}}{\log 2} \frac{\log(3L) + \kappa_L}{\log t - \log(3L)};
\end{split}
\end{equation}
one can of course also take larger values of $\gamma_2,\gamma_3$ if desired.
This lets us compute the quantity $\kappa_{**}$ defined in \eqref{kappastar-def}.

To estimate $\delta_3, \alpha_3$ we use

\begin{lemma}\label{val-bound} For every $3 < p \leq t/K$, one has
\begin{equation}\label{nup}
  \nu_p\left(\frac{N!}{\prod \tuple^{(1)}}\right) =
O_{\leq}\left(\frac{4A+3}{3} \left\lceil \frac{\log N}{\log p}  \right\rceil\right).
\end{equation}
\end{lemma}

\begin{proof}
One has
  \begin{align*}
    \nu_p(\prod \tuple^{(1)}) &= A \sum_{t < n \leq t(1+\sigma)}^* \nu_p(n) \\
    &= A \sum_{1 \leq j \leq \frac{\log N}{\log p}} \sum_{t/p^j < n \leq t(1+\sigma)/p^j}^* 1 \\
    &= A \sum_{1 \leq j \leq \frac{\log N}{\log p}} \left(\frac{N}{p^j A} + O_{\leq}(4/3)\right) \\
    &= \frac{N}{p-1} - O_{\leq}^+\left(\frac{1}{p-1}\right)
    + O_{\leq}\left(\frac{4A}{3} \left\lceil \frac{\log N}{\log p}  \right\rceil\right) \\
    &= \frac{N}{p-1}
    - O_{\leq}^+\left(\left\lceil \frac{\log N}{\log p}  \right\rceil\right)
    + O_{\leq}\left(\frac{4A}{3} \left\lceil \frac{\log N}{\log p}  \right\rceil\right).
  \end{align*}
  Meanwhile, from \eqref{legendre} one has
  $$ \nu_p(N!) = \frac{N}{p-1} - O_{\leq}^+\left(\left\lceil \frac{\log N}{\log p}  \right\rceil\right)$$
and the claim follows.
\end{proof}

\begin{corollary}\label{delta3-alpha3-bound} One has
$$ \delta_3 \leq \frac{(4A+3)\kappa_{4.5}}{3N} \left(\pi\left(\frac{t}{K} \right) + \frac{\log N}{\log 5} \pi\left(\sqrt{N}\right)\right)$$
and
$$ \alpha_3 \leq \frac{(4A+3) \left(\log \frac{t}{K} + \kappa_{**}\right)}{3N\log \sqrt{12}}  \left(\pi\left(\frac{t}{K} \right) + \frac{\log N}{\log 5} \pi\left(\sqrt{N}\right)\right).$$
\end{corollary}

\begin{proof} This is immediate from \Cref{val-bound} and \eqref{alpha3-def}, \eqref{delta3-def} after noting that $\lfloor \frac{\log N}{\log p} \rfloor \leq 1 + \frac{\log N}{\log 5} 1_{p \leq \sqrt{N}}$ for $3 < p \leq t/K$.
\end{proof}

The main quantities left to estimate are the quantities $\delta_4, \delta_5, \alpha_4, \alpha_5$ that involve $A_{p_1}$.  By construction of $\tuple^{(0)}$, we have
$$
A_{p_1} = \frac{1}{N} \sum_m^* \nu_{p_1}(m) \sum_{\frac{t}{K}, \frac{t}{m} < p \leq \frac{t(1+\sigma)}{m}} A.
$$
In particular, for $p > K(1+\sigma)$ the quantity $A_{p_1}$ vanishes entirely:
\begin{equation}\label{ap1-vanish}
  A_{p_1} = 0.
\end{equation}
For the remaining primes $3 < p \leq K(1+\sigma)$ one has
\begin{equation}\label{ap1-small}
A_{p_1} = \frac{A}{N} \sum_{m \leq K(1+\sigma)}^* \nu_{p_1}(m) \left( \pi\left(\frac{t(1+\sigma)}{m}\right) - \pi\left(\frac{t}{\min(m,K)} \right) \right).
\end{equation}
In practice, these expressions can be adequately controlled by \Cref{osc-lemma}, as can the quantities $B_{p_1}$.

\section{The asymptotic regime}

With the above estimates, we can now establish the lower bound in \Cref{main}(iv).  Thus we aim to show that $t(N) \geq t$ for sufficiently large $N$, where
\begin{equation}\label{main-lower}
   t \coloneqq \frac{N}{e} - \frac{c_0 N}{\log N} + \frac{N}{\log^{1+c} N} \asymp N
\end{equation}
and $0 < c < 1$ is a small absolute constant. We use the construction of $\tuple^{(0)}$ and  $K, \kappa_*, \gamma_2, \gamma_3$ from the previous section with the parameters
\begin{align}
  A &\coloneqq \lfloor \log^2 N \rfloor\label{a-asym}\\
K &\coloneqq \lfloor \log^3 N \rfloor \label{k-asym} \\
L &\coloneqq N^{0.1},\label{l-asym}
\end{align}
so from \eqref{sigma-def} one has
\begin{equation}\label{sigma-alt}
   \sigma = \frac{3N}{tA} \asymp \frac{1}{A} \asymp \frac{1}{\log^2 N}.
\end{equation}
The conditions \eqref{conditions}, \eqref{t-lower} are easily verified for $N$ large enough.

By \eqref{kappas-eq}, \eqref{l-asym}, and \Cref{power-lemma}(ii) we have
\begin{equation}\label{kappas-decay}
  \kappa_* \ll \log^{-c'} N
\end{equation}
for some absolute constant $c'>0$ (we can assume $c$ is smaller than $c'$).  From \eqref{gammas-def}, \eqref{main-lower}, \eqref{l-asym} we have
$$ \gamma_2 = \frac{1}{10} \frac{\log 2}{\log \sqrt{3}} + O\left(\frac{1}{\log N}\right), \quad
 \gamma_3 = \frac{1}{10} \frac{\log \sqrt{3}}{\log 2} + O\left(\frac{1}{\log N}\right)
$$
and hence by \eqref{kappastar-def},  \eqref{kappastar-2-def}, \eqref{kappastar-3-def} we have for sufficiently large $N$ that
$$ \kappa_{**} \ll 1.$$
By \Cref{repair}, it thus suffices to establish the inequalities \eqref{delta-cond}, \eqref{alpha-cond}.
Several of the quantities $\delta, \delta_i, \alpha_i$ can now be immediately estimated using \eqref{alpha1-vanish}, \eqref{delta1-bound}, \Cref{delta3-alpha3-bound}, \eqref{stirling}, \eqref{kappas-decay}, and the prime number theorem:
\begin{align*}
  \delta_1 &\ll \frac{1}{A} \asymp \frac{1}{\log^2 N} \\
  \delta_3 &\ll \frac{A}{K \log N} \asymp \frac{1}{\log^2 N} \\
  \delta_6 &\ll \frac{1}{N} \\
  \delta_7 &\ll \frac{\kappa_*}{\log N} \ll \frac{1}{\log^{1+c'} N} \\
  \delta_8 &\ll \frac{\log N}{N} \\
  \delta &= \frac{ec_0}{\log N} + \frac{e}{\log^{1+c} N} + O\left( \frac{1}{\log^2 N} \right)\\
\end{align*}
\begin{align*}
  \alpha_1 &= 0 \\
  \alpha_3 &\ll \frac{A}{K} \asymp \frac{1}{\log N} \\
  \alpha_6, \alpha_7 &\ll \frac{\log N}{N} \\
\end{align*}
On the interval $(t/NK,1]$, the function $f_{N/t}$ is piecewise monotone with $O(K)$ pieces, and bounded by $1$, so its augmented total variation norm is $O(K)$.  Applying \eqref{delta2-def} and \Cref{osc-lemma} (with classical error term), we have
\begin{align*}
\delta_2 &\leq \frac{1}{\log(t/K)} \int_{t/NK}^1 f_{N/t}(x)\ dx + O\left( \frac{1}{\log^2 N} \right) \\
&\leq \frac{1}{\log N} \int_{1/eK}^{N/et} f_{N/t}(etx/N)\ dx + O\left( \frac{1}{\log^2 N} \right)
\end{align*}
where we have used \eqref{falpha-bound} to manage error terms.
Similarly to the proof of \Cref{upper-bound}, the function $f_{N/t}(etx/N)$ differs from $f_e(x)$ outside of an exceptional set of measure $O(1/\log N)$, and hence by \eqref{c0-def} (and \eqref{falpha-bound}) we have
$$ \delta_2 \leq \frac{ec_0}{\log N} + O\left( \frac{1}{\log^2 N} \right).$$
To finish the verification of the conditions \eqref{delta-cond}, \eqref{alpha-cond}, it will suffice to show that
\begin{equation}\label{delta-remaining}
\delta_4, \delta_5 \ll \frac{(\log\log N)^{O(1)}}{\log^2 N}
\end{equation}
and
\begin{equation}\label{alpha-remaining}
\alpha_2, \alpha_4, \alpha_5 \ll \frac{(\log\log N)^{O(1)}}{\log N}.
\end{equation}
By Mertens' theorem (or \Cref{osc-lemma}) and \eqref{delta4-def}, \eqref{delta5-def}, \eqref{alpha2-def}, \eqref{alpha4-def}, \eqref{alpha5-def}, \eqref{sigma-alt}, it suffices to show that
\begin{equation}\label{ap1-bound}
A_{p_1}, B_{p_1} \ll \frac{(\log\log N)^{O(1)}}{p_1 \log N}
\end{equation}
for all $p_1 \leq K(1+\sigma)$ (recalling from \eqref{ap1-vanish} that $A_{p_1}$ vanishes for any larger $p_1$), as well as the variant
\begin{equation}\label{ap1-diff}
  |A_{p_1}-B_{p_1}|_{0,1} \ll \frac{(\log\log N)^{O(1)}}{p_1 \log^2 N}
\end{equation}
for $3 < p_1 \leq K$.

For \eqref{ap1-bound} we use \eqref{ap1-small}, \eqref{B-def}, and the crude bound
\begin{equation}\label{crude}
  \nu_{p_1}(m) \ll 1_{p_1|m} \log\log N
\end{equation}
 for $m \leq K(1+\sigma)$, and reduce to showing that
$$ \frac{A}{N} \sum_{m \leq K(1+\sigma)} 1_{p_1|m} \left( \pi\left(\frac{t(1+\sigma)}{m}\right) - \pi\left(\frac{t}{\min(m,K)} \right) \right)
\ll \frac{(\log\log N)^{O(1)}}{p_1 \log N}$$
and
$$ \frac{1}{N} \sum_{m \leq K} 1_{p_1|m} \sum_{\frac{t}{m} \leq p < \frac{t}{m-1}} \left \lfloor \frac{N}{p} \right\rfloor \ll
\frac{(\log\log N)^{O(1)}}{p_1 \log N}.$$
But from the Brun--Titchmarsh inequality (or \Cref{osc-lemma}) and \eqref{sigma-alt} one has
$$ \pi\left(\frac{t(1+\sigma)}{m}\right) - \pi\left(\frac{t}{\min(m,K)}\right) \ll \frac{t\sigma}{m \log N} \ll \frac{N}{Am \log N}$$
and
$$ \sum_{\frac{t}{m} \leq p < \frac{t}{m-1}} \left \lfloor \frac{N}{p} \right\rfloor \ll \frac{tm}{m^2 \log N} \ll \frac{N}{m \log N}$$
and the claim then follows from summing the harmonic series.

It remains to show \eqref{ap1-diff}.  If $3 < p_1 \leq K$, then from \eqref{ap1-small}, \eqref{sigma-alt}, \eqref{crude} and \Cref{osc-lemma} (with classical error term) we have
\begin{align*}
A_{p_1} &\geq \frac{1}{N} \sum_{m \leq K(1+\sigma)}^* \nu_{p_1}(m) \left( \frac{A t \sigma}{m \log N} + O\left( \frac{(\log\log N)^{O(1)}A t \sigma}{m \log^2 N} \right) \right) \\
&= \frac{1}{\log N} \sum_{m \leq K(1+\sigma)}^* \nu_{p_1}(m) \frac{3}{m}+ O\left( \frac{(\log\log N)^{O(1)}}{\log^2 N} \right) \\
&= \frac{1}{\log N} \sum_{m \leq K}^* \nu_{p_1}(m) \frac{3}{m}+ O\left( \frac{(\log\log N)^{O(1)}}{\log^2 N} \right)
\end{align*}
and similarly from \eqref{B-def}, \eqref{crude}, and \Cref{osc-lemma} (again with classical error term)
\begin{align*}
  B_{p_1} &\leq
  \frac{1}{N} \sum_{m \leq K} \nu_{p_1}(m) \sum_{\frac{t}{m} \leq p < \frac{t}{m-1}} \frac{N}{p}  \\
  &\leq \frac{1}{N} \sum_{m \leq K} \nu_{p_1}(m) \left( \frac{N}{\log(t/m)} \int_{t/m}^{t/(m-1)} \frac{dx}{x} + O\left( \frac{N}{\log^{10} N} \right) \right) \\
  &\leq \frac{1}{\log N} \sum_{m \leq K} \nu_{p_1}(m) \log \frac{m}{m-1} + O\left( \frac{(\log\log N)^{O(1)}}{\log^2 N} \right).
\end{align*}
It thus suffices to establish the inequality
\begin{equation}\label{key-ineq}
   \sum_{m \leq K} \nu_{p_1}(m) \log \frac{m}{m-1} \leq \sum_{m \leq K}^* \nu_{p_1}(m) \frac{3}{m}
\end{equation}
for all $p_1 > 3$.

Writing $\nu_{p_1}(m) = \sum_{j \geq 1} 1_{p_1^j|m}$, it suffices to show that
  $$ \sum_{m \leq K; p^j|m} \frac{3}{m} 1_{(m,6)=1} - \log \frac{m}{m-1} \geq 0.$$
Making the change of variables $m = p_1^j n$, it suffices to show that
  $$ \sum_{n \leq K'} \frac{3}{n} 1_{(n,6)=1} - p_1^j \log \frac{p_1^j n}{p_1^j n - 1} \geq 0$$
  for any $K' > 0$.   Using the bound
  $$ \log \frac{p_1^jn}{p_1^jn - 1} = \int_{p_1^jn-1}^{p_1^jn} \frac{dx}{x} \leq \frac{1}{p_1^jn-1}$$
  and $p^j \geq 5$, we have
  $$ p_1^j \log \frac{p_1^j n}{p_1^j n - 1} \leq \frac{1}{n-0.2}$$
  and so it suffices to show that
  \begin{equation}\label{kb}
  \sum_{n \leq K'}^* \frac{3}{n} 1_{(n,6)=1} - \frac{1}{n-0.2} \geq 0.
  \end{equation}
  Since
  $$ \sum_{n=1}^\infty \frac{1}{n-0.2}-\frac{1}{n} = \psi(0.8)-\psi(1) = 0.353473\dots,$$
  where $\psi$ here denotes the digamma function rather than the von Mangoldt summatory function, it will suffice to show that
  \begin{equation}\label{kb-2}
  \sum_{n \leq K'} \frac{3}{n} 1_{(n,6)=1} - \frac{1}{n} \geq 0.4.\end{equation}
  This can be numerically verified for $K' \leq 100$, with substantial room to spare for $K'$ large; see \Cref{fig:kb}. On a block $6a-1 \leq n \leq 6a+4$ with $a>1$, the sum is positive:
  \begin{align*}
  \sum_{6a-1 \leq n \leq 6a+4}^* \frac{3}{n}  - \frac{1}{n } &= \left(\frac{1}{6a-1} - \frac{1}{6a}\right) + \left(\frac{1}{6a-1} - \frac{1}{6a+2}\right)\\
  &\quad + \left(\frac{1}{6a+1} - \frac{1}{6a+3}\right) + \left(\frac{1}{6a+1} - \frac{1}{6a+4}\right) \\
  &\quad > 0.
  \end{align*}
The inequality for $K'>100$ is then easily verified from the $K' \leq 100$ data and the triangle inequality.

  \begin{figure}
  \centering
  \includegraphics[width=0.8\textwidth]{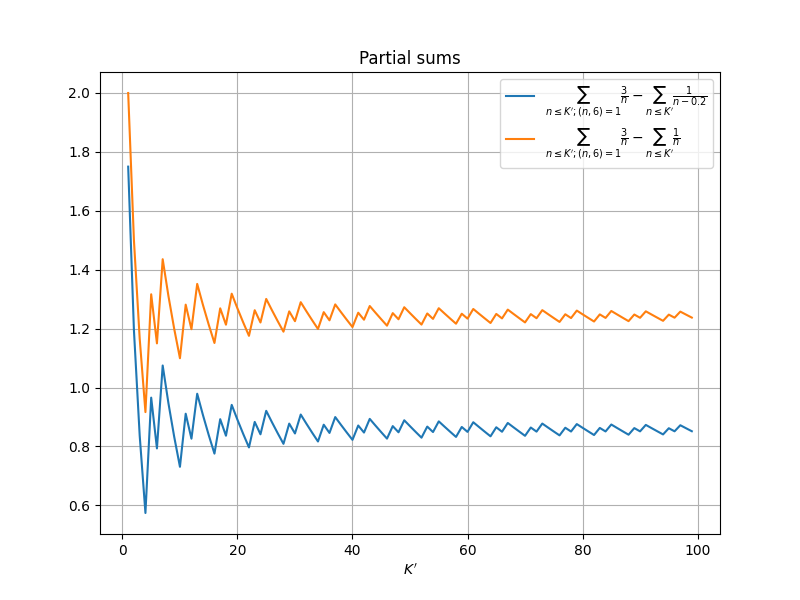}
  \vspace{-8pt}
  \caption{A plot of \eqref{kb}, \eqref{kb-2}.}
  \label{fig:kb}
  \end{figure}

\section{Guy--Selfridge conjecture} \label{10^{11}}

We now establish the Guy--Selfridge conjecture $t(N) \geq N/3$ in the range
$$ N \geq N_0 \coloneqq 10^{11}.$$
We will apply \Cref{repair} with the construction in \Cref{construction-sec} and the choice of parameters
\begin{align*}
  t &\coloneqq N/3\\
  A &\coloneqq 189\\
  K &\coloneqq 293 \\
  L &\coloneqq 4.5;
\end{align*}
the choice of $A$ and $K$ was obtained after some numerical experimentation.  In particular, by \eqref{sigma-def} we have
$$ \sigma = \frac{3N}{At}=\frac{9}{189} =0.047619\dots$$
One can readily check the required conditions \eqref{conditions}, \eqref{t-lower} for $N \geq N_0$, so it remains to verify the hypotheses \eqref{delta-cond}, \eqref{alpha-cond} of \Cref{repair} in this range.  Some of the quantities in these hypotheses involve sums over large ranges, such as $(t/K,N]$; but one can use \Cref{osc-lemma} to obtain adequate upper or lower bounds on such quantities, leaving one with sums over short ranges such as $p \leq K$ or $p \leq K(1+\sigma)$.  As such, all of the bounds needed can be quickly computed even for very large $N$ with simple computer code\footnote{\url{https://github.com/teorth/erdos-guy-selfridge/blob/main/src/python/interval/interval_computations.py}}.

Many of the bounds we will use will be monotone decreasing in $N$, so that they only need to be tested at the left endpoint $N=N_0$.  However, this is not the case for all of the bounds, as some involve subtracting one monotone quantity from another.  For those estimates, we will initially only establish bounds in two extreme cases, $N=N_0$ and $N \geq 10^{70}$, and discuss how to cover the intervening ranges $N_0 < N < 10^{70}$ at the end of the section.

We now bound some of the terms appearing in \Cref{repair}. From \Cref{power-lemma} we have
\begin{equation}\label{kappa-45}
  \kappa_{4.5} = \log \frac{4}{3} =0.28768\dots.
\end{equation}

From \eqref{gammas-def} one can take
$$ \gamma_2 \coloneqq \frac{\log 2}{\log \sqrt{3}} \frac{\log(2L) + \kappa_L}{\log(N_0/3) - \log(2L)} = 0.1423165\dots$$
and
$$\gamma_3 \coloneqq \frac{\log \sqrt{3}}{\log 2} \frac{\log(3L) + \kappa_L}{\log(N_0/3) - \log(3L)} =  0.1059116 \dots$$
for all $N \geq N_0$; by \eqref{kappastar-def} and some calculation we then have
$$ \kappa_{**} \leq 6.830101\dots.$$

From \eqref{stirling} one has
$$ \delta \geq \log N - \log t = \log \frac{3}{e} = 0.0986122\dots$$
for all $N \geq N_0$.

We will use this lower bound as our unit of reference for all other $\delta_i$ quantities, bounding them by suitable multiples of $\delta$.  For instance, from \eqref{delta1-bound} one has
$$ \delta_1 \leq \frac{9}{2A} + \frac{4}{N_0} \leq 0.241447 \delta$$
for all $N \geq N_0$.

From \eqref{delta2-def} and \Cref{osc-lemma}, and the monotonicity of $E(N)/N$, one has
\begin{align*}
  \delta_2 &\leq \frac{\int_{1/3K}^1 f_3(x)\ dx}{\log(t/K)}  + \frac{\|f_3\|_{\mathrm{TV}((1/3K,1])}}{\log(t/K)}  \frac{E(N)}{N} \\
  &\leq \frac{0.919785}{\log(N_0/3K)}  + \frac{1159.795}{\log(N_0/3K)}  \frac{E(N_0)}{N_0} \\
  &\leq 0.504735 \delta
\end{align*}
for all $N \geq N_0$.  Despite the seemingly large numerator, the second term here is in fact negligible for $N \geq N_0$, due to the square root type decay in $E(N)/N$. For $N \geq 10^{70}$ we may replace $N_0$ by $10^{70}$ and obtain the significantly better bound
$$ \delta_2 \leq 0.060410 \delta.$$

From \Cref{delta3-alpha3-bound} and \eqref{pi-upper} one has
\begin{align*}
\delta_3 &\leq \frac{(4A+3)\kappa_{4.5}}{3} \left(\!
\frac{1}{3K \log \frac{N}{3K}} + \frac{1.2762}{3K \log^2\! \frac{N}{3K}}
+ \frac{\log N}{\sqrt{N} \log\! \sqrt{N}\log 5 } + \frac{1.2762 \log N}{\sqrt{N} \log^2\!\! \sqrt{N}\log 5 } \!\right) \\
&\leq \frac{(4A+3)\kappa_{4.5}}{3} \left(\!
  \frac{1}{3K \log\frac{N_0}{3K}} + \frac{1.2762}{3K \log^2\!\frac{N_0}{3K}}
  + \frac{\log N_0}{\sqrt{N_0} \log\! \sqrt{N_0}\log 5 } + \frac{1.2762 \log N_0}{\sqrt{N_0} \log^2\!\! \sqrt{N_0}\log 5 }\! \right) \\
&\leq 0.051574 \delta
\end{align*}
for all $N \geq N_0$.

We skip $\delta_4, \delta_5, \delta_7$ for now.  From \eqref{delta6-def} we have
$$ \delta_6 \leq \frac{\kappa_{4.5}}{N_0} \leq 3 \times 10^{-11} \delta$$
for all $N \geq N_0$, and from \eqref{delta8-def} we have
$$ \delta_8 \leq \frac{2(\log(N_0/3) + \kappa_{4.5})}{N_0} \leq 6 \times 10^{-10} \delta$$
for all $N \geq N_0$, so these two terms are negligible in the analysis.

From \eqref{alpha1-vanish} we have
$$ \alpha_1 = 0.$$
We skip $\alpha_2, \alpha_4, \alpha_5$ for now.  From \Cref{delta3-alpha3-bound} and \eqref{pi-upper} one has
\begin{align*}
  \alpha_3 &\leq \frac{4A+3}{3\log \sqrt{12}} \left(\log \frac{N}{3K} + \kappa_{**}\right)  \\
&\quad \times \left(
  \frac{1}{3K \log\frac{N}{3K}} + \frac{1.2762}{3K \log^2\frac{N}{3K}}
  + \frac{\log N}{\sqrt{N} \log \sqrt{N}\log 5 } + \frac{1.2762 \log N}{\sqrt{N} \log^2 \sqrt{N}\log 5 } \right).
\end{align*}
Expanding out the product, one can check that all terms are non-increasing in $N$; so we may substitute $N_0$ for $N$ in the right-hand side, which after some calculation gives
$$ \alpha_3 \leq 0.361121$$
for all $N \geq N_0$.
From \eqref{alpha6-def} we have
\begin{align*}
   \alpha_6 &\leq \frac{\log(N_0/3) + \kappa_{**}}{N_0 \log \sqrt{12}} \leq 3 \times 10^{-10}
\end{align*}
for all $N \geq N_0$, and similarly from \eqref{alpha7-def} we have
\begin{align*}
\alpha_7 &\leq \max\left( \frac{\log(2N_0)}{(1-\gamma_2)N_0\log 2},  \frac{\log(3N_0)}{(1-\gamma_3)N_0\log \sqrt{3}}\right) \leq 6 \times 10^{-10}
\end{align*}
for all $N \geq N_0$.
so the contribution of these two terms are negligible.

Conveniently, the choice of parameters $A,K$ ensures that there are no primes in the range
$$ 293 = K < p \leq K(1+\sigma) = K(1+\sigma)=306.952\dots$$
and thus
$$ \delta_4=\alpha_4 = 0$$
for all $N \geq N_0$. (Even if this were not the case, the quantities $\delta_4,\alpha_4$ should be viewed as lower order terms, and are far smaller than several of the other $\delta_i$ or $\alpha_i$ for typical choices of parameters.)

The remaining terms $\delta_5, \delta_7, \alpha_2, \alpha_5$ to estimate involve the quantities $A_{p_1}, B_{p_1}$ defined in \eqref{A-def}, \eqref{B-def}, and require a bit more care.  For $B_{p_1}$, we can split the expression as
$$ B_{p_1} = \sum_{m \leq K} \nu_{p_1}(m) \sum_{k: a_{k,m} < b_{k,m}} \frac{k}{N} (\pi( N b_{k,m} ) - \pi(N a_{k,m}))$$
where
$$ a_{k,m} \coloneqq \max\left( \frac{1}{3m}-, \frac{1}{k} \right); \quad b_{k,m} \coloneqq \max\left( \frac{1}{3(m-1)}-, \frac{1}{k-1} \right)$$
where the $-$ denotes the subtraction of an infinitesimal quantity to reflect the restriction to the range $\frac{t}{m} \leq p < \frac{t}{m-1}$ rather than
$\frac{t}{m} < p \leq \frac{t}{m-1}$.  Using \Cref{osc-lemma} (and a limiting argument), we can upper bound this quantity by
$$ B_{p_1} \leq \sum_{m \leq K} \nu_{p_1}(m) \sum_{k: a_{k,m} < b_{k,m}} \left(\frac{k}{2\log(N a_{k,m})}+\frac{k}{2\log(N b_{k,m})}\right) (a_{k,m}-b_{k,m}) + 2\frac{E(N b_{k,m})}{N \log(N a_{k,m})} $$
and lower bound it by
\begin{align*}
   B_{p_1} &\geq \sum_{m \leq K} \nu_{p_1}(m) \sum_{k: a_{k,m} < b_{k,m}} \frac{k\left(1-\frac{2}{\sqrt{a_{k,m} N}}\right)}{\log(N (a_{k,m}+b_{k,m})/2)}  (a_{k,m}-b_{k,m}) - 2\frac{E(N b_{k,m})}{N \log(N a_{k,m})}
\end{align*}
We caution here that while the upper bound for $B_{p_1}$ is monotone decreasing in $N$, the lower bound does not have a favorable monotonicity property, particularly as it will be used when \emph{subtracting} copies of $B_{p_1}$ rather than \emph{adding} them.

From the monotonicity of the upper bound, one can use \eqref{alpha2-def} to calculate that
$$ \alpha_2 \leq 0.269878$$
for all $N \geq N_0$.  For \eqref{delta7-def}, subtraction is involved, and one must proceed with more caution.  For $N=N_0$, one has
$$ \delta_7 \leq 0.11359 \delta.$$
For $N \geq 10^{70}$, we simply discard the negative terms here and obtain the bound
$$ \delta_7 \leq \frac{\kappa_{4.5} \log \sqrt{12}}{\log(10^{70}/3)}  \leq 0.02212 \delta.$$

As for the $A_{p_1}$, we know from \eqref{ap1-vanish} that this vanishes unless $3 < p_1 \leq K(1+\sigma)$.  From \eqref{ap1-small} and \Cref{osc-lemma} one has the upper bound
\begin{align*}
A_{p_1} &\leq \sum_{m \leq K(1+\sigma)}^* \left(\frac{A \nu_{p_1}(m)}{2\log(N/3\min(m,K))} + \frac{A \nu_{p_1}(m)}{2\log(N(1+\sigma)/3m)} \right) \left( \frac{1+\sigma}{3m} - \frac{1}{3\min(m,K)}\right) \\
& \quad \quad + \frac{A \nu_{p_1}(m)}{\log(N/3\min(m,K))} \frac{2E(N(1+\sigma)/3m)}{N}
\end{align*}
and the lower bound
\begin{align*}
A_{p_1} &\geq \sum_{m \leq K(1+\sigma)}^* \frac{A \nu_{p_1}(m)\left(1-\frac{2}{\sqrt{N_0/3\min(m,K)}}\right) }{\log((N/3\min(m,K) + N(1+\sigma)/3m)/2)} \left(\frac{1+\sigma}{3m} - \frac{1}{3\min(m,K)}\right) \\
& \quad \quad - \frac{A \nu_{p_1}(m)}{\log(N/3\min(m,K))}  \frac{2E(N(1+\sigma)/3m)}{N}.
\end{align*}
Again, the upper bound is monotone decreasing in $N$, but the lower bound does not have a favorable monotonicity. At $N=N_0$, one can calculate using these bounds and \eqref{delta5-def}, \eqref{alpha5-def} to obtain
\begin{align*}
\delta_5 &\leq 0.06203 \delta; \qquad \alpha_5 \leq 0.31418
\end{align*}
which, when combined with the previous bounds, gives
$$ \sum_{i=1}^8 \delta_i  \leq 0.9740 \delta; \qquad \sum_{i=1}^7 \alpha_i \leq 0.9452$$
at $N=N_0$, thus verifying \eqref{delta-cond}, \eqref{alpha-cond} in those cases.

For $N \geq 10^{70}$, we use the triangle inequality to crudely upper bound
$$ \delta_5 \leq \kappa_{4.5} \sum_{3 < p_1 \leq K} \frac{\log p_1}{\log(t/K^2)} A_{p_1} + B_{p_1}$$
and
$$ \alpha_5 \leq \frac{1}{\log \sqrt{12}} \sum_{3 < p_1 \leq K}  \frac{(\log K^2+\kappa_{**})\log p_1}{\log(t/K^2)} A_{p_1} + (\log p_1 + \kappa_{**}) B_{p_1}.$$
The bounds available for the right-hand side are now monotone in $N$, and one can calculate that
\begin{align*}
  \delta_5 &\leq 0.077301 \delta; \qquad \alpha_5 \leq 0.184975
  \end{align*}
for $N \geq 10^{70}$. This is better than the previous bound for $\alpha_5$.  For $\delta_5$, the bound is slightly worse, but this is more than compensated for by the improved bounds on $\delta_2$, $\delta_7$, and \eqref{delta-cond}, \eqref{alpha-cond} can be verified here with significant room to spare.

This completes the proof of \Cref{main}(iii) (and hence \Cref{main}(ii)) in the cases $N=N_0$ and $N \geq 10^{70}$.  It remains to cover the intermediate range $N_0 < N \leq 10^{70}$.  Here we adopt the perspective of interval arithmetic.  If $N$ is restricted to a given interval, such as $[10^{11}, 5 \times 10^{11}]$, we can use the worst-case upper and lower bounds for $A_{p_1}, B_{p_1}$ to obtain conservative upper bounds on the most delicate quantities
$\delta_5, \delta_7, \alpha_2, \alpha_5$, thus potentially verifying the conditions \eqref{delta-cond}, \eqref{alpha-cond} simultaneously for all $N$ in such an interval.  As it turns out, there is enough room to spare in these estimates, particularly for large $N$, that this strategy works using only a small number of intervals; specifically, by considering $N$ in the intervals
$$ [10^{11}, 5 \times 10^{11}]; \quad [5 \times 10^{11}, 10^{14}]; \quad [10^{14}, 10^{20}]; \quad [10^{20}, 10^{70}]$$
one can check that such bounds are sufficient to verify \eqref{delta-cond}, \eqref{alpha-cond} in these cases.  This now verifies \Cref{main}(ii), (iii) for all $N \geq 10^{11}$.
(In fact, with more effort, this verification can be pushed down to $N \geq 6 \times 10^{10}$ using the same choice of parameters $A,K,L$, specifically by checking $$ [6 \times 10^{10}, 6.05 \times 10^{10}]; \ \  [6.05 \times 10^{10}, 6.1 \times 10^{10}]; \ \  [6.1 \times 10^{10}, 6.5 \times 10^{10}]; \ \  [6.5 \times 10^{10}, 7 \times 10^{10}]; \ \  [7 \times 10^{10}, 10^{11}].$$)

\appendix

\section{Distance to the next \texorpdfstring{$3$}{3}-smooth number}\label{power-sec}

We now establish the various claims in \Cref{power-lemma}.  We begin with part (iii).  The claim \eqref{mod-kappa} is immediate from \eqref{kappa-def}, \eqref{fancy-kappa-def}.  Now prove \eqref{12-2}, \eqref{12-3}.  If we let $\lceil x/12^a \rceil^{\langle 2,3 \rangle} = 2^b 3^c$, then by \eqref{kappa-def} we have
$$ b \log 2 + c \log 3 \leq \log x - a \log 12 + \kappa_L,$$
while from definition of $a$ we have
\begin{equation}\label{xa12}
  \log x - a \log 12 \leq \log(12L).
\end{equation}
We now compute
\begin{align*}
  \frac{\nu_2(\lceil x \rceil^{\langle 2,3\rangle}_L) - 2 \gamma \nu_3(\lceil x \rceil^{\langle 2,3\rangle}_L)}{1-\gamma}
  &= \frac{2a+b - 2\gamma(a+c)}{1-\gamma} \\
  &\leq 2a + \frac{\log x - a \log 12 + \kappa_L}{(1-\gamma) \log 2}  \\
  &= \frac{\log x}{\log \sqrt{12}} + \left( \frac{1}{(1-\gamma)\log 2} - \frac{1}{\log \sqrt{12}}\right) \left(\log x - a \log 12\right)\\
  &\quad
  + \frac{\kappa_L}{(1-\gamma)\log 2}
\end{align*}
giving \eqref{12-2} from \eqref{xa12}; similarly, we have
\begin{align*}
  \frac{2\nu_3(\lceil x \rceil^{\langle 2,3\rangle}_L) - \gamma \nu_2(\lceil x \rceil^{\langle 2,3\rangle}_L)}{1-\gamma}
  &= \frac{2(a+c) - \gamma(2a+b)}{1-\gamma} \\
  &\leq 2a + \frac{2(\log x - a \log 12 + \kappa_L)}{(1-\gamma) \log 3}  \\
  &= \frac{\log x}{\log \sqrt{12}} + \left( \frac{2}{(1-\gamma)\log 3} - \frac{1}{\log \sqrt{12}}\right) \left(\log x - a \log 12\right)\\
  &\quad + \frac{\kappa_L}{(1-\gamma)\log \sqrt{3}}
\end{align*}
giving \eqref{12-3} from \eqref{xa12}.

To prove parts (i) and (ii) of \Cref{power-lemma}, we establish the following lemma to upper bound $\kappa_L$.

\begin{lemma}\label{lemcount-0}  If $n_1,n_2,m_1,m_2$ are natural numbers such that $n_1+n_2, m_1+m_2 \geq 1$ and
$$ \frac{3^{m_1}}{2^{n_1}}, \frac{2^{n_2}}{3^{m_2}} \geq 1$$
then
$$ \kappa_{\min( 2^{n_1+n_2},3^{m_1+m_2})/6} \leq \log \max\left(\frac{3^{m_1}}{2^{n_1}}, \frac{2^{n_2}}{3^{m_2}}\right).$$
\end{lemma}

\begin{proof}  If $\min( 2^{n_1+n_2},3^{m_1+m_2})/6 \leq t \leq 2^{n_2-1} 3^{m_1-1}$, then we have
\begin{equation}\label{tb}
  t \leq 2^{n_2-1} 3^{m_1-1} \leq \max\left(\frac{3^{m_1}}{2^{n_1}}, \frac{2^{n_2}}{3^{m_2}}\right) t,
\end{equation}
so we are done in this case.  Now suppose that $t > 2^{n_2-1} 3^{m_1-1}$.
If we write $\lceil t \rceil^{\langle 2,3 \rangle} =2^n 3^m$ be the smallest $3$-smooth number that is at least $t$, then we must have $n \geq n_2$ or $m \geq m_1$ (or both).  Thus at least one of $\frac{2^{n_1}}{3^{m_1}} 2^n 3^m$ and $\frac{3^{m_2}}{3^{n_2}} 2^n 3^m$ is an integer, and is thus at most $t$ by construction.  This gives \eqref{tb}, and the claim follows.
\end{proof}

Some efficient choices of parameters for this lemma are given in \Cref{approx-table}.  For instance, $\kappa_{4.5} \leq \log\frac{4}{3} = 0.28768\dots$ and $\kappa_{40.5} \leq \log \frac{32}{27} = 0.16989\dots$.  In fact, since $\lceil 4.5+\eps \rceil^{\langle 2,3\rangle} = 6$ and $\lceil 40.5+\eps \rceil^{\langle 2,3\rangle} = 48$ for all sufficiently small $\eps>0$, we see that these bounds are sharp (and similarly for the other entries in \Cref{approx-table}); this establishes part (i).

\begin{table}[ht]
\centering
\begin{tabular}{|c|c|c|c|c|c|}
\hline
\rule{0pt}{12pt}$n_1$ & $m_1$ & $n_2$ & $m_2$ & $\min(2^{n_1+n_2},3^{m_1+m_2})/6$ & $\log \max(3^{m_1}/2^{n_1}, 2^{n_2}/3^{m_2})$ \\[2pt]
\hline
\rule{0pt}{12pt}$1$ & $1$ & $\mathbf{1}$ & $\mathbf{0}$ & $1/2 = 0.5$ & $\log 2 = 0.69314\dots$ \\
\hline
\rule{0pt}{12pt}$\mathbf{1}$ & $\mathbf{1}$ & $2$ & $1$ & $2^2/3 = 1.33\dots$ & $\log (3/2) = 0.40546\dots$\\
\hline
\rule{0pt}{12pt}$3$ & $2$ & $\mathbf{2}$ & $\mathbf{1}$ & $3^2/2 = 4.5$ & $\log (2^2/3) = 0.28768\dots$ \\
$3$ & $2$ & $\mathbf{5}$ & $\mathbf{3}$ & $3^4/2 = 40.5$ & $\log (2^5/3^3) = 0.16989\dots$ \\
\hline
\rule{0pt}{12pt}$\mathbf{3}$ & $\mathbf{2}$ & $8$ & $5$ & $2^{10}/3 = 341.33\dots$ & $\log (3^2/2^3) = 0.11778\dots$\\
$\mathbf{11}$ & $\mathbf{7}$ & $8$ & $5$ & $2^{18}/3 = 87381.33\dots$ & $\log (3^7/2^{11}) = 0.06566\dots$ \\
\hline
\rule{0pt}{12pt}$19$ & $12$ & $\mathbf{8}$ & $\mathbf{5}$ & $3^{17}/2 \approx 6.4 \times 10^7$ & $\log (2^8/3^5) = 0.05211\dots$ \\
$19$ & $12$ & $\mathbf{27}$ & $\mathbf{17}$ & $3^{29}/2 \approx 3.4 \times 10^{13}$ & $\log (2^{27}/3^{17}) = 0.03856\dots$ \\
$19$ & $12$ & $\mathbf{46}$ & $\mathbf{29}$ & $3^{41}/2 \approx 1.8 \times 10^{19} $ & $\log (2^{46}/3^{29}) = 0.02501\dots$ \\
\hline
\end{tabular}
\bigskip

\caption{Efficient parameter choices for \Cref{lemcount-0}.  The parameters used to attain the minimum or maximum are indicated in \textbf{boldface}. Note how the number of rows in each group matches the terms $1,1,2,2,3,\dots$ in the continued fraction expansion.}\label{approx-table}
\end{table}

\begin{remark}
It should be unsurprising that the continued fraction convergents $1/1$, $2/1$, $3/2$, $8/5$, $19/12$, $\dots$ to
$$\frac{\log 3}{\log 2} = 1.5849\dots = [1; 1,1,2,2,3,1,\dots]$$
are often excellent choices for $n_1/m_1$ or $n_2/m_2$, although other approximants such as $5/3$ or $11/7$ are also usable.
\end{remark}

Finally, we establish (ii).  From the classical theory of continued fractions, we can find rational approximants
\begin{equation}\label{abn}
 \frac{p_{2j}}{q_{2j}} \leq \frac{\log 3}{\log 2} \leq \frac{p_{2j+1}}{q_{2j+1}}
\end{equation}
to the irrational number $\log 3/\log 2$, where the convergents $p_j/q_j$ obey the recursions
$$ p_j = b_j p_{j-1} + p_{j-2}; \quad q_j = b_j q_{j-1} + q_{j-2}$$
with $p_{-1} = 1, q={-1}=0, p_0 = b_0, q_0=1$, and
$$[b_0;b_1,b_2,\dots] = [1;1,1,2,2,3,1\dots]$$
is the continued fraction expansion of $\frac{\log 3}{\log 2}$.  Furthermore, $p_{2j+1}q_{2j} - p_{2j} q_{2j+1} = 1$, and hence
\begin{equation}\label{abn-2}
  \frac{\log 3}{\log 2} - \frac{p_{2j}}{q_{2j}} = \frac{1}{q_{2j} q_{2j+1}}.
\end{equation}
By Baker's theorem (see, e.g., \cite{baker-wustholz}), $\frac{\log 3}{\log 2}$ is a Diophantine number, giving a bound of the form
\begin{equation}\label{bj1}
   q_{2j+1} \ll q_{2j}^{O(1)}
\end{equation}
and a similar argument (using $p_{2j+2} q_{2j+1}-p_{2j+1} q_{2j+2} = -1$) gives
\begin{equation}\label{bj2}
 q_{2j+2} \ll q_{2j+1}^{O(1)}.
\end{equation}
We can rewrite \eqref{abn} as
$$ \frac{3^{q_{2j}}}{2^{p_{2j}}}, \frac{2^{p_{2j+1}}}{3^{q_{2j+1}}}\geq 1$$
and routine Taylor expansion using \eqref{abn-2} gives the upper bounds
$$ \frac{3^{q_{2j}}}{2^{p_{2j}}}, \frac{2^{p_{2j+1}}}{3^{q_{2j+1}}}\leq \exp\left( O\left( \frac{1}{q_{2j}}\right)\right).$$
From \Cref{lemcount-0} we obtain
$$
\kappa_{\min(2^{p_{2j} + p_{2j+1}}, 3^{q_{2j}+q_{2j+1}})/6} \ll \frac{1}{q_{2j}}.$$
The claim then follows from \eqref{bj1}, \eqref{bj2} (and the fact that $\kappa$ is monotone non-increasing after optimizing in $j$).

\begin{remark}
It seems reasonable to conjecture that $c$ can be taken to be arbitrarily close to $1$, but this is essentially equivalent to the open problem of determining that the irrationality measure of $\log 3 / \log 2$ is $2$.
\end{remark}

\section{Estimating sums over primes}\label{primes-sec}

In this appendix we establish \Cref{osc-lemma}.  The key tool is

\begin{lemma}[Integration by parts]\label{integ-lemma}  Let $(y,x]$ be a half-open interval in $(0,+\infty)$.  Suppose that one has a function $a \colon \N \to \R$ and a continuous function $f: (y,x] \to \R$ such that
  $$ \sum_{y < n \leq z} a_n = \int_z^y f(t)\ dt + C + O_{\leq}(A)$$
  for all $y \leq z \leq x$, and some $C \in \R$, $A>0$.  Then, for any function $b: (y,x] \to \R$ of bounded total variation, one has
\begin{equation}\label{tve}
   \sum_{y < n \leq x} b(n) a_n = \int_x^y b(t) f(t)\ dt + O_{\leq}(A\|b\|_{\mathrm{TV}^*(y,x]}).
\end{equation}
\end{lemma}

\begin{proof}  If, for every natural number $y < n \leq x$, one modifies $b$ to be equal to the constant $b(n)$ in a small neighborhood of $n$, then one does not affect the left-hand side of \eqref{tve} or increase the total variation of $b$, while only modifying the integral in \eqref{tve} by an arbitrarily small amount.  Hence, by the usual limiting argument, we may assume without loss of generality that $b$ is locally constant at each such $n$.  If we define the function $g \colon (y,x] \to \R$ by
$$ g(z) \coloneqq  \sum_{y < n \leq z} a_n - \int_z^y f(u)\ du - C$$
then $g$ has jump discontinuities at the natural numbers, but is otherwise continuously differentiable, and is also bounded uniformly in magnitude by $A$.  We can then compute the Riemann--Stieltjes integral
$$ \int_{(y,x]} b\ dg = \sum_{y < n \leq x} b(n) a_n - \int_y^x f(t) b(t)\ dt.$$
Since the discontinuities of $g$ and $b$ do not coincide, we may integrate by parts to obtain
$$ \int_{(y,x]} b\ dg = b(x) g(x) - b(y^+) g(y^+) - \int_{(y,x]} g\ db.$$
The left-hand side is $O_{\leq}(A \|b\|_{\mathrm{TV}^*(y,x]})$, and the claim follows.
\end{proof}

We now prove \eqref{bv-exact}. In fact we prove the sharper estimate
\begin{equation}\label{bsharp}
  \sum_{y < p \leq x} b(p) \log p = \int_y^x b(t) \left(1 - \frac{2}{\sqrt{t}}\right) \ dt
+ O_\leq\left(\|b\|_{\mathrm{TV}^*((y,x])} \tilde E(x) \right)
\end{equation}
where
\begin{equation}\label{etil-def}
\tilde E(x) \coloneqq 0.95 \sqrt{x} + \min( \max(\eps_0,\eps_1(x)), \eps_2(x), \eps_3(x)) 1_{x \geq 10^{19}}
\end{equation}
and
\begin{align*}
  \eps_0(x) &\coloneqq \frac{\sqrt{x}}{8\pi} \log x(\log x - 3),\\
  \eps_1(x) &\coloneqq 1.12494 \times 10^{-10},\\
  \eps_2(x) &\coloneqq 9.39 (\log^{1.515} x) \exp(-0.8274\sqrt{\log x}),\text{ and}\\
  \eps_3(x) &\coloneqq 0.026 (\log^{1.801} x) \exp(-0.1853 (\log^{3/5} x) (\log\log x)^{-1/5}).
\end{align*}
From using the $\eps_2$ term, it is clear that
$$ \tilde E(x) \ll x \exp(-c \sqrt{\log x})$$
for some absolute constant $c>0$; and by using the $\eps_0$ and $\eps_1$ terms and routine calculations one can show that
$$ \tilde E(x) \leq E(x)$$
for all $x \geq 1423$.

Observe that $\tilde E$ is monotone non-decreasing. Thus by \Cref{integ-lemma}, to show \eqref{bsharp}, it will suffice to show that
$$ \sum_{p \leq x} \log p = x - \sqrt{x} + O_{\leq}(\tilde E(x))
  = \int_0^x \left(1-\frac{2}{\sqrt{t}}\right)\ dt + O_{\leq}(\tilde E(x))$$
for all $x \geq 1423$.

For $1423 \leq x \leq 10^{19}$, this claim follows from \cite[Theorem 2]{buthe-2}.  For $x > 10^{19}$, we apply \cite[(6.10), (6.11)]{buthe} to conclude that
$$
\sum_{p \leq x} \log p = \psi(x) - \psi(\sqrt{x}) + O_{\leq}(1.03883 (x^{1/3} + x^{1/5} + 2 (\log x) x^{1/13})),$$
where $\psi(x) \coloneqq \sum_{n \leq x}
\Lambda(n)$ is the usual von Mangoldt summatory function.
From \cite[Theorems 10,12]{rs} we have
$$ \psi(\sqrt{x}) = \sqrt{x} + O_{\leq}(0.18 \sqrt{x}).$$
Since
$$ 0.18 \sqrt{x} + 1.03883 (x^{1/3} + x^{1/5} + 2 (\log x) x^{1/13}) \leq 0.95 \sqrt{x}$$
in this range of $x$, it suffices to show that
$$ \psi(x) = x + O_{\leq}(\min( \max(\eps_0(x),\eps_1(x)), \eps_2(x), \eps_3(x)) )$$
for $x > 10^{19}$.  The claims for $i=2,3$ follow from \cite[Theorems 1.1, 1.4]{johnston-yang}.  In \cite[Theorem 2, (7.3)]{buthe}, the bound
$$ \psi(x) = x + O_{\leq}(\eps_0(x))$$
is established whenever $x \geq 5000$ and $4.92 \frac{x}{\sqrt{\log x}} \leq T$, where $T$ is a height up to which the Riemann hypothesis has been established.  Using the value $T = 3 \times 10^{12}$ from \cite{platt-rh}, we can therefore cover the range $10^{19} < x < e^{55}$ (in fact we could go up to $e^{58.33} \approx 2.1 \times 10^{25}$).  For $x \geq e^{55}$, we can use \cite[Table 2]{buthe} (the value $T = 2.445 \times 10^{12}$ used there following from \cite{platt-rh}).

\begin{remark} Assuming the Riemann hypothesis, the $\eps_1, \eps_2, \eps_3$ terms in the definition of $\tilde E(x)$ may be deleted, since \cite[(7.3)]{buthe} then holds for all $x \geq 5000$.
\end{remark}

The claim \eqref{pix} now follows from \eqref{bv-exact} by setting $b(t) \coloneqq \frac{1}{\log t}$.  The bounds \eqref{pixy-upper}, \eqref{pixy-lower} then follow by estimating
$$ 1-\frac{2}{\sqrt{y}} \leq 1-\frac{2}{\sqrt{t}} \leq 1$$
and using the convexity of $t \mapsto \frac{1}{\log t}$.  Finally, for non-negative $b$, the bounds \eqref{bv-upper}, \eqref{bv-lower} follow from the trivial inequalities
$$ \frac{b(p) \log p}{\log x}  \leq b(p) \leq \frac{b(p) \log p}{\log y}.$$

\section{Computation of \texorpdfstring{$c_0$}{c\_0} and related quantities}\label{c0-app}

In this appendix we give some details regarding the numerical estimation of the constants $c_0, c'_1, c''_1, c_1$ defined in \eqref{c0-def}, \eqref{c1p-def}, \eqref{c1pp-def}, \eqref{c1-def}.

We begin with $c_0$.  As one might imagine from an inspection of \Cref{fig-mean}, direct application of numerical quadrature converges quite slowly due to the oscillatory singularity.  To resolve the singularity, we can perform a change of variables $x=1/y$ to express $c_0$ as an improper integral:
\begin{equation}\label{c0-alt}
   c_0 = \frac{1}{e} \int_1^\infty \lfloor y \rfloor \log \frac{\lceil y/e \rceil}{y/e}\ \frac{dy}{y^2}.
\end{equation}
Next, observe\footnote{We thank an anonymous commenter on the blog of one of the authors for this suggestion.} that
\begin{align*}
  \frac{1}{e} \int_e^\infty y \log \frac{\lceil y/e \rceil}{y/e}\ \frac{dy}{y^2}
  &= \sum_{k=1}^\infty \int_{ke}^{(k+1)e} y \log \frac{k+1}{y/e}\ \frac{dy}{y^2} \\
  &= \frac{1}{e} \sum_{k=1}^\infty \int_k^{k+1} \left(\log(k+1)-\log y\right)\ \frac{dy}{y}\\
  &= \frac{1}{2e} \sum_{k=1}^\infty \log^2 \left(1 + \frac{1}{k}\right)\\
  &= 0.1797439053\dots;
\end{align*}
The value here was computed in interval arithmetic by subtracting off the asymptotically similar sum $\frac{1}{2e} \sum_{k=1}^\infty \frac{1}{k^2} = \frac{1}{2e} \frac{\pi^2}{6}$, summing the resulting partial sum up to $k=10^5$, and bounding the tail of the sum rigorously. We have
$$ \frac{1}{e} \int_1^e \lfloor y \rfloor \log \frac{e}{y}\ \frac{dy}{y^2} = \frac{2}{e^2} - \frac{\log 2}{2e} = 0.143173268\dots$$
and hence
$$ c_0 = \frac{1}{2e} \sum_{k=1}^\infty \log^2 \left(1 + \frac{1}{k}\right)
+ \frac{2}{e^2} - \frac{\log 2}{2e} - \frac{1}{e} \int_e^\infty \{y\} \log \frac{\lceil y/e \rceil}{y/e}\ \frac{dy}{y^2}$$
where $\{x\} \coloneqq x - \lfloor x \rfloor$.  The integrand here lies between $0$ and $1/y^3$, so the integral for $y \geq T$ lies between $0$ and $1/2T^2$.  Truncating to say $T = 10^5$ and performing the integral exactly, one can evaluate
$$ \frac{1}{e} \int_e^\infty \{y\} \log \frac{\lceil y/e \rceil}{y/e}\ \frac{dy}{y^2} = 0.018498162\dots$$
so that
$$ c_0 = 0.30441901\dots.$$
A similar calculation (which we omit) reveals that
\begin{align*}
  c'_1 &= \sum_{k=1}^\infty \frac{1+\log(k+1)}{2e}  \log^2\left(1+\frac{1}{k}\right) - \frac{1}{3e} \log^3\left(1+\frac{1}{k}\right)  \\
  &\quad + \frac{6}{e^2} - \frac{\log^2 2 + \log 2 + 3}{2e}\\
 &\quad - \frac{1}{e} \int_e^\infty \{y\} (\log y) \log \frac{\lceil y/e \rceil}{y/e}\ \frac{dy}{y^2} \\
 &\approx 0.3702051\dots.
\end{align*}
Computing the sum $c''_1$ to reasonable accuracy requires some further analysis.  From the crude bound
$$0 \leq \frac{1}{k} \log\left(\frac{e}{k} \left\lceil \frac{k}{e} \right\rceil\right) \leq \frac{e}{k^2}$$
and the integral test, one has the simple tail bound
$$ 0 \leq \sum_{k=K+1}^\infty \frac{1}{k} \log\left(\frac{e}{k} \left\lceil \frac{k}{e} \right\rceil\right) \leq \frac{e}{K}$$
but the convergence rate here is slow.  To accelerate the convergence, we write $\lceil\frac{k}{e} \rceil = \frac{k}{e} + \{-\frac{k}{e}\}$ and use the more precise Taylor approximation
$$ \frac{e \left\{-\frac{k}{e}\right\}}{k^2} - \frac{e^2 \left\{-\frac{k}{e}\right\}^2}{2k^3}
\leq \frac{1}{k} \log\left(\frac{e}{k} \left\lceil \frac{k}{e} \right\rceil\right) \leq \frac{e \left\{-\frac{k}{e}\right\}}{k^2}.$$
Bounding $\{-k/e\}$ by one, we have the tail bound
$$ 0 \leq \sum_{k=K+1}^\infty \frac{e^2 \left\{-\frac{k}{e}\right\}^2}{2k^3} \leq \frac{e^2}{4K^2}$$
so the main task is then to control the simplified tail
$$ \sum_{k=K+1}^\infty \frac{e\left\{-\frac{k}{e}\right\}}{k^2}.$$
From the integral test one has
$$ \frac{e}{2(K+1)} \leq \sum_{k=K+1}^\infty \frac{\frac{e}{2}}{k^2} \leq \frac{e}{2K}$$
so one can instead look at the normalized tail
$$ \sum_{k=K+1}^\infty e\frac{\left\{-\frac{k}{e}\right\}-\frac{1}{2}}{k^2}.$$
The Erd\H{o}s--Tur\'an inequality states that, for any absolutely convergent non-negative weights $c_k$, any interval $I \subset [0,1]$ of length $|I|$, and any real numbers $\xi_k$, and any $N \geq 1$, one has
$$ \left|\sum_k c_k (1_I(\xi_k \mod 1) - |I|)\right| \leq \frac{1}{N+1} \sum_k c_k + \sum_{n=1}^N \left(\frac{2}{\pi n} + \frac{2}{N+1}\right) \left|\sum_k c_k e^{2\pi i n \xi_k}\right|;$$
see the inequality after \cite[Theorem 20]{vaaler}.  (In this reference, only the special case in which $c_k$ is a uniform probability distribution function on $\{1,\dots,M\}$ is discussed, but it is easy to see that the argument in fact works for arbitrary absolutely convergent non-negative weights $c_k$.)  Applying this for $I = [0,h]$ and then averaging in $h$ from $0$ to $1$, we conclude that
$$ \left|\sum_k c_k \left(\{\xi_k\}-\frac{1}{2}\right)\right| \leq \frac{1}{N+1} \sum_k c_k + \sum_{n=1}^N \left(\frac{2}{\pi n} + \frac{2}{N+1}\right) \left|\sum_k c_k e^{2\pi i n \xi_k}\right|.$$
In particular, we have
$$ \left| \sum_{k=K+1}^\infty e\frac{\left\{-\frac{k}{e}\right\}-\frac{1}{2}}{k^2}\right| \leq \frac{1}{N+1} \sum_{k=K+1}^\infty \frac{e}{k^2}
+ \sum_{n=1}^N \left(\frac{2e}{\pi n} + \frac{2e}{N+1}\right) \left|\sum_{k=K+1}^\infty \frac{e^{-2\pi i n k/e}}{k^2}\right|.
$$
To estimate the exponential sum
$$ S_{n,K} \coloneqq \sum_{k=K+1}^\infty \frac{e^{-2\pi i n k/e}}{k^2}$$
observe from shifting $k$ by one that
$$ S_{n,K} = e^{-2\pi i n/e} \sum_{k=K}^\infty \frac{e^{-2\pi i n k/e}}{(k+1)^2}
= e^{-2\pi i n/e} S_{n,K} + O_{\leq}\left( \frac{1}{(K+1)^2} + \sum_{k=K+1}^2 \frac{1}{k^2} -\frac{1}{(k+1)^2}\right)$$
and hence on summing the telescoping series
$$ |S_{n,K}| \leq \frac{2}{|e^{-2\pi i n/e}-1| (K+1)^2} = \frac{1}{(K+1)^2 \sin(\pi n/e)}.$$
Because the irrationality measure of $e$ is $2$, this will give error terms of the shape $O(\log K/K^2)$ if one sets $N \approx K / \log K$.  Setting for instance $K = 10^6$, $N = 10^5$, an interval arithmetic computation then gives
$$c''_1 = 1.679578996\dots$$
and thus by \eqref{c1-def} we have
$$ c_1 = 0.7554808\dots.$$
This concludes our discussion of the numerical estimation of $c_0, c'_1, c''_1, c_1$.

\end{document}